\documentclass{amsart}
\usepackage[backref]{hyperref}
\hypersetup{hidelinks}
\newtheorem{theorem}{Theorem}[section]
\newtheorem{lemma}{Lemma}
\newtheorem{assumption}{Assumption}

\theoremstyle{definition}

\theoremstyle{remark}
\newtheorem{remark}{Remark}

\numberwithin{equation}{section}

\usepackage{cleveref}
\usepackage{graphicx}
\usepackage{epstopdf}
\usepackage{cite}
\usepackage{color,soul}
\usepackage{lipsum}
\usepackage{amsfonts}
\usepackage{algorithmic}
\usepackage{amssymb,amsmath,amscd}
\usepackage{pifont}
\crefname{theorem}{Theorem}{Theorems}
\crefformat{equation}{(#2#1#3)}
\Crefformat{equation}{Eq.~(#2#1#3)}
\crefrangeformat{equation}{(#3#1#4)-(#5#2#6)}
\Crefrangeformat{equation}{Eqs.~(#3#1#4)-(#5#2#6)}
\crefmultiformat{equation}{(#2#1#3)}{ and~(#2#1#3)}{, (#2#1#3)}{ and~(#2#1#3)}
\Crefmultiformat{equation}{Eqs.~(#2#1#3)}{ and~(#2#1#3)}{, (#2#1#3)}{ and~(#2#1#3)}
\crefrangemultiformat{equation}{(#3#1#4)}{-(#5#2#6)}{, (#3#1#4)}{-(#5#2#6)}
\Crefrangemultiformat{equation}{Eqs.~(#3#1#4)}{-(#5#2#6)}{, (#3#1#4)}{-(#5#2#6)}
\crefname{lemma}{Lemma}{Lemmas}
\crefname{figure}{Fig.}{Figs.}
\crefname{example}{Example}{Examples}
\crefname{remark}{Remark}{Remarks}
\crefname{table}{Table}{Tables}
\crefname{assumption}{Assumption}{Assumptions}
\crefname{section}{Section}{Sections}

\begin{document}

\title[Longer time simulation of the unsteady NSEs]{Longer time simulation of the unsteady Navier-Stokes equations based on a modified convective formulation}
\author{Xu Li}
\address{School of Mathematics, Shandong University, Jinan 250100, China}
\email{xulisdu@126.com}
\thanks{The second author is the corresponding author.}

\author{Hongxing Rui}
\address{School of Mathematics, Shandong University, Jinan 250100, China}
\email{hxrui@sdu.edu.cn}
\thanks{This work was supported by the National Natural Science Foundation of China grant 12131014.}





\begin{abstract}
For the discretization of the convective term in the Navier-Stokes equations (NSEs), the commonly used convective formulation (CONV) does not preserve the energy if the divergence constraint is only weakly enforced. In this paper, we apply the skew-symmetrization technique in [B. Cockburn, G. Kanschat and D. Sch\"{o}tzau, Math. Comp., 74 (2005), pp. 1067-1095] to conforming finite element methods, which restores energy conservation for CONV. The crucial idea is to replace the discrete advective velocity with its a $H(\operatorname{div})$-conforming divergence-free approximation in CONV.
We prove that the modified convective formulation also conserves linear momentum, helicity, 2D enstrophy and total vorticity under some appropriate senses. Its a Picard-type linearization form also conserves them.
Under the assumption $\boldsymbol{u}\in L^{2}(0,T;\boldsymbol{W}^{1,\infty}(\Omega)),$ it can be shown that the Gronwall constant does not explicitly depend on the Reynolds number in the error estimates. The long time numerical simulations show that the linearized and modified convective formulation has a similar performance with the EMAC formulation and outperforms the usual skew-symmetric formulation (SKEW).
\end{abstract}

\maketitle

\section{Introduction}
This paper is concerned with the finite element discretizations of the unsteady Navier-Stokes equations
\begin{subequations}\label{NSE}
\begin{align}
\boldsymbol{u}_{t}-\nu\Delta\boldsymbol{u}+(\boldsymbol{u}\cdot\nabla)\boldsymbol{u}+\nabla p&=\boldsymbol{f}&\text{in} ~(0,T]\times\Omega,\label{NSE1}\\
\nabla\cdot\boldsymbol{u}&=0 &\text{in} ~(0,T]\times\Omega,\label{NSE2}\\
\boldsymbol{u}(0)&=\boldsymbol{u}^{0} &\text{in} ~\Omega,\label{NSE3}\\
\boldsymbol{u}&=\boldsymbol{0}&\text{on} ~(0,T]\times\Gamma,\label{NSE4}
\end{align}
\end{subequations}
where $\Omega\subset \mathbb{R}^{d}$ ($d=2,3$) is a bounded domain with Lipschitz and polyhedral boundary $\Gamma$; the velocity $\boldsymbol{u}$ and pressure $p$ are two unknowns; $\nu>0$ is the constant kinematic viscosity; $\boldsymbol{f}=\boldsymbol{f}(\boldsymbol{x})$ represents the external body force at $\boldsymbol{x}\in\Omega$; and $\boldsymbol{u}^{0}$ denotes the initial velocity.

Physically, there are several balance laws implied in the Navier-Stokes equations such as the balances of kinetic energy, linear momentum and angular momentum (EMA) \cite{evans_isogeometric_2013,Rebholz2017}. In some cases (e.g., vanishing external force and viscosity), these quantities are conservative under some appropriate assumptions \cite{Rebholz2017}.
Since at least Arakawa devised an energy and enstrophy conservation method for two-dimensional NSEs in \cite{Arakawa1966}, scientists have recognised that these conservation (balance) laws are important references for unlocking efficient and stable numerical methods for NSEs \cite{Fix1975,Abramov2003,Rebholz2007,Palha2017,evans_isogeometric_2013,Rebholz2020}.
A violation of these laws may bring unexpected instability.

In this paper, we consider inf-sup stable mixed methods for NSEs. The crucial points for these balance laws are the discretization of the nonlinear term and the finite elements one uses. The most common formulation for this term might be the so-called convective formulation (CONV). However, it was shown in \cite{Rebholz2017} that this formulation was not EMA-conserving unless exactly divergence-free elements were used. We note that most classical elements are not divergence-free \cite{john_divergence_2017}. Also in \cite{Rebholz2017}, Charnyi {\it et al}. proposed an EMA-conserving formulation ("EMAC") for the common elements. The numerical experiments in  \cite{Rebholz2017,Lehmkuhl20182,Lehmkuhl2019,Lehmkuhl2020} demonstrated that this formulation was always among the best discretizations of the nonlinear term. In addition, the EMAC formulation was also shown to preserve the 2D enstrophy, helicity and total vorticity under certain appropriate assumptions. In the paper \cite{Rebholz2020}, Olshanskii and Rebholz proved that the Gronwall constant in the EMAC error estimates does not explicitly depend on the Reynolds number, which is contrast to other formulations such as SKEW. We note that, although the EMAC formulation has many fascinating properties, it is not easy to find a fully satisfactory linearization scheme for it. In \cite{EMAC2019}, Charnyi {\it et al.} proposed two linearization schemes for this method: a skew-symmetric linearization and the Newton linearization. The former did not conserve momentum and angular momentum, which gave a poor performance for some problems, the latter conserved momentum and angular momentum but was not energy-stable. The objective of this paper is to design an alternative formulation that is very suited to Picard linearization.

The idea for this paper goes back to a class of locally discontinuous Galerkin (DG) methods for NSEs, cf. \cite{cockburn_locally_2004}. By replacing one of the velocity with its a divergence-free approximation in the nonlinear term, a class of energy-stable DG methods were developed in \cite{cockburn_locally_2004}. Similar techniques were also applied in \cite{guzman_hdiv_2017,lederer_hybrid_2019} and to coupled flow-transport problems \cite{matthies_mass_2007,Gmeiner2014,john_divergence_2017}. Meanwhile, seeking an exactly divergence-free approximation to some discretely divergence-free velocity is also an objective of the pressure-robust reconstruction methods \cite{Linke_robust_2016,Linke2017,Linke2014on}, where a large class of divergence-free reconstruction operators were proposed, almost covered all the classical conforming elements. In this paper, we will use these operators to reconstruct CONV for conforming elements.

One motivation for our methods is that, the CONV formulation not only has the easiest form but also usually shows good performance before the scheme blows up due to the energy instability (e.g., see the comparison experiments in \cite{Rebholz2017}). In our opinion, it has the potential to simulate a problem accurately and the divergence-free reconstruction is one of the keys to unlock this potential. Here we focus on unsteady NSEs and give some theoretical analysis of the reconstructed CONV formulation. In what follows, we shall refer to it simply as ``modified CONV" sometimes. We prove that modified CONV conserves kinetic energy, linear momentum, total vorticity, 2D enstrophy and helicity under the same assumptions as EMAC \cite{Rebholz2017}. Further, we also prove that the Picard linearization conserves them as well. Compared to EMAC, our formulation does not conserve angular momentum theoretically. However, the numerical experiments show that modified CONV preserves this quantity well. For the semi-discrete scheme, we prove that the Gronwall constant of the error bounds does not depend on the Reynolds number explicitly under the assumption $\boldsymbol{u}\in L^{2}(0,T;\boldsymbol{W}^{1,\infty}(\Omega))$. In this aspect a related concept is called ({\it Re}-)semi-robustness \cite{schroeder_towards_2018,john_finite_2018}, which is a type of robustness that the constants (including Gronwall constant) in velocity error estimates do not explicitly depend on the inverse of the viscosity. For {\it Re}-semi-robust methods/analysis, we refer to \cite{schroeder_divergence-free_2018,schroeder_pressure-robust_2017,schroeder_towards_2018,CIP2007,graddiv2018,LPS2019,SPS2021}.

The rest of this paper is organized as follows. In \cref{sec:2} we state the reconstructed scheme and discuss the conservative properties and error estimates of the semi-discrete scheme. \Cref{sec:3} is devoted to investigating the conservative properties of Picard linearization, matching the Crank-Nicolson time stepping method. \Cref{sec:4} studies some numerical experiments.

Throughout the paper we use $C,$ with or without subscript, to denote a generic positive constant. The norm (seminorm) of the Sobolev space $[W^{m,p}(X)]^{n}$ or $[W^{m,p}(X)]^{n\times n}$ ($n\in\mathbb{N}^{+}$) is denoted by $\|\cdot\|_{m,p,X}$ ($|\cdot|_{m,p,X}$, respectively). The standard $L^{2}$ inner product of $[L^{2}(X)]^{n}$ or $[L^{2}(X)]^{n\times n}$ is denoted by $(\cdot, \cdot)_{X}$. When $m=0$ ($X=\Omega$, $p=2$), with the convention that the index $m$ ($X$, $p$, respectively) is omitted.
\section{A modified convective formulation}
\label{sec:2}
Let $\mathcal{T}_{h}$ be shape-regular partitions of $\Omega$ \cite{Ciarlet2002The} and $h=\max_{T\in\mathcal{T}_{h}}h_{T}$ with $h_{T}$ the diameter of element $T$. Let $V=\boldsymbol{H}_{0}^{1}(\Omega)$, $W=L_{0}^{2}(\Omega)$ and $V_{h}\times W_{h}\in V\times W$ denotes a pair of inf-sup stable finite element spaces on $\mathcal{T}_{h}$. Here we do not give them a concrete definition first.

A continuous variational form for \cref{NSE} is:
\begin{subequations}\label{continuousform}
\begin{align}
{\rm Find}~\boldsymbol{u}: (0,T] \rightarrow V \text{ and } p: (0,T] \rightarrow W~ {\rm such ~that}&\quad\quad\quad\quad\quad\qquad\nonumber\\
(\frac{\partial\boldsymbol{u}}{\partial t},\boldsymbol{v})+\nu a(\boldsymbol{u}, \boldsymbol{v})+c(\boldsymbol{u},\boldsymbol{u},\boldsymbol{v})-b(\boldsymbol{v}, p) &=(\boldsymbol{f}, \boldsymbol{v})
\quad\forall~\boldsymbol{v}\in V,\label{continuousform1} \\
b(\boldsymbol{u}, q) &=0\quad\  \forall~q\in W,\label{continuousform2}
\end{align}
\end{subequations}
where
\begin{displaymath}
a(\boldsymbol{u}, \boldsymbol{v})=(\nabla\boldsymbol{u},\nabla\boldsymbol{v}),
\end{displaymath}
\begin{displaymath}
c(\boldsymbol{u},\boldsymbol{v},\boldsymbol{w})=((\boldsymbol{u}\cdot\nabla)\boldsymbol{v},\boldsymbol{w}),
\end{displaymath}
and
\begin{displaymath}
b(\boldsymbol{v},q)=(\nabla\cdot\boldsymbol{v},q),
\end{displaymath}
for any $\boldsymbol{u},\boldsymbol{v},\boldsymbol{w}\in V$ and $q\in W$. Here the trilinear form $c(\cdot,\cdot,\cdot)$ is the so-called convective formulation (CONV).

A semi-discrete analog of \cref{continuousform} reads:
\begin{subequations}\label{semidiscretefrom}
\begin{align}
{\rm Find}~(\boldsymbol{u}_{h},p_{h}): (0,T]\rightarrow V_{h}\times W_{h}~ {\rm such ~that}\quad\quad\quad&\quad\quad\qquad\nonumber\\
(\frac{\partial\boldsymbol{u}_{h}}{\partial t},\boldsymbol{v}_{h})+\nu a(\boldsymbol{u}_{h}, \boldsymbol{v}_{h})+c(\boldsymbol{u}_{h},\boldsymbol{u}_{h},\boldsymbol{v}_{h})-b(\boldsymbol{v}_{h}, p_{h}) &=(\boldsymbol{f}, \boldsymbol{v}_{h})
\quad\forall~\boldsymbol{v}_{h}\in V_{h},\label{semidiscreteform1} \\
b(\boldsymbol{u}_{h}, q_{h}) &=0\qquad\ \ \quad\forall~q_{h}\in W_{h}.\label{semidiscreteform2}
\end{align}
\end{subequations}

Let $\boldsymbol{u}^{*}$ be the velocity solution of \cref{continuousform} or \cref{semidiscretefrom}. Introduce the kinetic energy:
$$
\text { Kinetic energy } \quad E:=\frac{1}{2} \int_{\Omega}|\boldsymbol{u}^{*}|^{2} {~d} \boldsymbol{x}.
$$
The following identity is widely used in energy-stability analysis:
\begin{equation}\label{identityEq}
c(\boldsymbol{u},\boldsymbol{v},\boldsymbol{w})=-c(\boldsymbol{u},\boldsymbol{w},\boldsymbol{v})-\left(\left(\nabla\cdot\boldsymbol{u}\right)\boldsymbol{v},\boldsymbol{w}\right)\quad\forall~ \boldsymbol{u},\boldsymbol{w},\boldsymbol{v} \in V.
\end{equation}

\Cref{identityEq} gives a skew-symmetric structure for $``\boldsymbol{v}"$ and $``\boldsymbol{w}"$ in $c(\cdot,\cdot,\cdot)$ when $``\boldsymbol{u}"$ is divergence-free. Based on this fact, one could further prove that the continuous velocity is energy-stable in \cref{continuousform}. In the discrete level, however, the divergence constraint is usually relaxed for most of the classical elements such as the Bernardi-Raugel elements \cite{bernardi_analysis_1985} and the Taylor-Hood elements \cite{girault_finite_1986,boffi_mixed_2013}. Then the skew-symmetric structure loses.
The discretely divergence-free subspace of $V_{h}$ is defined by
\begin{equation}\label{divfreesubspace}
V_{h}^{0}=\{\boldsymbol{v}_{h}\in{V}_{h}: b(\boldsymbol{v}_{h},q_{h})=0\quad\forall~q_{h}\in W_{h}\}.
\end{equation}

To fix the lack of skew-symmetry, one commonly used method is to modify the CONV formulation into the following SKEW formulation:
\begin{equation}\label{skewformula}
c_{\text{skew}}(\boldsymbol{u},\boldsymbol{v},\boldsymbol{w})=c(\boldsymbol{u},\boldsymbol{v},\boldsymbol{w})+\frac{1}{2}\left(\left(\nabla\cdot\boldsymbol{u}\right)\boldsymbol{v},\boldsymbol{w}\right).
\end{equation}
However, Olshanskii and Rebholz \cite{Rebholz2020} showed that the SKEW formulation might give a poor accuracy for long time simulations with higher Reynolds number. In this paper, we consider another skew-symmetrization way for CONV, where we shall use the divergence-free projection operators.

We define
\begin{displaymath}
\boldsymbol{H}_{0}(\operatorname{div};\Omega)=\{\boldsymbol{v}\in\boldsymbol{H}(\operatorname{div};\Omega): \boldsymbol{v}\cdot\boldsymbol{n}|_{\Gamma}=0\},
\end{displaymath}
with $\boldsymbol{n}$ the unit external normal vector on $\Gamma$.
Denote by $\Pi_{h}: V_{h} \rightarrow \boldsymbol{H}_{0}(\operatorname{div};\Omega)$ a generic divergence-free projection operator, which may be different for different space pairs. Introduce a corresponding finite element space $X_{h}\subset \boldsymbol{H}_{0}(\operatorname{div};\Omega)$ such that $\Pi_{h}V_{h}\subseteq X_{h}$.
Here we assume that $\Pi_{h}$ has the following properties:
\begin{assumption}\label{assum1}
\ding{192} $\nabla \cdot\Pi_{h}\boldsymbol{v}_{h}\equiv0$ for all $\boldsymbol{v}_{h}\in V_{h}^{0}$;
\ding{193} there exists a constant $C$ independent of $h$ such that $\left\|\Pi_{h}\boldsymbol{v}_{h}\right\|\leq C \left\|\boldsymbol{v}_{h}\right\|$ for all $\boldsymbol{v}_{h}\in V_{h}^{0}$.
\end{assumption}
For most cases, the inequality in \cref{assum1} \ding{193} can be obtained by a combination of the $L^{2}$ approximation property of $\Pi_{h}$, the inverse inequality and triangle inequality.

Introduce the modified convective formulation (modified CONV)
\begin{equation}\label{modconv}
c_{h}(\boldsymbol{u}_{h},\boldsymbol{v}_{h},\boldsymbol{w}_{h}):=c(\Pi_{h}\boldsymbol{u}_{h},\boldsymbol{v}_{h},\boldsymbol{w}_{h}).
\end{equation}
We consider the reconstructed scheme: ${\rm Find}~(\boldsymbol{u}_{h},p_{h}): (0,T]\rightarrow V_{h}\times W_{h}~ {\rm such ~that}$
\begin{subequations}\label{EMsemidiscreteform}
\begin{align}
(\frac{\partial\boldsymbol{u}_{h}}{\partial t},\boldsymbol{v}_{h})+\nu a(\boldsymbol{u}_{h}, \boldsymbol{v}_{h})+c_{h}(\boldsymbol{u}_{h},\boldsymbol{u}_{h},\boldsymbol{v}_{h})-b(\boldsymbol{v}_{h}, p_{h}) &=(\boldsymbol{f}, \boldsymbol{v}_{h})
\quad\forall~\boldsymbol{v}_{h}\in V_{h}, \label{EMsemidiscreteform1} \\
b(\boldsymbol{u}_{h}, q_{h}) &=0\qquad\ \, \quad\forall~q_{h}\in W_{h}. \label{EMsemidiscreteform2}
\end{align}
\end{subequations}

The following lemma is very essential for our analysis.
\begin{lemma}\label{essentiallemma}
For any $(\boldsymbol{u}_{h},\boldsymbol{v}_{h},\boldsymbol{w}_{h})\in X_{h}\times V_{h} \times V_{h}$ we have
\begin{equation}\nonumber
c(\boldsymbol{u}_{h},\boldsymbol{v}_{h},\boldsymbol{w}_{h})=-c(\boldsymbol{u}_{h},\boldsymbol{w}_{h},\boldsymbol{v}_{h})
\end{equation} if
\begin{itemize}
   \item[\text{1)}] $\nabla\cdot\boldsymbol{u}_{h}=0$;
   \item[\text{2)}] $\boldsymbol{u}_{h}\cdot\boldsymbol{n}=0$ or $\boldsymbol{v}_{h}=\boldsymbol{0}$ or $\boldsymbol{w}_{h}=\boldsymbol{0}$ on $\Gamma$.
\end{itemize}
\end{lemma}
\begin{proof}
\cref{essentiallemma} can be regarded as a special case of the skew-symmetric discontinuous Galerkin formulation (see, e.g., \cite[Lemma 6.39]{di_pietro_mathematical_2012}). Here for completeness we reprove this special case. To make the proof process more clear, we will derive it in the component form. Let $(v_{1},v_{2},...,v_{d})$ denotes the component form of any $\boldsymbol{v}_{h}\in \boldsymbol{L}^{2}(\Omega)$. We have
\begin{equation}\label{essential1}
c(\boldsymbol{u}_{h},\boldsymbol{v}_{h},\boldsymbol{w}_{h})=\left(\left(\nabla\boldsymbol{v}_{h}\right)\boldsymbol{u}_{h},\boldsymbol{w}_{h}\right)
=\sum_{1\leq i,j\leq d}\sum_{T\in\mathcal{T}_{h}}\int_{T}u_{i}\frac{\partial v_{j}}{\partial x_{i}}w_{j}~d\boldsymbol{x}.
\end{equation}
The integration by parts on each element $T$ gives
\begin{equation}\label{essential2}
\begin{aligned}
\int_{T}u_{i}\frac{\partial v_{j}}{\partial x_{i}}w_{j}~d\boldsymbol{x}&=-\int_{T}\frac{\partial (u_{i}w_{j})}{\partial x_{i}}v_{j}~d\boldsymbol{x}+\int_{\partial T}u_{i}n_{i}v_{j}w_{j}~ds\\
&=-\int_{T}\left(u_{i}\frac{\partial w_{j}}{\partial x_{i}}v_{j}+\frac{\partial u_{i}}{\partial x_{i}}v_{j}w_{j}\right)~d\boldsymbol{x}+\int_{\partial T}u_{i}n_{i}v_{j}w_{j}~ds,
\end{aligned}
\end{equation}
where $n_{i}$ denotes the $i$-th component of the unit external normal vector of $T$, $\boldsymbol{n}$. A combination of the continuity and boundary conditions of $\boldsymbol{u}_{h},\boldsymbol{v}_{h}$ and $\boldsymbol{w}_{h}$ implies that
\begin{equation}\label{essential3}
\sum_{1\leq i,j\leq d}\sum_{T\in\mathcal{T}_{h}}\int_{\partial T}u_{i}n_{i}v_{j}w_{j}~ds=\sum_{T\in\mathcal{T}_{h}}\int_{\partial T}(\boldsymbol{u}_{h}\cdot\boldsymbol{n})(\boldsymbol{v}_{h}\cdot\boldsymbol{w}_{h})~ds=0,
\end{equation}
and the first condition in \cref{essentiallemma} gives
\begin{equation}\label{essential4}
\begin{aligned}
\sum_{1\leq i,j\leq d}\sum_{T\in\mathcal{T}_{h}}\int_{T}u_{i}\frac{\partial w_{j}}{\partial x_{i}}v_{j}+\frac{\partial u_{i}}{\partial x_{i}}v_{j}w_{j}~d\boldsymbol{x}&=
c(\boldsymbol{u}_{h},\boldsymbol{w}_{h},\boldsymbol{v}_{h})+\left(\left(\nabla\cdot\boldsymbol{u}_{h}\right)\boldsymbol{v}_{h},\boldsymbol{w}_{h}\right)\\
&=c(\boldsymbol{u}_{h},\boldsymbol{w}_{h},\boldsymbol{v}_{h}).
\end{aligned}
\end{equation}
Then \cref{essentiallemma} follows immediately from \cref{essential1,essential2,essential3,essential4}.
\end{proof}

By \cref{essentiallemma} and \cref{assum1} one can obtain
\begin{equation}\label{essentialformu1}
c_{h}(\boldsymbol{u}_{h},\boldsymbol{v}_{h},\boldsymbol{w}_{h})=-c_{h}(\boldsymbol{u}_{h},\boldsymbol{w}_{h},\boldsymbol{v}_{h}) \text{ for all } \boldsymbol{u}_{h}\in V_{h}^{0}, \boldsymbol{v}_{h},\boldsymbol{w}_{h}\in V_{h}.
\end{equation}
\subsection{Conservation laws}
In this subsection, we focus on the conservative properties of the modified method \cref{EMsemidiscreteform}.
First, let us give the definitions of momentum, angular momentum, helicity, enstrophy and vorticity. Denote by $\boldsymbol{u}^{*}$ the exact velocity in \cref{NSE} or the discrete velocity in \cref{EMsemidiscreteform}. Let $\boldsymbol{w}^{*}=\textbf{curl}\boldsymbol{u}^{*}$ when $\boldsymbol{u}^{*}$ represents the continuous solution; let $\boldsymbol{w}^{*}$ be the solution of some finite element discretization of the Navier-Stokes vorticity equation (see \cref{curlNSE} and \cref{curlNSEdiscrete} below) when $\boldsymbol{u}^{*}$ is the discrete velocity. Note that the operator \textbf{curl} in two dimensions denotes a scalar operator \cite{girault_finite_1986}. Then we define
$$
\text { Linear momentum } \quad M:=\int_{\Omega} \boldsymbol{u}^{*} ~{d} \boldsymbol{x} ; $$
$$\text { Angular momentum } \quad M_{\boldsymbol{x}}:=\int_{\Omega} \boldsymbol{u}^{*} \times \boldsymbol{x} ~{d} \boldsymbol{x} ;
$$
$$
\text { Helicity } \quad H:=\int_{\Omega} \boldsymbol{u}^{*} \cdot \boldsymbol{w}^{*} {~d} \boldsymbol{x} ;$$
$$
\text { Enstrophy } \quad H_{E}:=\frac{1}{2} \int_{\Omega} \boldsymbol{w}^{*} \cdot \boldsymbol{w}^{*} {~d} \boldsymbol{x} ;
$$
$$\text { Total vorticity } \quad W:=\int_{\Omega} \boldsymbol{w}^{*} {~d} \boldsymbol{x}.
$$
\begin{remark}
In what follows we shall discuss some conservative properties for the case $\nu=0$ (the Euler equations) sometimes. For an Euler equation, the boundary conditions should be altered in \cref{NSE}. Here we replace the no-slip boundary condition ($\boldsymbol{u}=\boldsymbol{0}$) with the no-penetration boundary condition ($\boldsymbol{u}\cdot\boldsymbol{n}=0$) for it. In this case the method \cref{EMsemidiscreteform} should only strongly enforce the no-penetration boundary condition also. By abuse of notation, we shall use \cref{NSE} and \cref{EMsemidiscreteform} to denote the Euler equation and the corresponding finite element methods as well, respectively.
\end{remark}
\begin{theorem}
Under \cref{assum1} \ding{192}, the modified method \cref{EMsemidiscreteform} properly balances the kinetic energy:
\begin{equation}\label{energyconser}
\frac{1}{2}\frac{d}{dt}||\boldsymbol{u}_{h}||^{2}+\nu\|\nabla\boldsymbol{u}_{h}\|^{2}=(\boldsymbol{f},\boldsymbol{u}_{h}).
\end{equation}
\end{theorem}
\begin{proof}
\Cref{essentialformu1} implies that $c_{h}(\boldsymbol{u}_{h},\boldsymbol{u}_{h},\boldsymbol{u}_{h})=0$. Then \cref{energyconser} follows immediately  from taking $\boldsymbol{v}_{h}=\boldsymbol{u}_{h}$ in \cref{EMsemidiscreteform1}.
\end{proof}

\begin{remark}
When $\nu=0$ and $\boldsymbol{f}=\boldsymbol{0}$, one can similarly obtain $\frac{1}{2}\frac{d}{dt}\|\boldsymbol{u}_{h}\|^{2}=0$. In this case, kinetic energy is conserved by our method.
\end{remark}

To analysis the conservative properties of the other quantities, we need some extra assumptions, which are similar to the EMAC analysis \cite{Rebholz2017}.
\begin{assumption}\label{compactsupport}
The exact solution $(\boldsymbol{u},p)$, the discrete solution $(\boldsymbol{u}_{h},p_{h})$ and the source term $\boldsymbol{f}$ are only supported on a sub-domain $\hat{\Omega}\subset\Omega$ such that there exists restriction $\chi(\boldsymbol{g})\in V_{h}$ for $\boldsymbol{g}=\boldsymbol{e}_{i}, \boldsymbol{x}\times \boldsymbol{e}_{i}$ with $\chi(\boldsymbol{g})=\boldsymbol{g}$ in $\hat{\Omega}$ and $\chi(\boldsymbol{g})=\boldsymbol{0}$ on $\Gamma$. Here $\boldsymbol{e}_{i}$ represents the $i$-th usual basis of $\mathbb{R}^{d}$.
\end{assumption}
\begin{theorem}\label{mhevconser}
Under \cref{assum1} \ding{192} and \cref{compactsupport}, the modified method \cref{EMsemidiscreteform} conserves momentum (for $\boldsymbol{f}$ with zero linear momentum), helicity (for $\nu=0$ and $\boldsymbol{f}=\boldsymbol{0}$), 2D enstrophy (for $\nu=0$ and $\boldsymbol{f}=\boldsymbol{0}$), and total vorticity.
\end{theorem}
We divide the proof of the above theorem into several subsections.
\subsubsection{Momentum}
Testing with $(\boldsymbol{v}_{h},q_{h})=(\chi({\boldsymbol{e}_{i}}),0)$ in \cref{EMsemidiscreteform} and applying \cref{essentialformu1} give that
\begin{displaymath}
\left(\boldsymbol{f}, \boldsymbol{e}_{i}\right)=\frac{d}{d t}\left(\boldsymbol{u}_{h}, \boldsymbol{e}_{i}\right)+c_{h}(\boldsymbol{u}_{h},\boldsymbol{u}_{h},\boldsymbol{e}_{i})=\frac{d}{d t}\left(\boldsymbol{u}_{h}, \boldsymbol{e}_{i}\right)- c_{h}(\boldsymbol{u}_{h},\boldsymbol{e}_{i},\boldsymbol{u}_{h})=\frac{d}{d t}\left(\boldsymbol{u}_{h}, \boldsymbol{e}_{i}\right).
\end{displaymath}
Then the conservation of momentum follows.
\begin{remark}
For the analysis of angular momentum, note that $\left(\boldsymbol{u}_{h}\times \boldsymbol{x}\right)\cdot\boldsymbol{e}_{i}=\boldsymbol{u}_{h}\cdot (\boldsymbol{x}\times\boldsymbol{e}_{i})$. Then substituting $(\boldsymbol{v}_{h},q_{h})=(\chi(\boldsymbol{x}\times \boldsymbol{e}_{i}),0)$ into \cref{EMsemidiscreteform}, one could similarly obtain that
\begin{displaymath}
\frac{d}{d t}\left(\boldsymbol{u}_{h}, \boldsymbol{x}\times\boldsymbol{e}_{i}\right)- c_{h}(\boldsymbol{u}_{h},\boldsymbol{x}\times\boldsymbol{e}_{i},\boldsymbol{u}_{h})=\left(\boldsymbol{f}, \boldsymbol{x}\times\boldsymbol{e}_{i}\right).
\end{displaymath}
Here we apply the fact that $\nabla\cdot(\boldsymbol{x}\times \boldsymbol{e}_{i})=0$ and $\Delta(\boldsymbol{x}\times \boldsymbol{e}_{i})=0$.
Assuming that $\boldsymbol{f}$ has zero angular momentum, i.e., $\left(\boldsymbol{f}, \boldsymbol{x}\times\boldsymbol{e}_{i}\right)=0$, the above equation gives that $\frac{d}{d t}\left(\boldsymbol{u}_{h}, \boldsymbol{x}\times\boldsymbol{e}_{i}\right)=c_{h}(\boldsymbol{u}_{h},\boldsymbol{x}\times\boldsymbol{e}_{i},\boldsymbol{u}_{h})=-\left(\Pi_{h}\boldsymbol{u}_{h}\times\boldsymbol{u}_{h},\boldsymbol{e}_{i}\right)\neq0,$
since $\Pi_{h}\boldsymbol{u}_{h}\neq \boldsymbol{u}_{h}$ in general. Thus angular momentum is not preserved by our methods.
\end{remark}
\subsubsection{Helicity}
Let $\boldsymbol{w}$ be the curl of the exact solution $\boldsymbol{u}$: $\boldsymbol{w}=\nabla\times \boldsymbol{u}$. Note that $\nabla\cdot\boldsymbol{w}=0$. In two dimensions, $\boldsymbol{w}$ represents a scalar and we use the symbol $w$ to denote it. Taking the curl operator on \cref{NSE1} gives the following Navier-Stokes
vorticity equations:
\begin{equation}\label{curlNSE}
\boldsymbol{w}_{t}+\left(\boldsymbol{u} \cdot \nabla\right) \boldsymbol{w}-\left(\boldsymbol{w} \cdot \nabla\right) \boldsymbol{u}-\nu \Delta \boldsymbol{w}=\nabla\times \boldsymbol{f}.
\end{equation}

For the discrete case, we apply the strategy in \cite{Rebholz2010,Rebholz2017}. Here we do not analyze the quantity $\boldsymbol{w}_{h}=\textbf{curl}\boldsymbol{u}_{h}$ for the discrete solution $\boldsymbol{u}_{h}$. We consider a sightly altered virticity $\boldsymbol{w}_{h}$ which is a solution of some direct discretization of \cref{curlNSE}. As it was said in \cite{Rebholz2017}, {\it `this discrete vorticity still depends on the computed
velocity $\boldsymbol{u}_{h}$, but more implicitly, through the equation coefficients'}. We further assume
vorticity also vanishes on and near the boundary due to \cref{compactsupport}. The corresponding finite element methods to \cref{curlNSE} reads: Find $\boldsymbol{w}_{h}: (0,T] \rightarrow V_{h}$ and the Lagrange multifier $\eta_{h}: (0,T] \rightarrow W_{h}$ such that
\begin{equation}\label{curlNSEdiscrete}
\begin{aligned}
\left(\frac{\partial \boldsymbol{w}_{h}}{\partial t}, \boldsymbol{v}_{h}\right)+c_{h}(\boldsymbol{u}_{h}, \boldsymbol{w}_{h}, \boldsymbol{v}_{h})&-c_{h}(\boldsymbol{w}_{h}, \boldsymbol{u}_{h}, \boldsymbol{v}_{h})+\\ \nu(\nabla \boldsymbol{w}_{h}, \nabla \boldsymbol{v}_{h})-&b( \boldsymbol{v}_{h},\eta_{h})+b(\boldsymbol{w}_{h},q_{h})=(\nabla\times \boldsymbol{f}, \boldsymbol{v}_{h}),
\end{aligned}
\end{equation}
for any $\boldsymbol{v}_{h}\in V_{h}$ and $q_{h}\in W_{h}$, where $\boldsymbol{u}_{h}$ is the solution of \cref{EMsemidiscreteform}. Note that \cref{curlNSEdiscrete} implies that $b(\boldsymbol{w}_{h},q_{h})=0$ for all $q_{h}\in W_{h}$ and further $\nabla\cdot\Pi_{h}\boldsymbol{w}_{h}\equiv 0$.

Testing with $(\boldsymbol{v}_{h},q_{h})=(\boldsymbol{w}_{h},0)$ in \cref{EMsemidiscreteform} gives
\begin{equation}\label{helicity1}
\left(\frac{\partial \boldsymbol{u}_{h}}{\partial t}, \boldsymbol{w}_{h}\right)+\nu(\nabla \boldsymbol{u}_{h}, \nabla \boldsymbol{w}_{h})+c_{h}\left(\boldsymbol{u}_{h},\boldsymbol{u}_{h}, \boldsymbol{w}_{h}\right)=(\boldsymbol{f},\boldsymbol{w}_{h}).
\end{equation}
Meanwhile, testing with $(\boldsymbol{v}_{h},q_{h})=(\boldsymbol{u}_{h},0)$ in \cref{curlNSEdiscrete} gives
\begin{equation}\label{helicity2}
\left(\frac{\partial \boldsymbol{w}_{h}}{\partial t}, \boldsymbol{u}_{h}\right)+c_{h}(\boldsymbol{u}_{h}, \boldsymbol{w}_{h}, \boldsymbol{u}_{h})-c_{h}(\boldsymbol{w}_{h}, \boldsymbol{u}_{h}, \boldsymbol{u}_{h})+ \nu(\nabla \boldsymbol{w}_{h}, \nabla \boldsymbol{u}_{h})=(\nabla\times \boldsymbol{f}, \boldsymbol{u}_{h}).
\end{equation}
Adding \cref{helicity1} and \cref{helicity2} and setting $\nu=0$ and $\boldsymbol{f}=\boldsymbol{0}$, one obtain
\begin{equation}\nonumber
\frac{d}{d t}\left(\boldsymbol{u}_{h},\boldsymbol{w}_{h}\right)+c_{h}(\boldsymbol{u}_{h}, \boldsymbol{w}_{h}, \boldsymbol{u}_{h})-c_{h}(\boldsymbol{w}_{h}, \boldsymbol{u}_{h}, \boldsymbol{u}_{h})+c_{h}(\boldsymbol{u}_{h}, \boldsymbol{u}_{h}, \boldsymbol{w}_{h})=\frac{d}{d t}\left(\boldsymbol{u}_{h},\boldsymbol{w}_{h}\right)=0,
\end{equation}
where we use the property \cref{essentialformu1}. Thus helicity is also conserved.
\subsubsection{2D Enstrophy}
In two dimensions the vorticity $w$ satisfies
\begin{displaymath}
w_{t}+(\boldsymbol{u} \cdot \nabla) w-\nu\Delta w=\nabla\times\boldsymbol{f}
\end{displaymath}
The corresponding finite element scheme reads:
\begin{equation}\label{2DcurlNSEdiscrete}
\left(\frac{\partial w_{h}}{\partial t}, v\right)+((\Pi_{h}\boldsymbol{u}_{h} \cdot \nabla) w_{h}, v_{h})+\nu(\nabla w_{h}, \nabla v_{h})=(\nabla\times\boldsymbol{f}, v_{h})
\end{equation}
Similarly, we have $((\Pi_{h}\boldsymbol{u}_{h} \cdot \nabla) w_{h}, v_{h})=-((\Pi_{h}\boldsymbol{u}_{h} \cdot \nabla) v_{h}, w_{h})$.
Thus, taking $v_{h}=w_{h}$ and for the case $\nu=0,\boldsymbol{f}=0$ we have
\begin{equation}
\frac{1}{2} \frac{d}{d t}\|w_{h}\|^{2}=0.
\end{equation}
Thus we prove the conservative property of the 2D enstrophy.

For 3D flows, taking $(\boldsymbol{v}_{h},q_{h})=(\boldsymbol{w}_{h},0)$ in \cref{curlNSEdiscrete} we arrive at
\begin{equation}\label{3Denstrophy}
\frac{1}{2}\frac{d}{dt}\|\boldsymbol{w}_{h}\|^{2}-c_{h}(\boldsymbol{w}_{h}, \boldsymbol{u}_{h}, \boldsymbol{w}_{h})=0,
\end{equation}
in the case where $\nu=0$ and $\boldsymbol{f}=\boldsymbol{0}$. Note that the exact vorticity $\boldsymbol{w}$ satisfies
\begin{equation}\label{3Denstrophycontinuous}
\frac{1}{2}\frac{d}{dt}\|\boldsymbol{w}\|^{2}-c(\boldsymbol{w}, \boldsymbol{u}, \boldsymbol{w})=0.
\end{equation}
There is a little difference between the continuous case \cref{3Denstrophycontinuous} and the discrete case \cref{3Denstrophy} formally.
\subsubsection{Vorticity}
Note that $(\nabla\times\boldsymbol{f},\boldsymbol{e}_{i})=(\boldsymbol{f},\nabla\times\boldsymbol{e}_{i})=0$ under \cref{compactsupport}. Set $(\boldsymbol{v}_{h},q_{h})=(\chi(\boldsymbol{e}_{i}),0)$ in \cref{curlNSEdiscrete} and we obtain
\begin{equation}\label{vorticity1}
(\frac{\partial\boldsymbol{w}_{h}}{\partial t},\boldsymbol{e}_{i})+c_{h}(\boldsymbol{u}_{h}, \boldsymbol{w}_{h}, \boldsymbol{e}_{i})-c_{h}(\boldsymbol{w}_{h}, \boldsymbol{u}_{h},\boldsymbol{e}_{i})=0.
\end{equation}
\Cref{essentialformu1} implies that
\begin{equation}\label{vorticity2}
c_{h}(\boldsymbol{u}_{h}, \boldsymbol{w}_{h}, \boldsymbol{e}_{i})=-c_{h}(\boldsymbol{u}_{h}, \boldsymbol{e}_{i}, \boldsymbol{w}_{h})=0,
\end{equation}
and
\begin{equation}\label{vorticity3}
c_{h}(\boldsymbol{w}_{h}, \boldsymbol{u}_{h}, \boldsymbol{e}_{i})=-c_{h}(\boldsymbol{w}_{h}, \boldsymbol{e}_{i}, \boldsymbol{u}_{h})=0.
\end{equation}
Substituting \cref{vorticity2} and \cref{vorticity3} into \cref{vorticity1} gives that
\begin{displaymath}
(\frac{\partial\boldsymbol{w}_{h}}{\partial t},\boldsymbol{e}_{i})=0,
\end{displaymath}
which implies that the total vorticity is conserved.
\subsection{Semi-discrete error analysis}
Denote by $\Pi_{h}^{S}: \boldsymbol{H}^{1}_{0}(\Omega)\rightarrow V_{h}^{0}$ the Stokes projection which satisfies that
\begin{equation}\label{Stokesprojection}
\left(\nabla \Pi_{h}^{S} \boldsymbol{z}, \nabla \boldsymbol{v}_{h}\right)=\left(\nabla \boldsymbol{z}, \nabla \boldsymbol{v}_{h}\right) \quad \forall~\boldsymbol{z}\in \boldsymbol{H}^{1}_{0}(\Omega), \boldsymbol{v}_{h} \in {V}_{h}^{0}.
\end{equation}
\begin{assumption}\label{assumeinfty}
The Stokes projection satisfies the following estimate
\begin{equation}\label{inftyestimate}
\left\|\nabla\Pi_{h}^{S} \boldsymbol{w}\right\|_{p} \leqslant C\left\|\nabla \boldsymbol{w}\right\|_{p}, \quad 1 \leq p \leq\infty.
\end{equation}
\end{assumption}
To satisfy the estimate \cref{inftyestimate}, some extra conditions might be required. For example, the domain $\Omega$ is convex. For the concrete analysis of \cref{inftyestimate}, we refer the readers to the paper \cite{2015Max}.
Let
$$\boldsymbol{e}=\boldsymbol{u}-\boldsymbol{u}_{h}=\boldsymbol{u}-\Pi_{h}^{S}\boldsymbol{u}
+\Pi_{h}^{S}\boldsymbol{u}-\boldsymbol{u}_{h}=\boldsymbol{\eta}+\boldsymbol{\phi}_{h},$$
and
$$\boldsymbol{e}_{\pi}=\boldsymbol{u}-\Pi_{h}\boldsymbol{u}_{h}=\boldsymbol{u}-\Pi_{h}\Pi_{h}^{S}\boldsymbol{u}
+\Pi_{h}\left(\Pi_{h}^{S}\boldsymbol{u}-\boldsymbol{u}_{h}\right)=\boldsymbol{\eta}_{\pi}+\Pi_{h}\boldsymbol{\phi}_{h}.$$

The following error equation will be used in error analysis.
\begin{equation}\label{erroreq}
\begin{aligned}
\left(\boldsymbol{e}_{t}, \boldsymbol{v}_{h}\right)+c\left(\boldsymbol{u}, \boldsymbol{u}, \boldsymbol{v}_{h}\right)&-c\left(\Pi_{h}\boldsymbol{u}_{h}, \boldsymbol{u}_{h}, \boldsymbol{v}_{h}\right)\\&+\nu\left(\nabla \boldsymbol{e}, \nabla \boldsymbol{v}_{h}\right)-b(\boldsymbol{v}_{h},p-q_{h})=0 \quad \forall~ \boldsymbol{v}_{h} \in {V}_{h}^{0}, q_{h}\in W_{h}.
\end{aligned}
\end{equation}
The above equation can be further rewritten as
\begin{equation}\label{erroreq1}
\begin{aligned}
\left(\left(\boldsymbol{\phi}_{h}\right)_{t}, \boldsymbol{v}_{h}\right)&+\nu\left(\nabla \boldsymbol{\phi}_{h}, \nabla \boldsymbol{v}_{h}\right)=-\left(\boldsymbol{\eta}_{t}, \boldsymbol{v}_{h}\right)-\left(c\left(\boldsymbol{u}, \boldsymbol{u}, \boldsymbol{v}_{h}\right)-c\left(\Pi_{h}\boldsymbol{u}_{h}, \boldsymbol{u}_{h}, \boldsymbol{v}_{h}\right)\right)\\&+b(\boldsymbol{v}_{h},p-q_{h})=0 \quad \forall~ \boldsymbol{v}_{h} \in {V}_{h}^{0}, q_{h}\in W_{h},
\end{aligned}
\end{equation}
where we use the fact that $(\nabla\boldsymbol{\eta},\boldsymbol{v}_{h})=0$.
\begin{theorem}[Semi-discrete estimate]
Let $\boldsymbol{u}_{h}$ solve \cref{EMsemidiscreteform} and $(\boldsymbol{u}, p)$ solve \cref{continuousform} with $\boldsymbol{u}_{t} \in L^{2}\left(0, T ; \boldsymbol{H}^{-1}(\Omega)\right)$, $\boldsymbol{u} \in L^{4}\left(0, T ; \boldsymbol{H}^{1}(\Omega)\right)\cap L^{2}\left(0,T ; \boldsymbol{W}^{1,\infty}(\Omega)\right)$ and $p \in L^{2}\left(0, T ; L^{2}(\Omega)\right) .$ Under \cref{assum1} and \cref{assumeinfty} it holds
\begin{equation}\label{errorestimate1}
\begin{aligned}
\|\boldsymbol{e}(T)\|^{2}&+\nu \int_{0}^{T}\|\nabla \boldsymbol{e}\|^{2} ~d t \leq \|\boldsymbol{\eta}(T)\|^{2}+\nu\|\nabla \boldsymbol{\eta}\|_{L^{2}\left(0,T ; \boldsymbol{L}^{2}\right)}^{2}
+K(\boldsymbol{u})\Bigg(\|\boldsymbol{\phi}_{h}(0)\|^{2}\left.\right.\\&\left.+\nu^{-1}\left\|\boldsymbol{\eta}_{t}\right\|_{L^{2}\left(0,T ; \boldsymbol{H}^{-1}\right)}^{2}\right.
\left.+\nu^{-1} \inf _{q_{h} \in L^{2}\left(0,T ; W_{h}\right)}\left\|p-q_{h}\right\|_{L^{2}\left(0,T ; L^{2}\right)}^{2}\right.\\
&+B(\boldsymbol{u})\left(\|\nabla\boldsymbol{\eta}\|_{L^{4}\left(0,T ; \boldsymbol{L}^{2}\right)}^{2}+\|\boldsymbol{\eta}_{\pi}\|_{L^{4}\left(0,T ; \boldsymbol{L}^{2}\right)}^{2}\right)\Bigg)
\end{aligned}
\end{equation}
with
$$
K(\boldsymbol{u})=\exp \left(C\left(\|\boldsymbol{u}\|_{L^{1}\left(0, T ; \boldsymbol{L}^{\infty}\right)}+\|\nabla \boldsymbol{u}\|_{L^{1}\left(0, T ; \boldsymbol{L}^{\infty}\right)}\right)\right),
$$
and
$$
B(\boldsymbol{u})=C\left(\|\boldsymbol{u}\|_{L^{2}\left(0, T ; \boldsymbol{L}^{\infty}\right)}+\|\nabla \boldsymbol{u}\|_{L^{2}\left(0, T ; \boldsymbol{L}^{\infty}\right)}\right),
$$
independent of $\nu^{-1}.$
\end{theorem}
\begin{proof}
Set $\boldsymbol{v}_{h}=\boldsymbol{\phi}_{h}$ in \cref{erroreq1} and we arrive at
\begin{equation}\label{erroreq2}
\begin{aligned}
\frac{1}{2}\frac{d}{dt}||\boldsymbol{\phi}_{h}||^{2}+&\nu||\nabla\boldsymbol{\phi}_{h}||^{2}=
-(\boldsymbol{\eta}_{t},\boldsymbol{\phi}_{h})-\\
&\left(c(\boldsymbol{u},\boldsymbol{u},\boldsymbol{\phi}_{h})
-c(\Pi_{h}\boldsymbol{u}_{h},\boldsymbol{u}_{h},\boldsymbol{\phi}_{h})\right)+b(\phi_{h},p-q_{h}).
\end{aligned}
\end{equation}
Applying the Cauchy-Schwarz inequality and the weighted Young's inequality one obtains
\begin{equation}\label{errorieq1}
\left|\left(\boldsymbol{\eta}_{t}, \boldsymbol{\phi}_{h}\right)\right| \leq  \nu^{-1}\left\|\boldsymbol{\eta}_{t}\right\|_{\boldsymbol{H}^{-1}}^{2}+\frac{\nu}{4}\left\|\nabla \boldsymbol{\phi}_{h}\right\|^{2},
\end{equation}
\begin{equation}\label{errorieq2}
\left|b\left(\boldsymbol{\phi}_{h},p-q_{h}\right)\right| \leq C \nu^{-1}\left\|p-q_{h}\right\|^{2}+\frac{\nu}{4}\left\|\nabla \boldsymbol{\phi}_{h}\right\|^{2}.
\end{equation}
To analyze the nonlinear terms, here we use a similar splitting way as \cite{schroeder_towards_2018,Rebholz2020}. We split the difference of the two nonlinear terms as
\begin{equation}\label{errorieq3}
c(\boldsymbol{u},\boldsymbol{u},\boldsymbol{\phi}_{h})
-c(\Pi_{h}\boldsymbol{u}_{h},\boldsymbol{u}_{h},\boldsymbol{\phi}_{h})=c(\boldsymbol{u},\boldsymbol{\eta},\boldsymbol{\phi}_{h})
+c(\boldsymbol{u},\Pi_{h}^{S}\boldsymbol{u},\boldsymbol{\phi}_{h})
-c(\Pi_{h}\boldsymbol{u}_{h},\boldsymbol{u}_{h},\boldsymbol{\phi}_{h}).
\end{equation}
Further,
\begin{equation}\label{errorieq4}
\begin{aligned}
c(\boldsymbol{u},\Pi_{h}^{S}\boldsymbol{u},\boldsymbol{\phi}_{h})
-c(\Pi_{h}\boldsymbol{u}_{h},\boldsymbol{u}_{h},\boldsymbol{\phi}_{h})&=
c(\boldsymbol{e}_{\pi},\Pi_{h}^{S}\boldsymbol{u},\boldsymbol{\phi}_{h})
+c(\Pi_{h}\boldsymbol{u}_{h},\Pi_{h}^{S}\boldsymbol{u},\boldsymbol{\phi}_{h})\\
&-c(\Pi_{h}\boldsymbol{u}_{h},\boldsymbol{u}_{h},\boldsymbol{\phi}_{h})\\&=
c(\boldsymbol{e}_{\pi},\Pi_{h}^{S}\boldsymbol{u},\boldsymbol{\phi}_{h})
+c(\Pi_{h}\boldsymbol{u}_{h},\boldsymbol{\phi}_{h},\boldsymbol{\phi}_{h})\\&=
c(\boldsymbol{e}_{\pi},\Pi_{h}^{S}\boldsymbol{u},\boldsymbol{\phi}_{h})\\
& = c(\boldsymbol{\eta}_{\pi},\Pi_{h}^{S}\boldsymbol{u},\boldsymbol{\phi}_{h})+
c(\Pi_{h}\boldsymbol{\phi}_{h},\Pi_{h}^{S}\boldsymbol{u},\boldsymbol{\phi}_{h}).
\end{aligned}
\end{equation}
Then the Cauchy-Schwarz inequality, \Cref{inftyestimate} and the Young's inequality give that
\begin{equation}\label{errorieq5}
|c(\boldsymbol{u},\boldsymbol{\eta},\boldsymbol{\phi}_{h})|\leq  ||\boldsymbol{u}||_{\infty}||||\nabla\boldsymbol{\eta}||
||\boldsymbol{\phi}_{h}||
\leq C ||\boldsymbol{u}||_{\infty}(||\nabla\boldsymbol{\eta}||^{2}+||\boldsymbol{\phi}_{h}||^{2}),
\end{equation}
\begin{equation}\label{errorieq6}
|c(\boldsymbol{\eta}_{\pi},\Pi_{h}^{S}\boldsymbol{u},\boldsymbol{\phi}_{h})|\leq
||\boldsymbol{\eta}_{\pi}||||\nabla\Pi_{h}^{S}\boldsymbol{u}||_{\infty}||\boldsymbol{\phi}_{h}||
\leq C ||\nabla\boldsymbol{u}||_{\infty}
(||\boldsymbol{\eta}_{\pi}||^{2}+||\boldsymbol{\phi}_{h}||^{2}),
\end{equation}
and
\begin{equation}\label{errorieq7}
\begin{aligned}
|c(\Pi_{h}\boldsymbol{\phi}_{h},\Pi_{h}^{S}\boldsymbol{u},\boldsymbol{\phi}_{h})|\leq
||\Pi_{h}\boldsymbol{\phi}_{h}||||\nabla\Pi_{h}^{S}\boldsymbol{u}||_{\infty}||\boldsymbol{\phi}_{h}||
\leq C ||\nabla\boldsymbol{u}||_{\infty}||\boldsymbol{\phi}_{h}||^{2},
\end{aligned}
\end{equation}
where in the last inequality we also use \cref{assum1}.
Substituting \cref{errorieq5,errorieq6,errorieq7} into \cref{errorieq3,errorieq4} provides
\begin{equation}\nonumber
\begin{array}{l}
\left|c\left(\boldsymbol{u}, \boldsymbol{u}, \boldsymbol{\phi}_{h}\right)-c\left(\Pi_{h}\boldsymbol{u}_{h}, \boldsymbol{u}_{h}, \boldsymbol{\phi}_{h}\right)\right| \\
\quad \leq C\left(\|\boldsymbol{u}\|_{{\infty}}\|\nabla \boldsymbol{\eta}\|^{2}+\|\nabla \boldsymbol{u}\|_{{\infty}}\|\boldsymbol{\eta}_{\pi}\|^{2}+\left[\|\boldsymbol{u}\|_{{\infty}}+\|\nabla \boldsymbol{u}\|_{{\infty}}\right]\left\|\boldsymbol{\phi}_{h}\right\|^{2}\right),
\end{array}
\end{equation}
which, together with \cref{errorieq1,errorieq2}, gives that
\begin{equation}
\begin{array}{l}
\frac{1}{2}\frac{d}{d t}\left\|\boldsymbol{\phi}_{h}\right\|^{2}+\frac{\nu}{2}\left\|\nabla \boldsymbol{\phi}_{h}\right\|^{2} \leq  \nu^{-1}\left\|\boldsymbol{\eta}_{t}\right\|_{\boldsymbol{H}^{-1}}^{2}+C \nu^{-1}\left\|p-q_{h}\right\|^{2} \\
\quad+C \left(||\boldsymbol{u}||_{\infty}||\nabla\boldsymbol{\eta}||^{2}+||\nabla\boldsymbol{u}||_{\infty}||\boldsymbol{\eta}_{\pi}||^{2}\right)
+C (||\boldsymbol{u}||_{\infty}+||\nabla\boldsymbol{u}||_{\infty})||\boldsymbol{\phi}_{h}||^{2}.
\end{array}
\end{equation}
Then integrating over $(0,T]$ and using the Gronwall inequality and H{\"o}lder inequality we finally obtain
\begin{equation}\label{errorestimate2}
\begin{aligned}
\|\boldsymbol{\phi}_{h}(T)\|^{2}+&\nu \int_{0}^{T}\|\nabla \boldsymbol{\phi}_{h}\|^{2} ~d t \leq K(\boldsymbol{u})\left(\|\boldsymbol{\phi}_{h}(0)\|^{2}+\nu^{-1}\left\|\boldsymbol{\eta}_{t}\right\|_{L^{2}\left(0,T ; \boldsymbol{H}^{-1}\right)}^{2}\right.\\
&\left.+\nu^{-1} \inf _{q_{h} \in L^{2}\left(0,T ; W_{h}\right)}\left\|p-q_{h}\right\|_{L^{2}\left(0,T ; L^{2}\right)}^{2}\right.\\
&+B(\boldsymbol{u})\left(\|\nabla\boldsymbol{\eta}\|_{L^{4}\left(0,T ; \boldsymbol{L}^{2}\right)}^{2}+\|\boldsymbol{\eta}_{\pi}\|_{L^{4}\left(0,T ; \boldsymbol{L}^{2}\right)}^{2}\right)\Big).
\end{aligned}
\end{equation}
Finally, the estimate \cref{errorestimate1} follows immediately from \cref{errorestimate2} and the triangle inequality.
\end{proof}
\section{The Picard linearization scheme}
\label{sec:3}
In practice, a commonly used linearization way is replacing one of the velocity solutions of the nonlinear term with last time step solutions. In this section we prove that this way preserves all the conservative properties from the semi-discrete version when matching the Crank-Nicolson time discretizations. The linearized Crank-Nicolson scheme is that
\begin{equation}\label{EMfullydiscreteform}
\begin{aligned}
\text{Given} \ \boldsymbol{u}_{h}^{pre},\boldsymbol{u}_{h}^{n}\in V_{h} \text { and } p_{h}^{n}\in W_{h}, \ \text{find} &\ (\boldsymbol{u}_{h}^{n+1},p_{h}^{n+1})\in V_{h}\times W_{h}\  \text{such that }\\
(\frac{\boldsymbol{u}_{h}^{n+1}-\boldsymbol{u}_{h}^{n}}{\Delta t},\boldsymbol{v}_{h})+\nu a(\boldsymbol{u}_{h}^{n+\frac{1}{2}}, \boldsymbol{v}_{h})&+c_{h}(\boldsymbol{u}_{h}^{pre},\boldsymbol{u}_{h}^{n+\frac{1}{2}},\boldsymbol{v}_{h})\\
-b(\boldsymbol{v}_{h}, p_{h}^{n+\frac{1}{2}})&+b(\boldsymbol{u}_{h}^{n+1}, q_{h})=(\boldsymbol{f}^{n+\frac{1}{2}}, \boldsymbol{v}_{h})
\end{aligned}
\end{equation}
for all $(\boldsymbol{v}_{h},q_{h})\in V_{h}\times W_{h}$, where $\boldsymbol{u}_{h}^{n+\frac{1}{2}}=\frac{1}{2}(\boldsymbol{u}_{h}^{n+1}+\boldsymbol{u}_{h}^{n})$ and $p_{h}^{n+\frac{1}{2}}=\frac{1}{2}(p_{h}^{n}+p_{h}^{n+1})$. One could choose $\boldsymbol{u}_{h}^{pre}=\frac{3}{2}\boldsymbol{u}_{h}^{n}-\frac{1}{2}\boldsymbol{u}_{h}^{n-1}$ for the first step of the Picard iteration. This one-step Picard lineatization is also called extrapolated Crank-Nicolson scheme \cite{layton2008} sometimes.
One fully discrete scheme for the vorticity equation \cref{curlNSE} reads:
\begin{equation}\label{curlNSEfullydiscrete}
\begin{aligned}
\left(\frac{\boldsymbol{w}_{h}^{n+1}-\boldsymbol{w}_{h}^{n}}{\Delta t}, \boldsymbol{v}_{h}\right)+&c_{h}(\boldsymbol{u}_{h}^{pre}, \boldsymbol{w}_{h}^{n+\frac{1}{2}}, \boldsymbol{v}_{h})-c_{h}(\boldsymbol{w}_{h}^{n+\frac{1}{2}}, \boldsymbol{u}_{h}^{n+\frac{1}{2}}, \boldsymbol{v}_{h})+\\ \nu(\nabla \boldsymbol{w}_{h}^{n+\frac{1}{2}}, \nabla \boldsymbol{v}_{h})-&b( \boldsymbol{v}_{h},\eta_{h}^{n+\frac{1}{2}})+b(\boldsymbol{w}_{h}^{n+1},q_{h})=(\nabla\times \boldsymbol{f}^{n+\frac{1}{2}}, \boldsymbol{v}_{h}).
\end{aligned}
\end{equation}
\begin{theorem}
In the case $\nu=0$ and $\boldsymbol{f}=\boldsymbol{0}$, the linearized method \cref{EMfullydiscreteform} conserves the kinetic energy under \cref{assum1} \ding{192}, that is,
\begin{equation}\label{energyconser1}
||\boldsymbol{u}_{h}^{n}||^{2}=\|\boldsymbol{u}_{h}({0})\|^{2} \quad \text{for all} ~n.
\end{equation}
\end{theorem}
\begin{proof}
Taking $\boldsymbol{v}_{h}=\boldsymbol{u}_{h}^{n+1/2}$ in \cref{EMfullydiscreteform} and considering the case $\nu=0$ and $\boldsymbol{f}=\boldsymbol{0}$, we get
\begin{equation}\label{discreteenergy}
\begin{aligned}
\frac{1}{2\Delta t}\left(\|\boldsymbol{u}_{h}^{n+1}\|^{2}-\|\boldsymbol{u}_{h}^{n}\|^{2}\right)&+
c_{h}(\boldsymbol{u}_{h}^{pre},\boldsymbol{u}_{h}^{n+\frac{1}{2}},\boldsymbol{u}_{h}^{n+\frac{1}{2}})\\
&=\frac{1}{2\Delta t}\left(\|\boldsymbol{u}_{h}^{n+1}\|^{2}-\|\boldsymbol{u}_{h}^{n}\|^{2}\right)=0,
\end{aligned}
\end{equation}
which implies that
\begin{displaymath}
\|\boldsymbol{u}_{h}^{n+1}\|^{2}=\|\boldsymbol{u}_{h}^{n}\|^{2}\quad \text{for all } n.
\end{displaymath}
Then \cref{energyconser1} follows.
\end{proof}
\begin{theorem}\label{mhevconser2}
Under \cref{assum1} \ding{192} and the fully discrete version of \cref{compactsupport}, the methods \cref{EMfullydiscreteform}-\cref{curlNSEfullydiscrete} conserve momentum (for $\boldsymbol{f}$ with zero linear momentum), helicity (for $\nu=0$ and $\boldsymbol{f}=\boldsymbol{0}$), 2D enstrophy (for $\nu=0$ and $\boldsymbol{f}=\boldsymbol{0}$), and total vorticity.
\end{theorem}
\begin{proof}
The conservation of momentum, 2D enstrophy and vorticity follow immediately from \cref{essentiallemma} and respectively taking $\boldsymbol{v}_{h}=\chi(\boldsymbol{e}_{i})$, $v_{h}=w_{h}^{{n+\frac{1}{2}}}$ and $\boldsymbol{v}_{h}=\chi(\boldsymbol{e}_{i})$ in \cref{EMfullydiscreteform} or \cref{curlNSEfullydiscrete}. We mainly prove the conservation of helicity.

Suppose that $\nu=0$ and $\boldsymbol{f}=\boldsymbol{0}$. Testing with $(\boldsymbol{v}_{h},q_{h})=(\boldsymbol{w}_{h}^{n+\frac{1}{2}},0)$ in \cref{EMfullydiscreteform} gives
\begin{equation}\label{Fhelicity1}
\left(\frac{\boldsymbol{u}_{h}^{n+1}-\boldsymbol{u}_{h}^{n}}{\Delta t}, \boldsymbol{w}_{h}^{n+\frac{1}{2}}\right)+c_{h}\left(\boldsymbol{u}_{h}^{pre},\boldsymbol{u}_{h}^{n+\frac{1}{2}}, \boldsymbol{w}_{h}^{n+\frac{1}{2}}\right)=0.
\end{equation}
Meanwhile, testing with $(\boldsymbol{v}_{h},q_{h})=(\boldsymbol{u}_{h}^{n+\frac{1}{2}},0)$ in \cref{curlNSEfullydiscrete} gives
\begin{equation}\label{Fhelicity2}
\left(\frac{\boldsymbol{w}_{h}^{n+1}-\boldsymbol{w}_{h}^{n}}{\Delta t}, \boldsymbol{u}_{h}^{n+\frac{1}{2}}\right)+c_{h}(\boldsymbol{u}_{h}^{pre}, \boldsymbol{w}_{h}^{n+\frac{1}{2}}, \boldsymbol{u}_{h}^{n+\frac{1}{2}})-c_{h}(\boldsymbol{w}_{h}^{n+\frac{1}{2}}, \boldsymbol{u}_{h}^{n+\frac{1}{2}}, \boldsymbol{u}_{h}^{n+\frac{1}{2}})=0.
\end{equation}
Summing \cref{Fhelicity1} and \cref{Fhelicity2} and using \cref{essentialformu1}, one arrives at
\begin{equation}
\frac{1}{\Delta t}\left\{\left(\boldsymbol{u}_{h}^{n+1},\boldsymbol{w}_{h}^{n+1}\right)-\left(\boldsymbol{u}_{h}^{n},\boldsymbol{w}_{h}^{n}\right)\right\}=0
\end{equation}
by some simple calculations. Thus the proof is completed.
\end{proof}
\section{Numerical experiments}
\label{sec:4}
\subsection{The choice of the divergence-free reconstruction operator}
For all examples below, the modified convective scheme (``MOD CONV" in legends below) is computed by one-step Picard iteration. We test two commonly used simplicial elements: one is the second order Taylor-Hood element ($P_{2}/P_{1}$), the other is a second order locally divergence-free element, $P_{2}^{bubble}/P_{1}^{disc}$ \cite[pp. 139-144]{girault_finite_1986}. The Taylor-Hood element has continuous pressure and thus is only globally mass conservative. To obtain an exactly divergence-free approximation of the Taylor-Hood velocity, we first seek a locally divergence-free approximation by projecting (under $L^{2}$ or Stokes sense) the velocity to the discretely divergence-free Bernardi-Raugel subspace \cite{bernardi_analysis_1985}. This is equal to solving a Darcy or Stokes system. Then for Bernardi-Raugel pair and $P_{2}^{bubble}/P_{1}^{disc}$ we apply the reconstruction operators in \cite[Remark 4.2]{linke_pressure-robustness_2016} to get an exactly divergence-free approximation. A class of much cheaper reconstruction operators for Taylor-Hood elements can be found in \cite{Gmeiner2014,Linke2017}.

We also show some results by SKEW, CONV and EMAC. For Taylor-Hood element, to do a fair comparison, the SKEW, CONV and EMAC are linearized by two-step Picard or other (for EMAC) iterations, since the reconstruction  requires solving an extra global system in modified CONV. For $P_{2}^{bubble}/P_{1}^{disc}$ pair, all formulations are computed by one-step linearization. The linearized EMAC scheme used here can be found in \cite{EMAC2019}, which is a skew-symmetric (energy-conserving) scheme but does not conserve the momentum and angular momentum. This scheme sometimes is unable to show the advantages of EMAC. For this reason, we also give some results from the nonlinear EMAC scheme (solved by Newton iterations) as a reference.
\subsection{Example 1: Gresho problem}
For the first example we study the Gresho problem \cite{Gresho1990,Rebholz2017}, which is a Euler equation ($\nu=0$) with zero external force ($\boldsymbol{f}=\boldsymbol{0}$). The exact velocity solution is constant-in-time and prescribed as
\begin{subequations}\nonumber
\begin{align}
r &\leq 0.2:
\boldsymbol{u}=(-5 y ,5 x)^{\top}, \\
0.2 < r &\leq 0.4:
\boldsymbol{u}=(-\frac{2 y}{r}+5 y,
\frac{2 x}{r}-5 x)^{\top}, \\
 r & > 0.4:
\boldsymbol{u}=
(0,
0)^{\top}
\end{align}
\end{subequations}
where $r=\sqrt{x^{2}+y^{2}}$, $\Omega=(-0.5,0.5)^{2}$.
We choose the exact solution as the initial condition and apply the no-penetration boundary condition, with $T=10$. \cref{Greshofig3} is a speed contour plot of the exact solution. An accurate scheme should preserve the structure in a long time. This problem is a good benchmark to test the conservative properties of our schemes since it almost satisfies all the assumptions in our analysis. The numerical test is built on the uniform $48\times48$ triangular mesh. For time discretizations, we apply the Crank-Nicolson scheme with a time step $\Delta t=0.01$. Some results can be found in \cref{Greshofig1} and \cref{Greshofig2}. Here the ``momentum" represents the sum of the components of momentum. The convective formulation blows up very quickly after $t=2$ and $t=3$ for $P_{2}/P_{1}$ and $P_{2}^{bubble}/P_{1}^{disc}$, respectively, while the modified convective formulation stays stable up to $t=10$. The speed contour plots of the discrete velocity for $P_{2}/P_{1}$ ($t=2$) and $P_{2}^{bubble}/P_{1}^{disc}$ ($t=3$) are shown in \cref{Greshofig4} and \cref{Greshofig5}, respectively.

In all cases shown below, the modified convective formulation gives the best performance. It is worth mentioning that, although we have analytically proven that modified CONV does not conserve angular momentum, its numerical performance is even as good as EMAC on this aspect. Further, the sub-figures in \cref{Greshofig4} and \cref{Greshofig5} demonstrate that the modified convective formulation does preserve the structure best. Although energy-conserving, SKEW and skew-linearized EMAC give a low accuracy and do not preserve angular momentum.
\begin{figure}[htbp]
\centering
\includegraphics[width=0.48\textwidth]{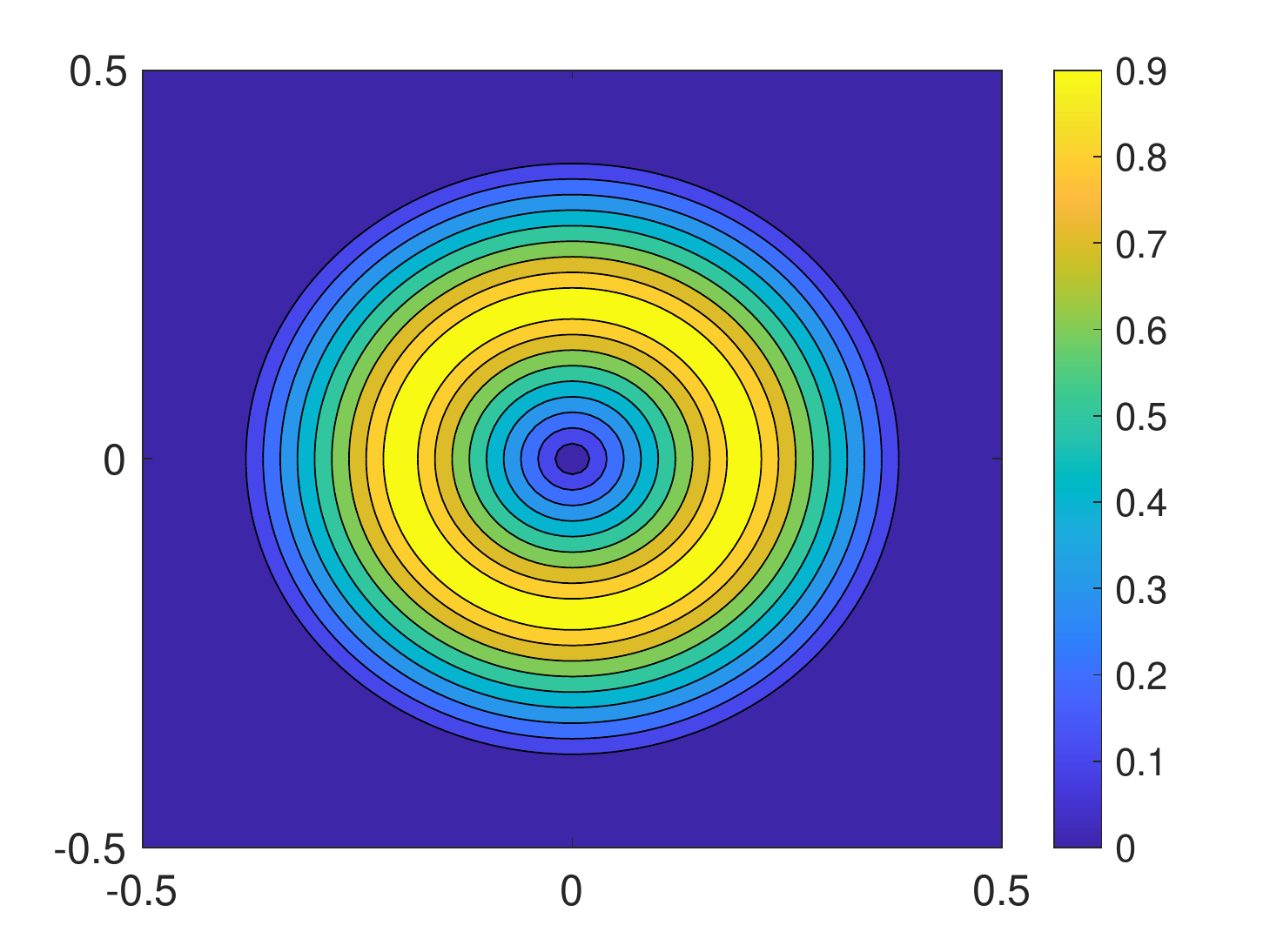}
\caption{Example 1. Plot of exact speed contour.}
\label{Greshofig3}
\end{figure}
\begin{figure}[htbp]
\centering
\includegraphics[width=0.48\textwidth]{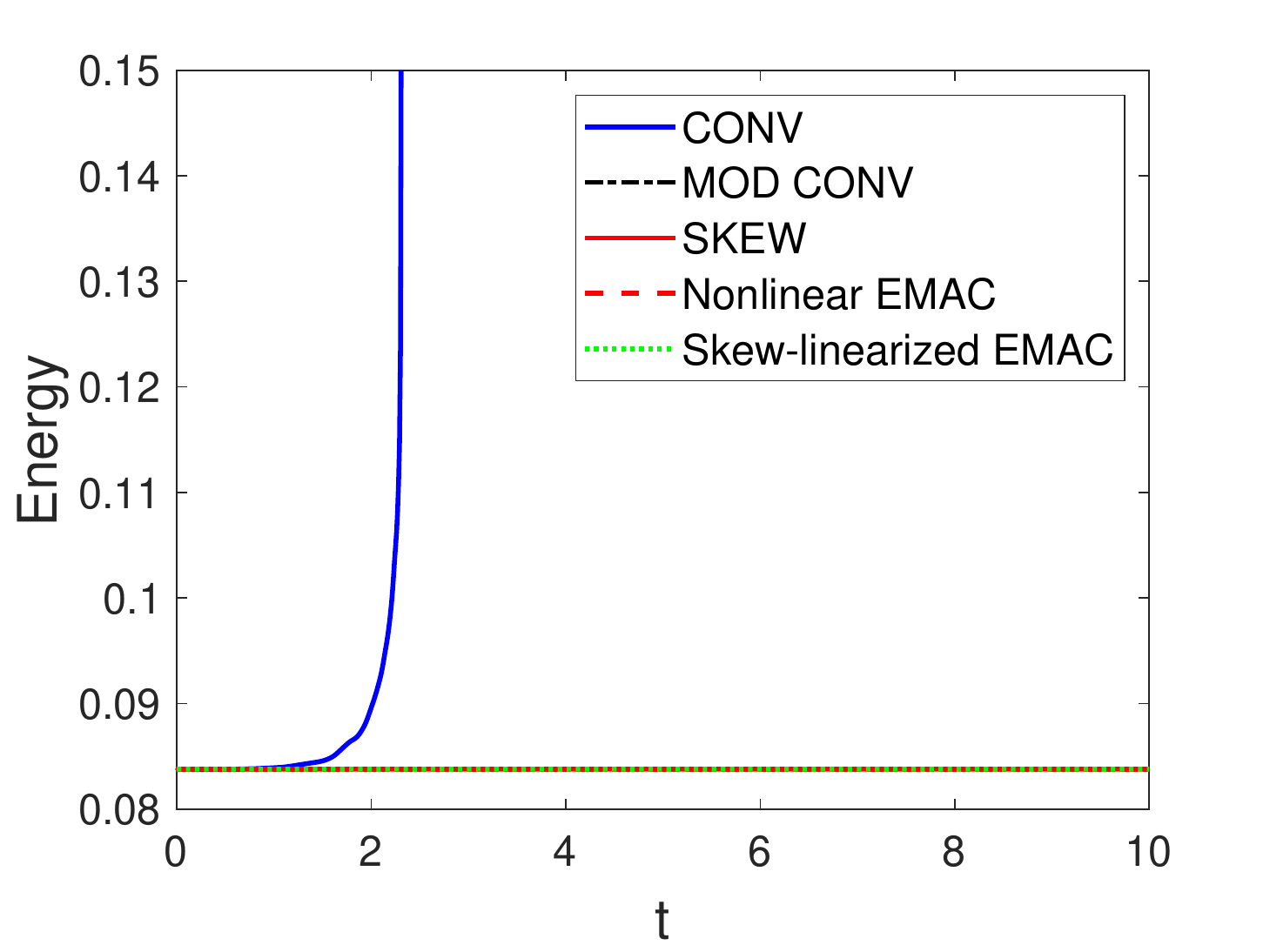}
\includegraphics[width=0.48\textwidth]{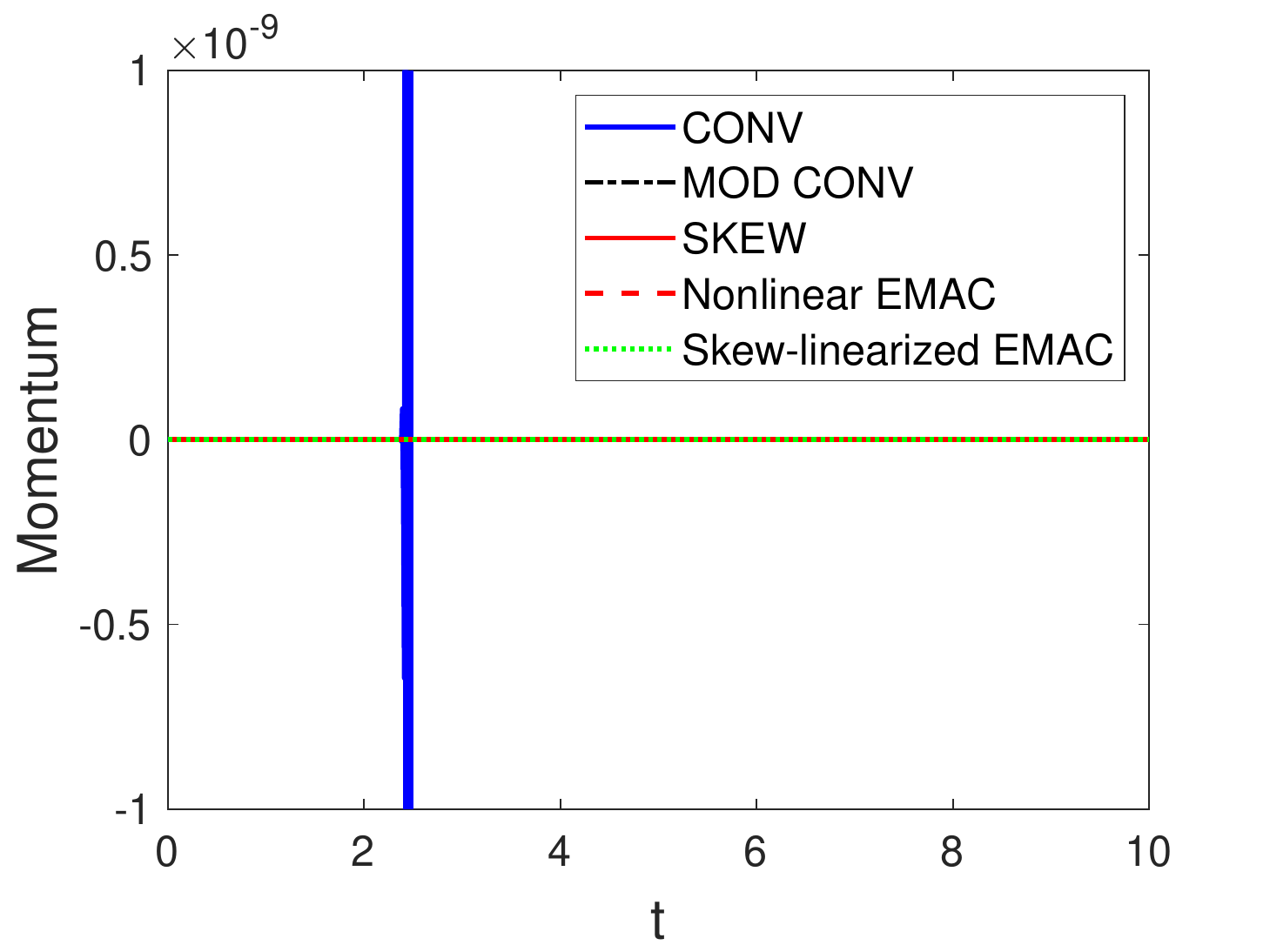}
\includegraphics[width=0.48\textwidth]{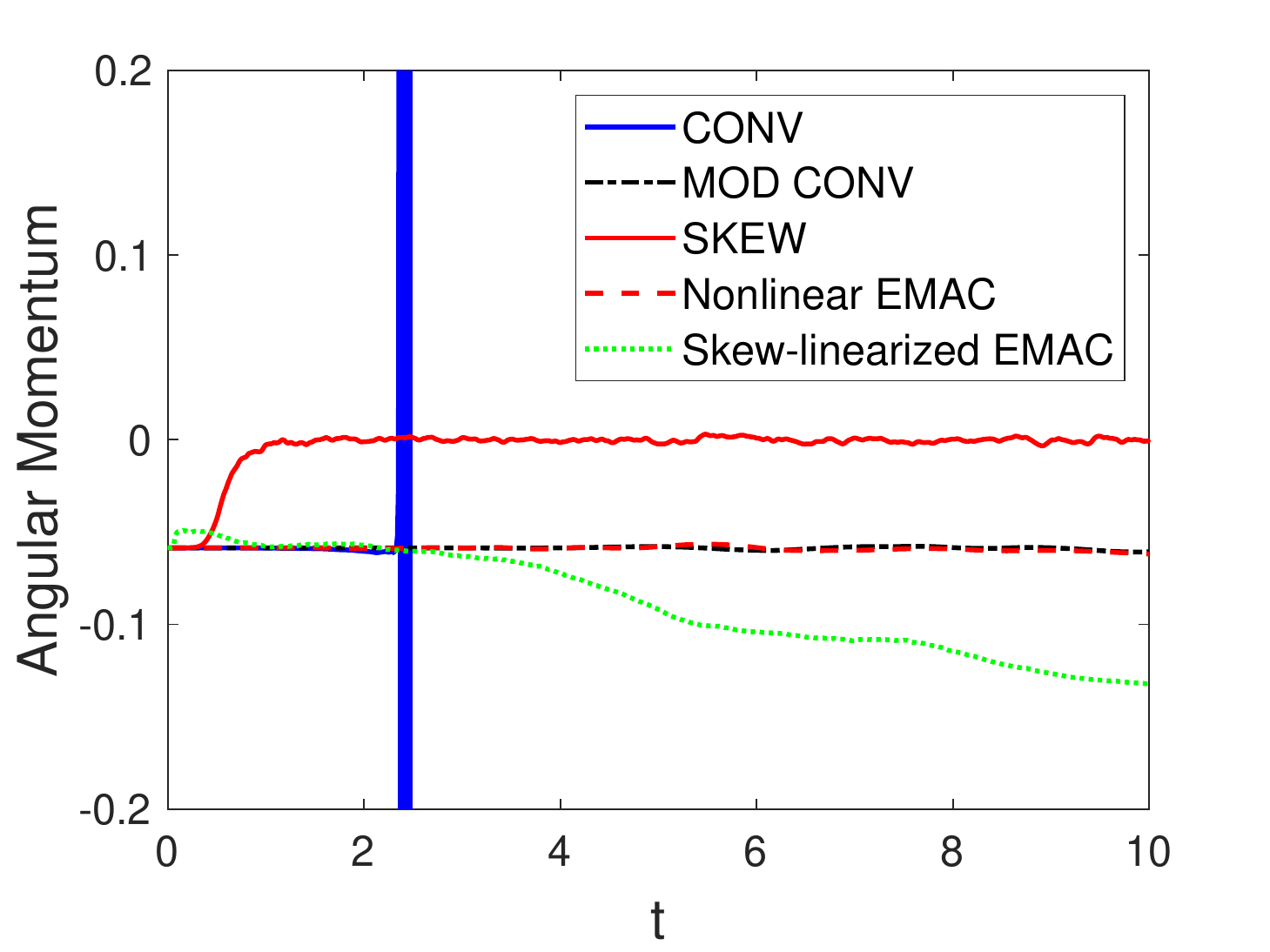}
\includegraphics[width=0.48\textwidth]{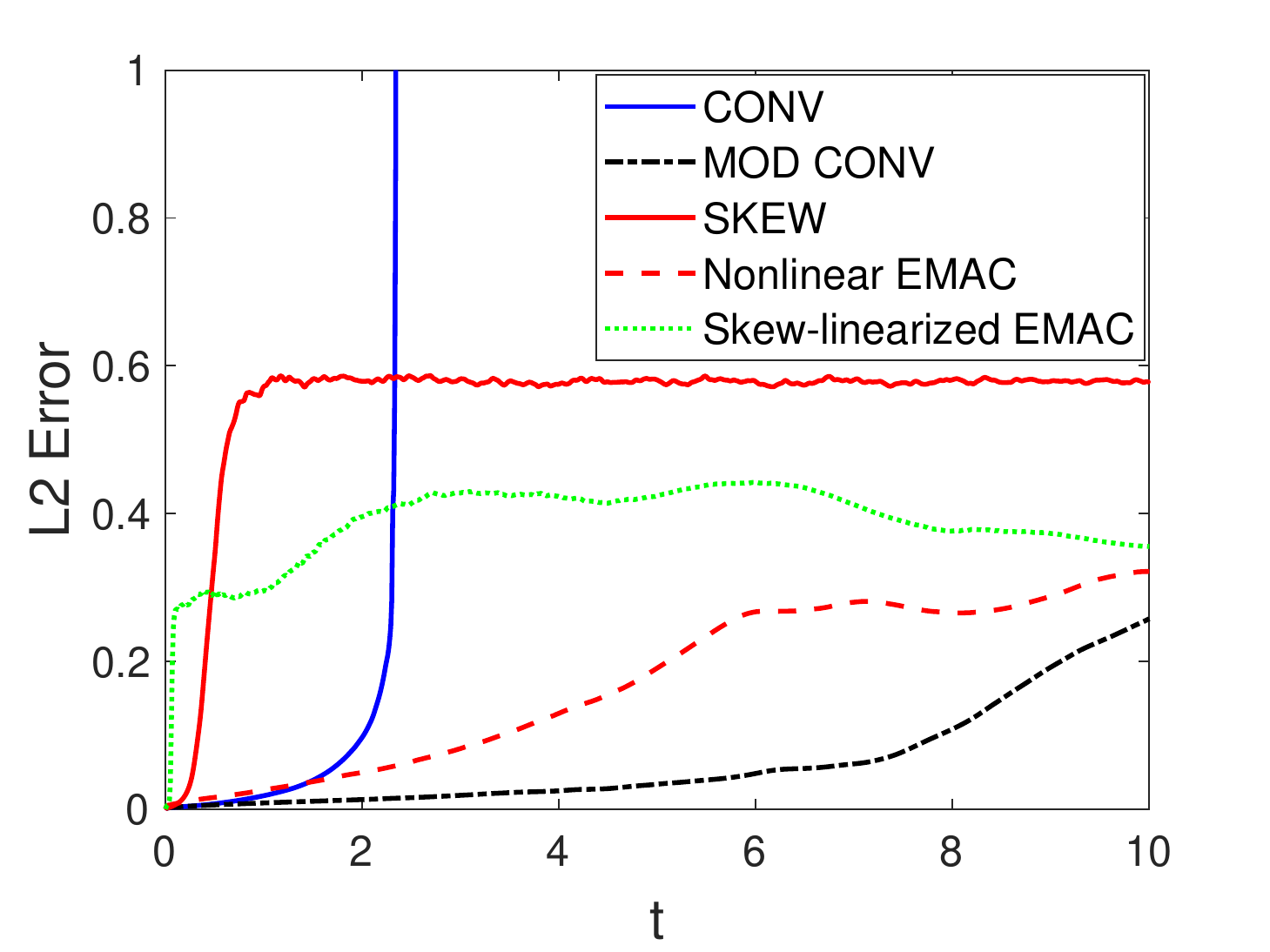}
\caption{Example 1. Plots of the energy, momentum, angular momentum and L2 errors by $P_{2}/P_{1}$ versus time.}
\label{Greshofig1}
\end{figure}
\begin{figure}[htbp]
\centering
\includegraphics[width=0.48\textwidth]{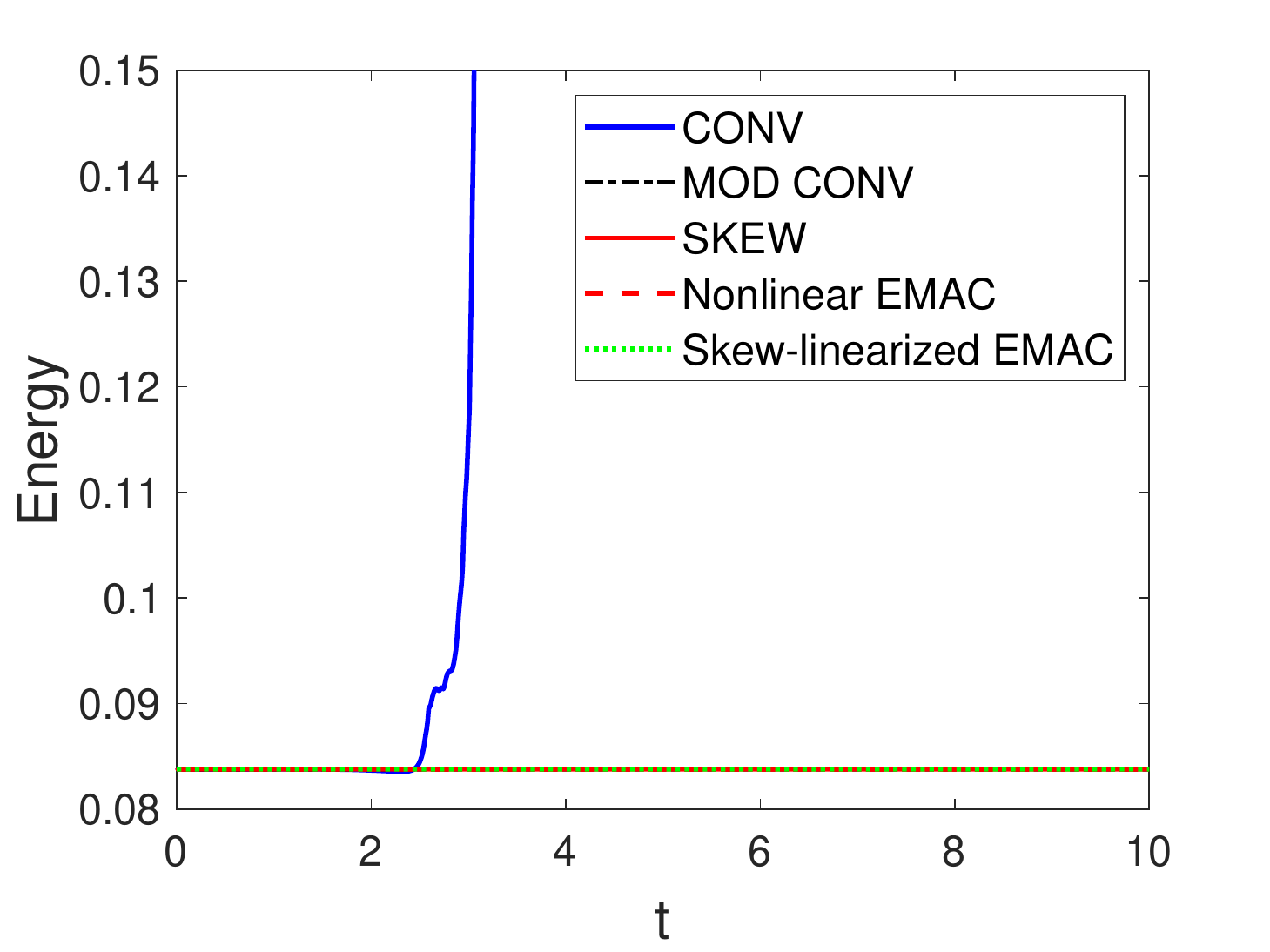}
\includegraphics[width=0.48\textwidth]{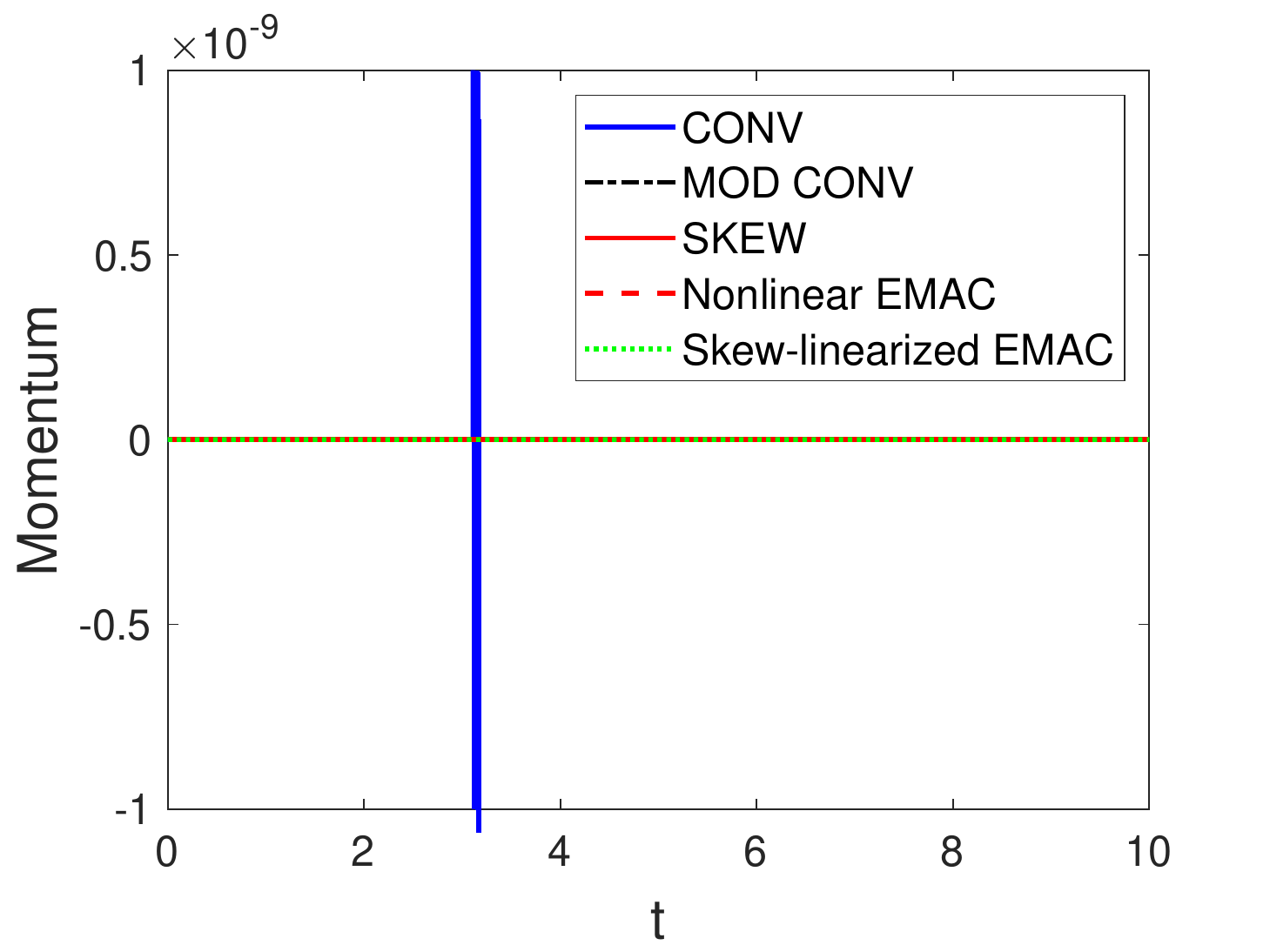}
\includegraphics[width=0.48\textwidth]{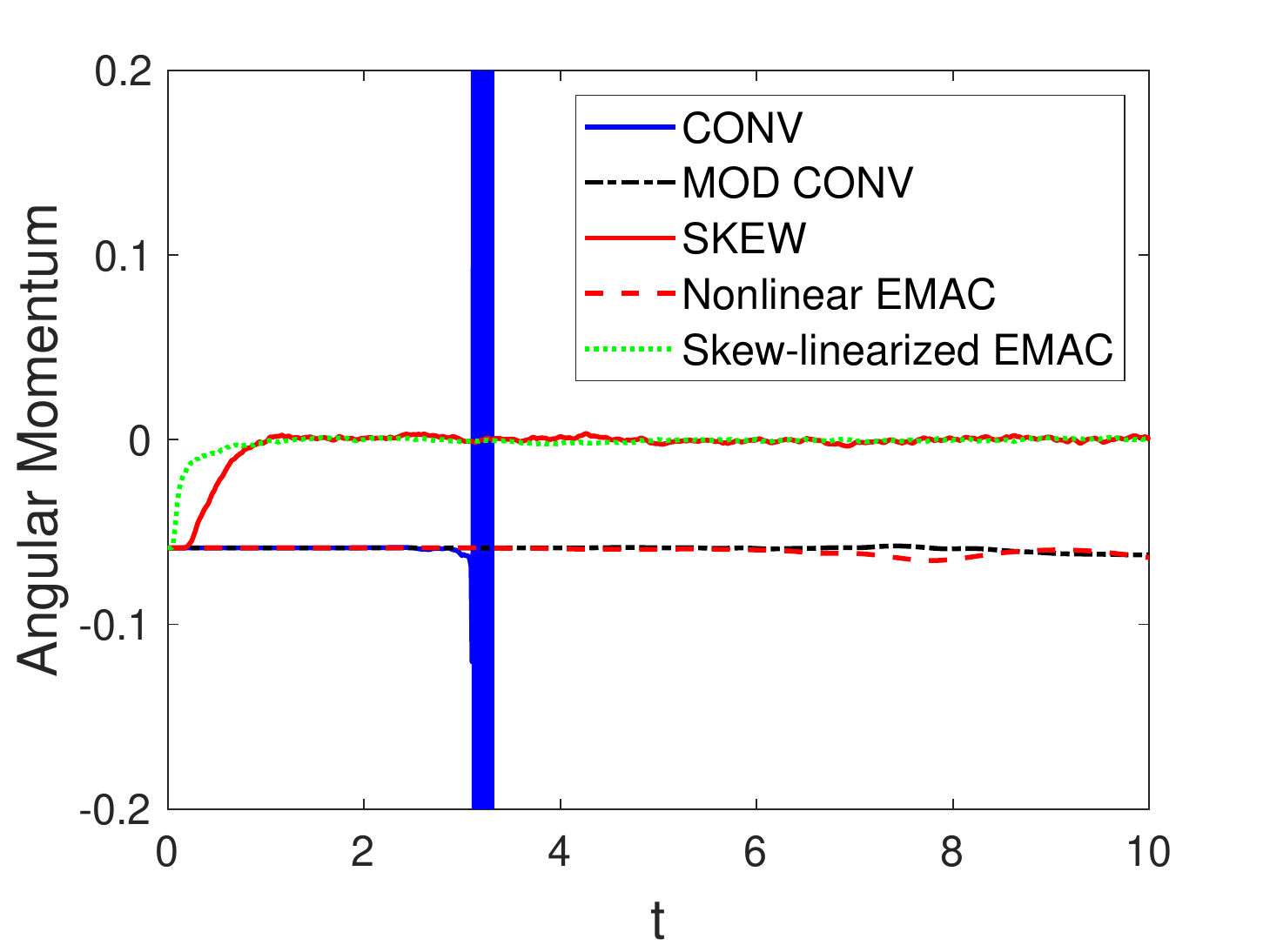}
\includegraphics[width=0.48\textwidth]{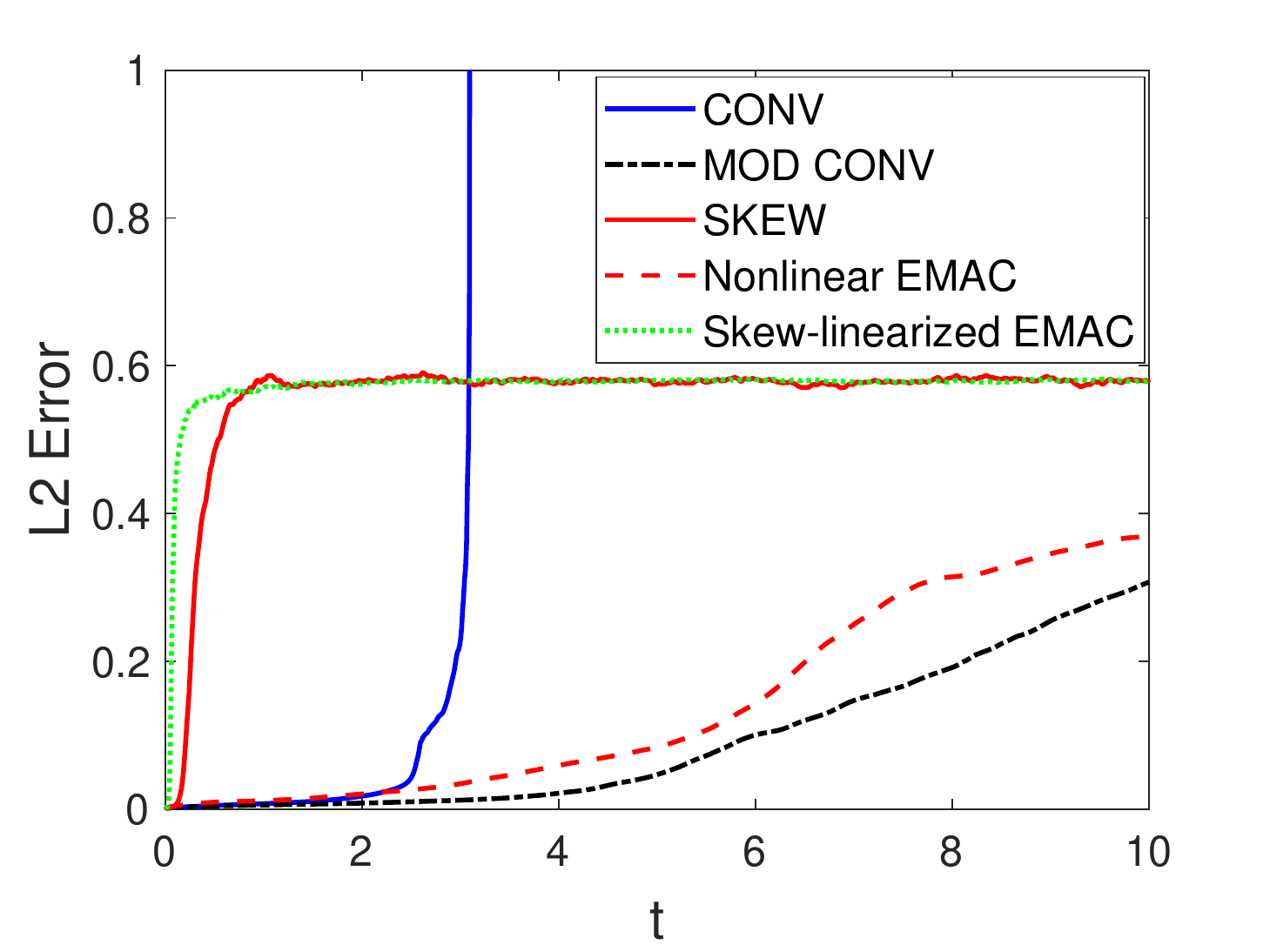}
\caption{Example 1. Plots of the energy, momentum, angular momentum and L2 errors by $P_{2}^{bubble}/P_{1}^{disc}$ versus time.}
\label{Greshofig2}
\end{figure}
\begin{figure}[htbp]
\centering
\includegraphics[width=0.30\textwidth]{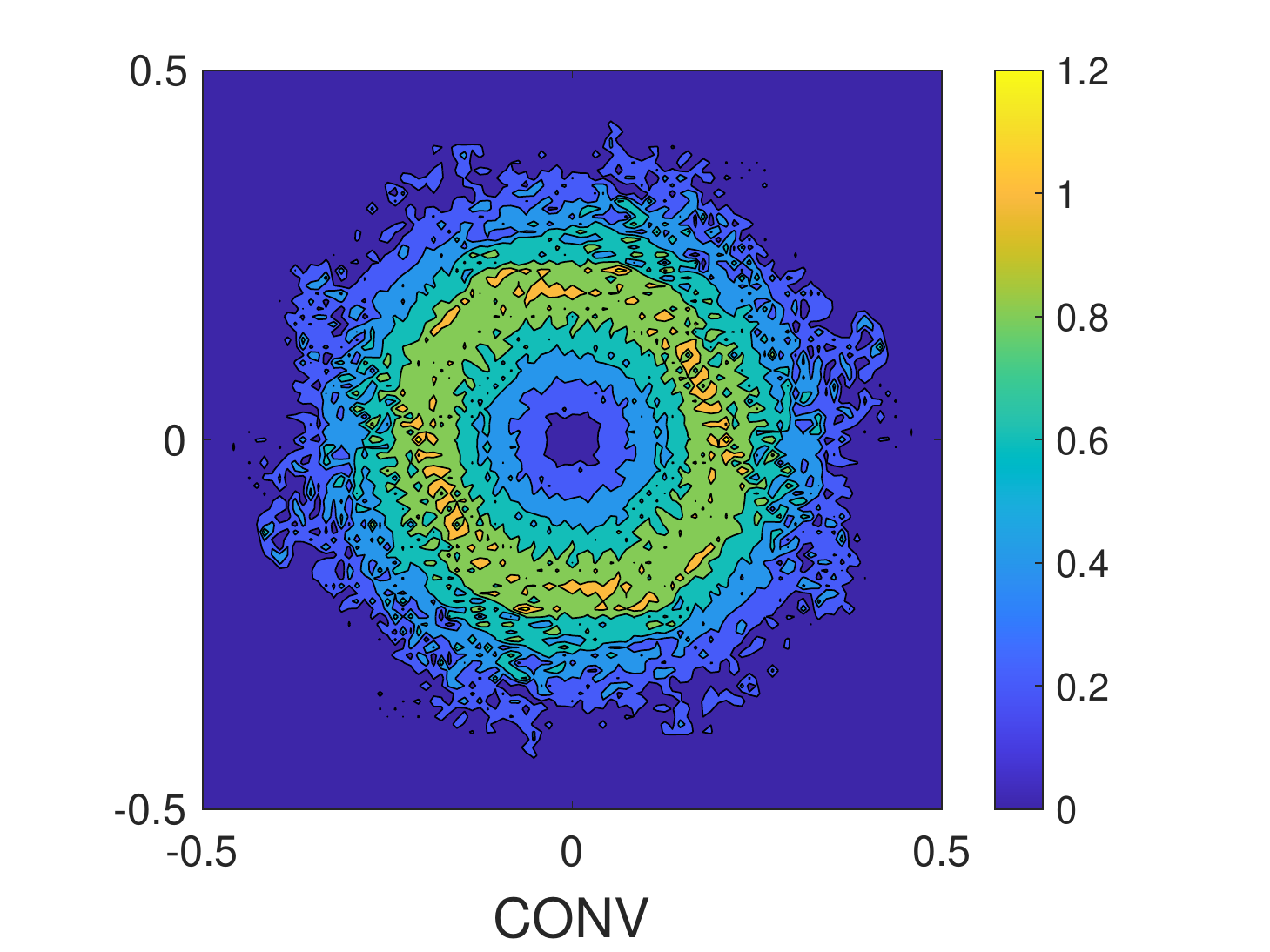}
\includegraphics[width=0.30\textwidth]{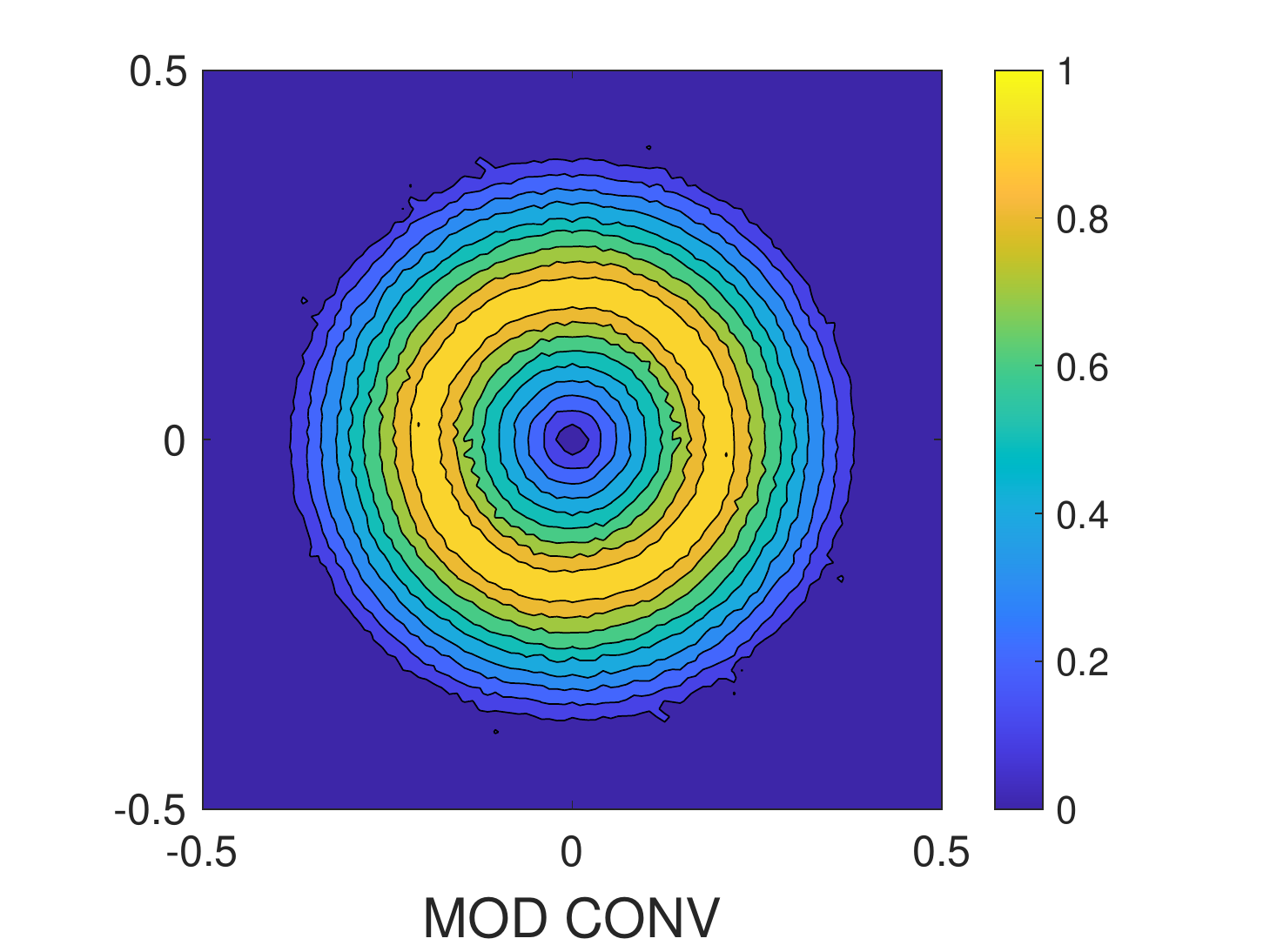}
\includegraphics[width=0.30\textwidth]{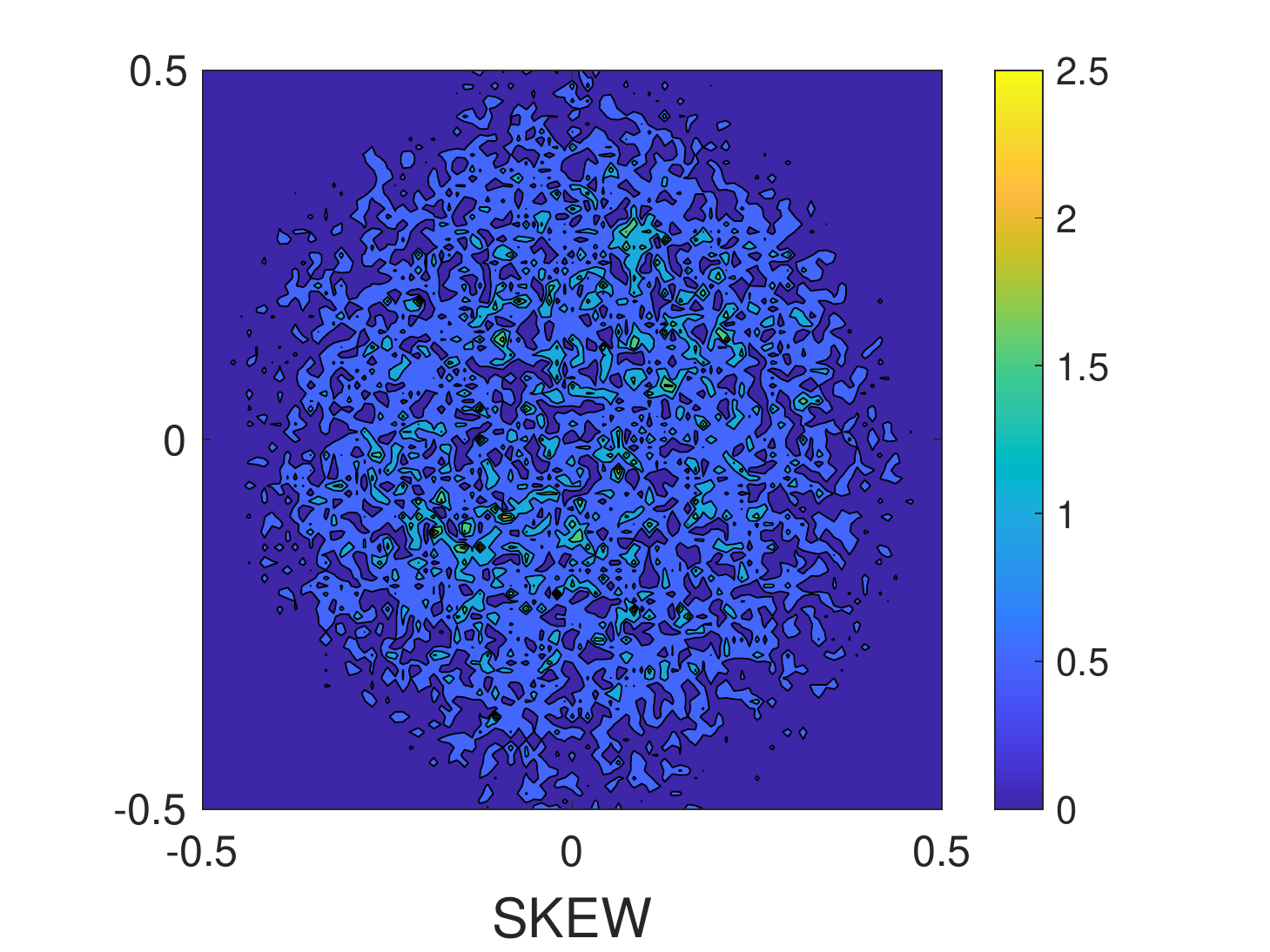}\\
\includegraphics[width=0.30\textwidth]{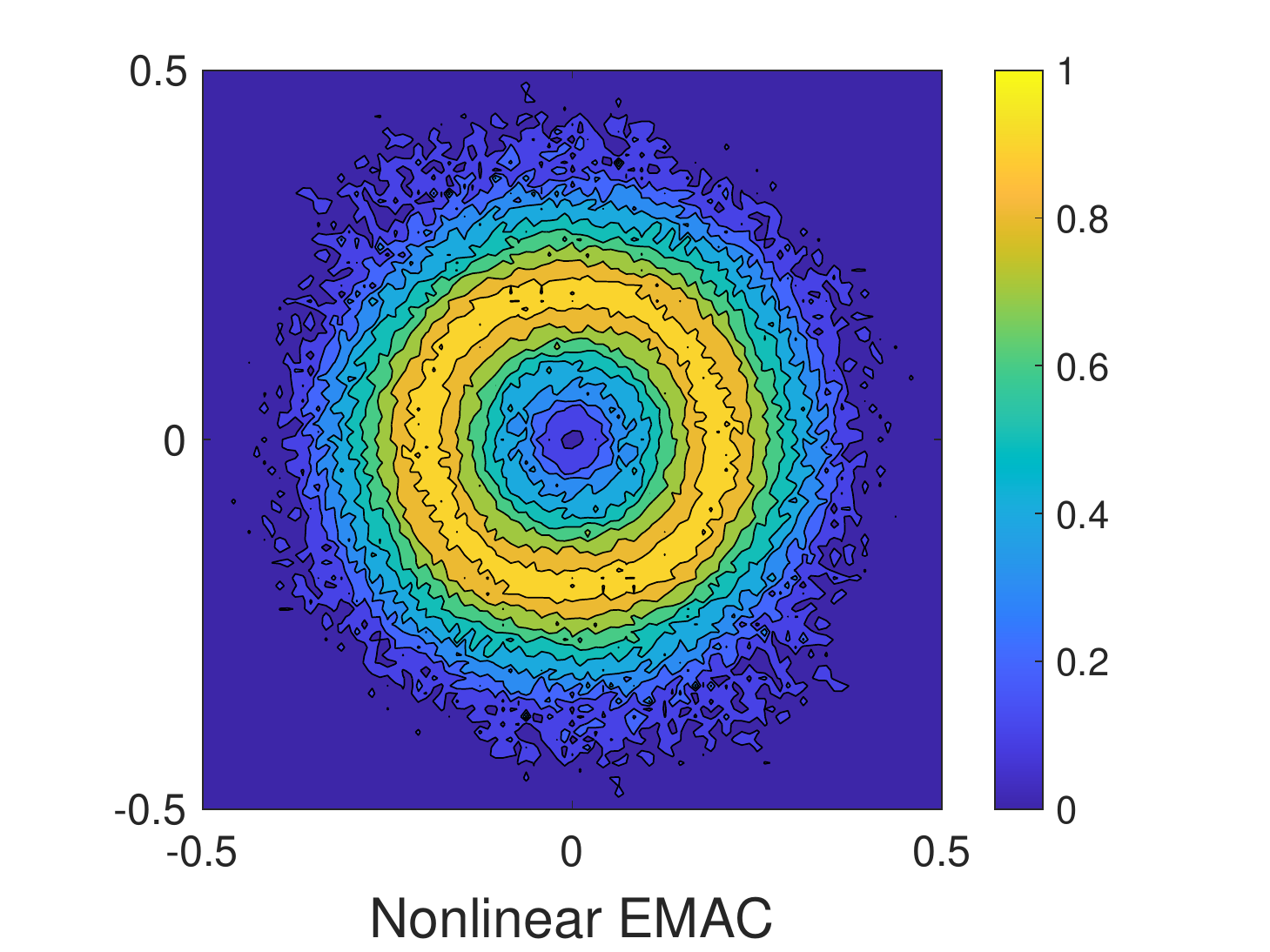}
\includegraphics[width=0.30\textwidth]{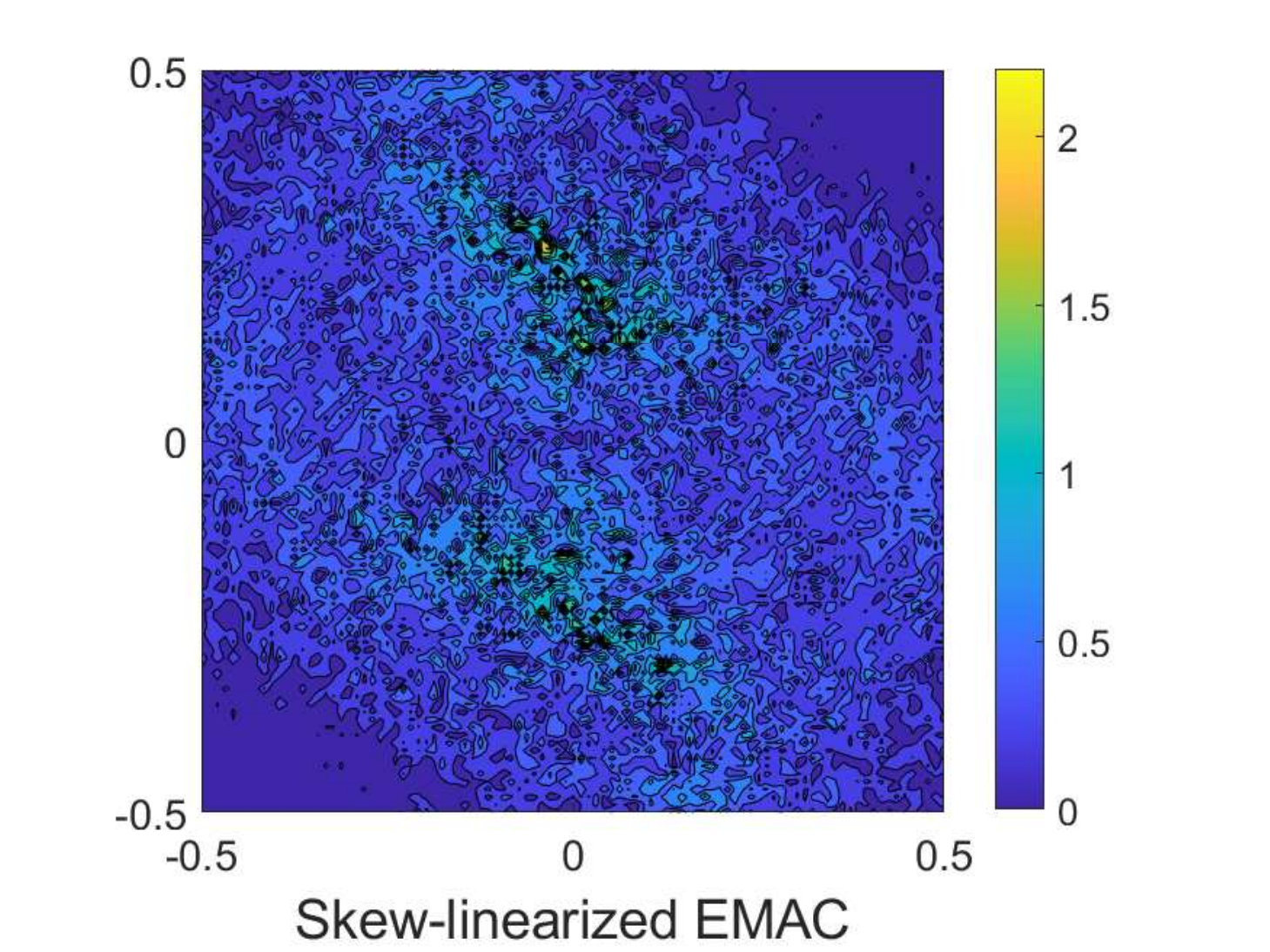}
\caption{Example 1. Speed contour at $t=2$ by $P_{2}/P_{1}$.}
\label{Greshofig4}
\end{figure}
\begin{figure}[htbp]
\centering
\includegraphics[width=0.30\textwidth]{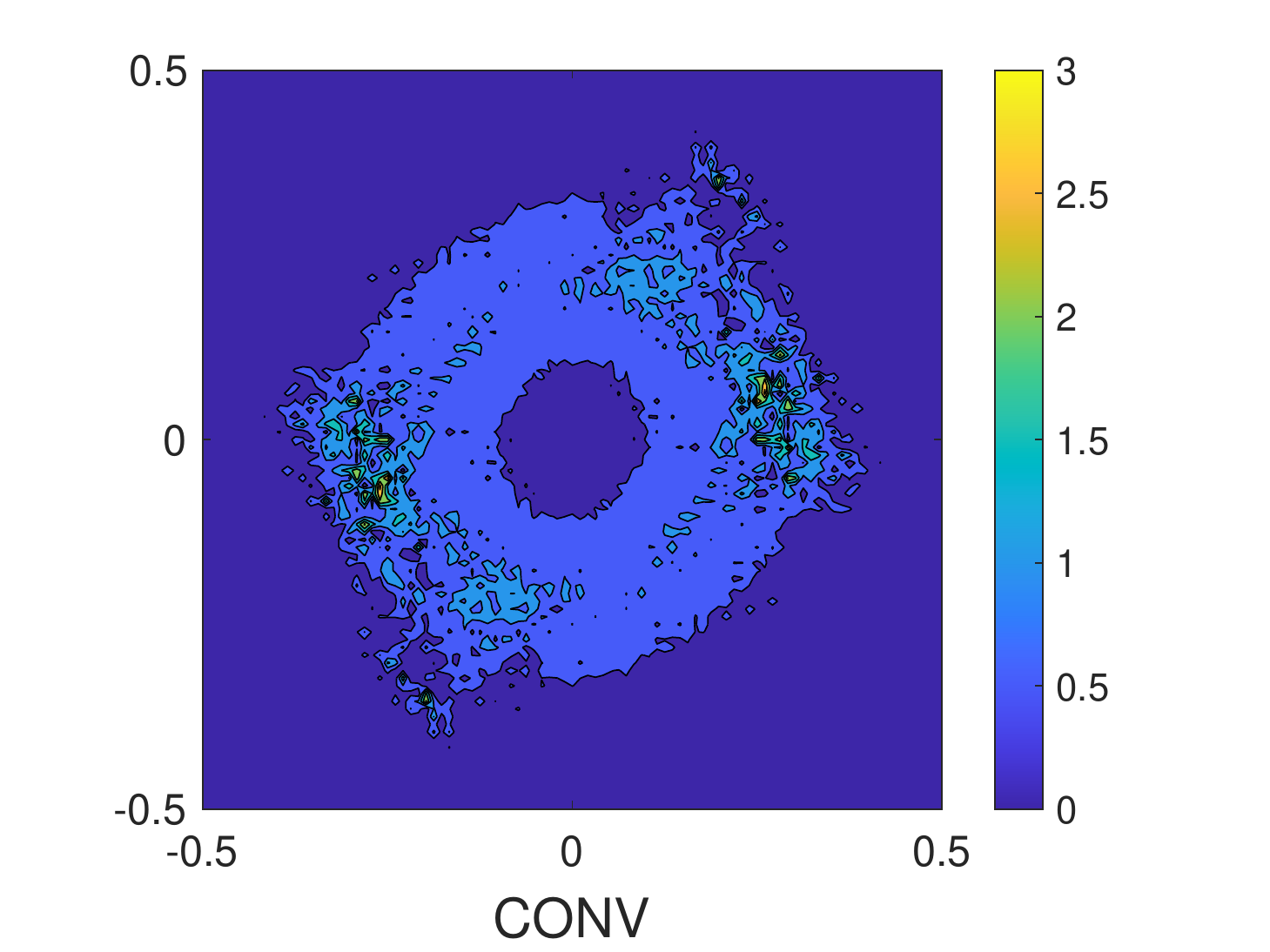}
\includegraphics[width=0.30\textwidth]{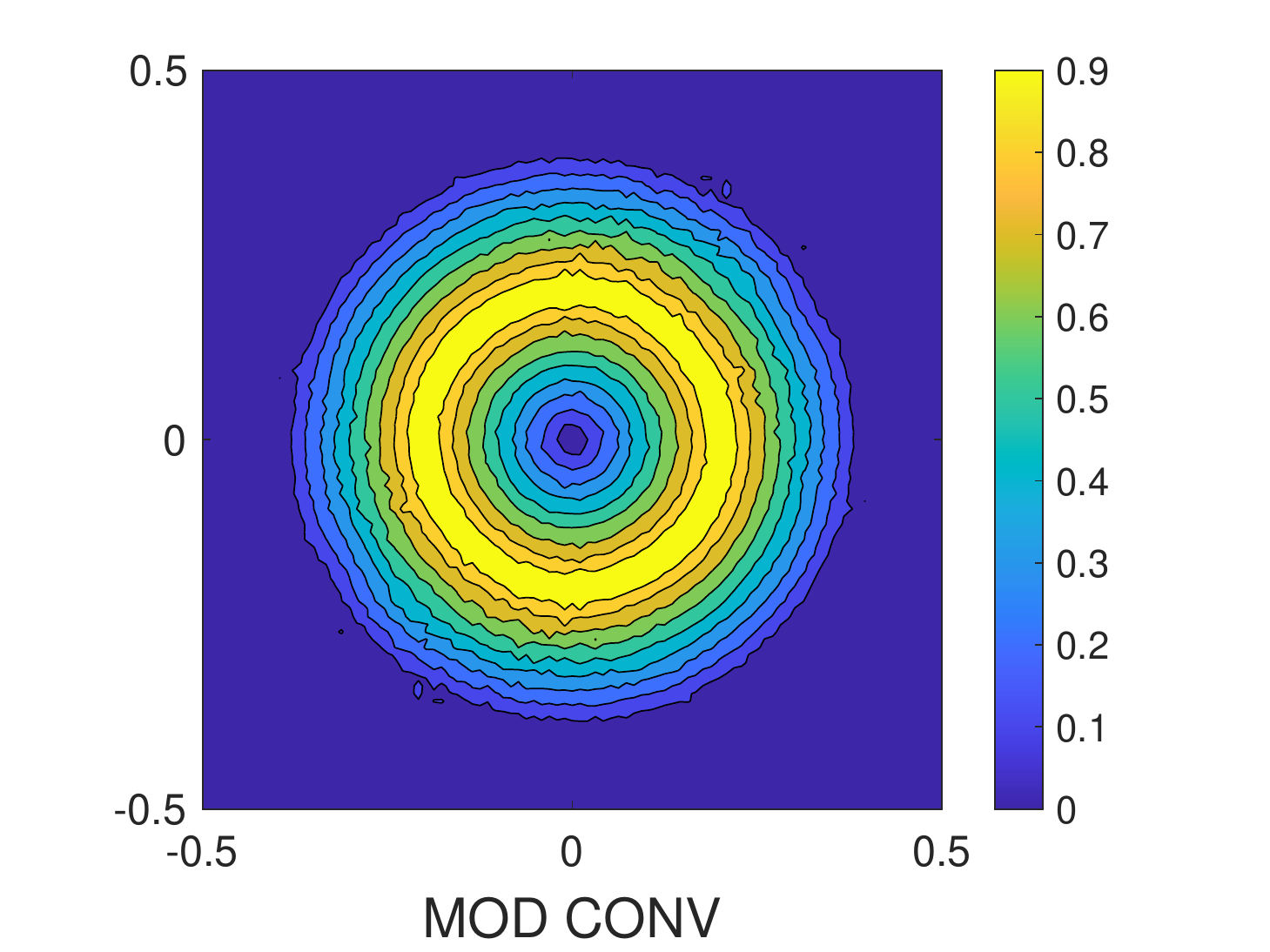}
\includegraphics[width=0.30\textwidth]{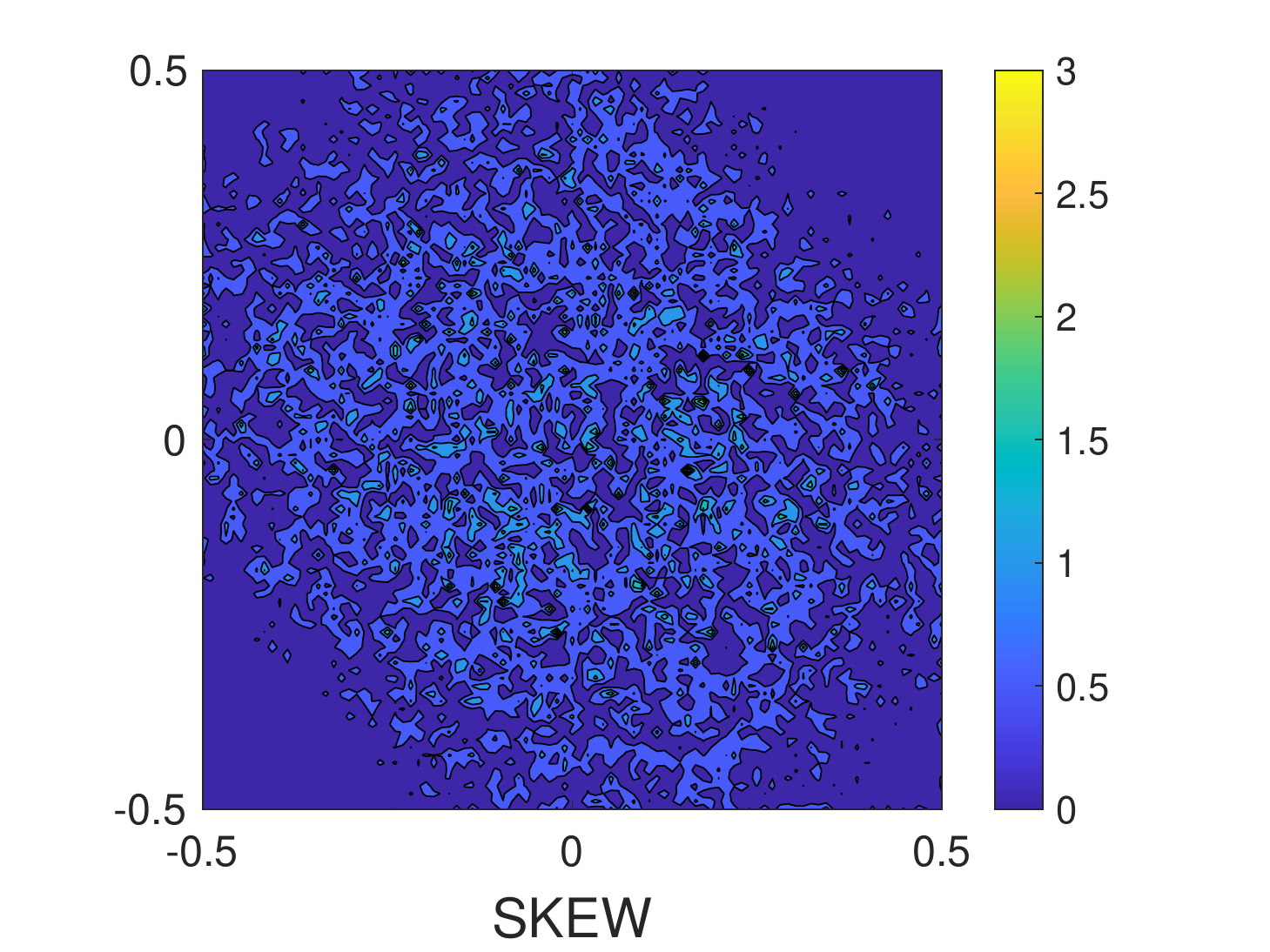}\\
\includegraphics[width=0.30\textwidth]{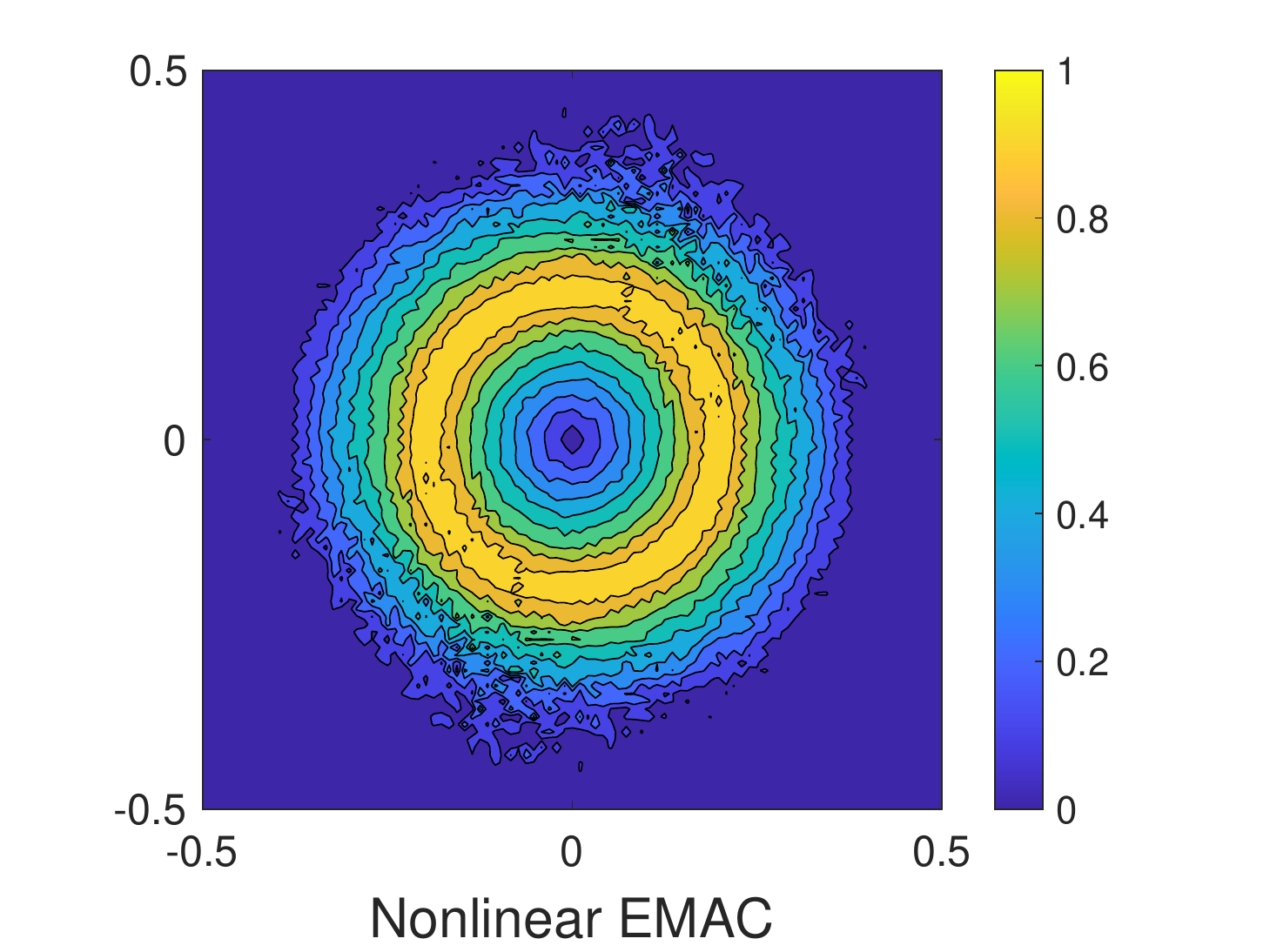}
\includegraphics[width=0.30\textwidth]{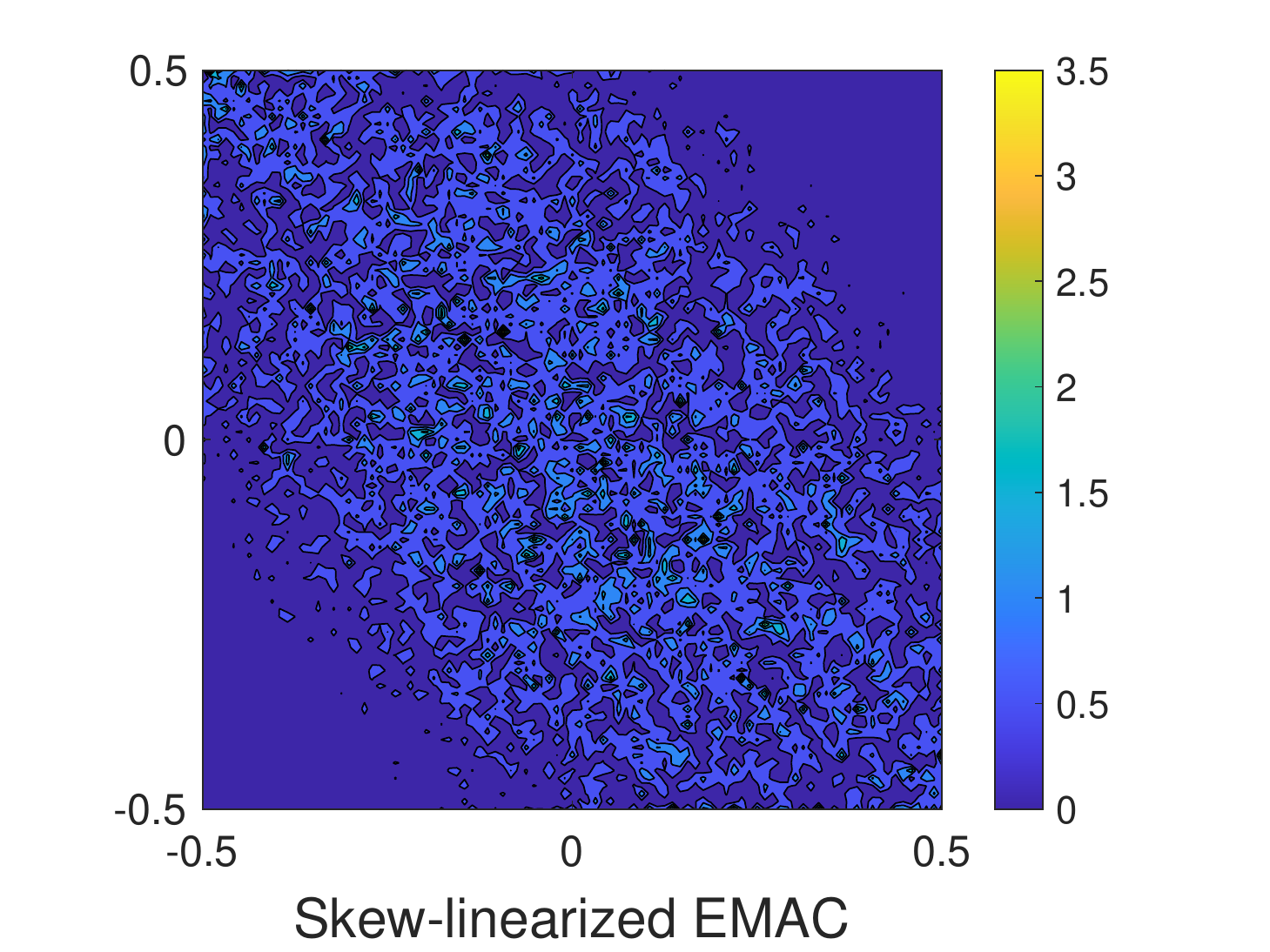}
\caption{Example 1. Speed contour at $t=3$ by $P_{2}^{bubble}/P_{1}^{disc}$.}
\label{Greshofig5}
\end{figure}
\subsection{Example 2: the lattice vortex problem}
In the second test we study the performance on the lattice vortex flow. This problem is a Navier-Stokes equation with zero external force, $\boldsymbol{f}\equiv0$. The exact velocity solution is
\begin{equation}\nonumber
\boldsymbol{u}=\boldsymbol{u}^{0} e^{-8 \nu \pi^{2} t},
\end{equation}
with $\boldsymbol{u}^{0}=\left(\sin \left(2 \pi x\right) \sin \left(2 \pi y\right), \cos \left(2 \pi x\right) \cos \left(2 \pi y\right)\right)^{\top}$ (see \cref{latticeinitial}). We set $\nu=10^{-5}$ and $T=10$, which simulates a high Reynolds number case and longer time evolution. The test is run on the uniform $64\times64$ triangulation of the domain $\Omega=(0,1)^{2}$ and the Dirichlet velocity boundary condition is strongly enforced. The Crank-Nicolson difference is used to discretize the time with the step $\Delta t=0.001$. In a sense, this problem is even more difficult than the first problem because its flow structure is {\it `dynamically unstable so that small perturbations result in a very chaotic motion'} \cite{2001Vorticity}. In the papers \cite{schroeder_towards_2018,Rebholz2020}, this problem was used to show the independence of the Reynolds number in the Gronwall constant for exactly divergence-free elements and the EMAC formulation, respectively. Here we show that our method possesses this advantage also.
\begin{figure}[htbp]
\centering
\includegraphics[width=0.48\textwidth]{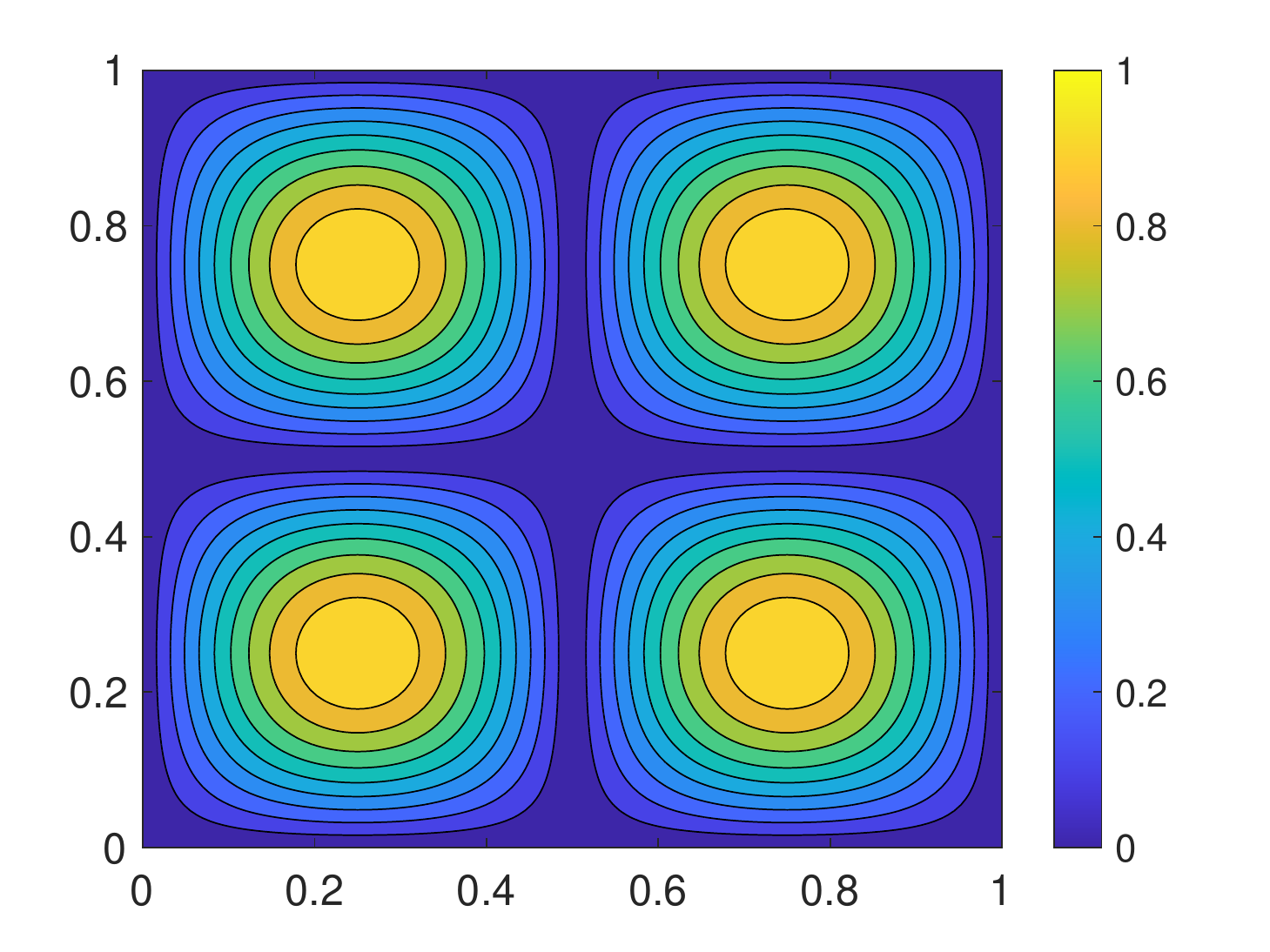}
\includegraphics[width=0.48\textwidth]{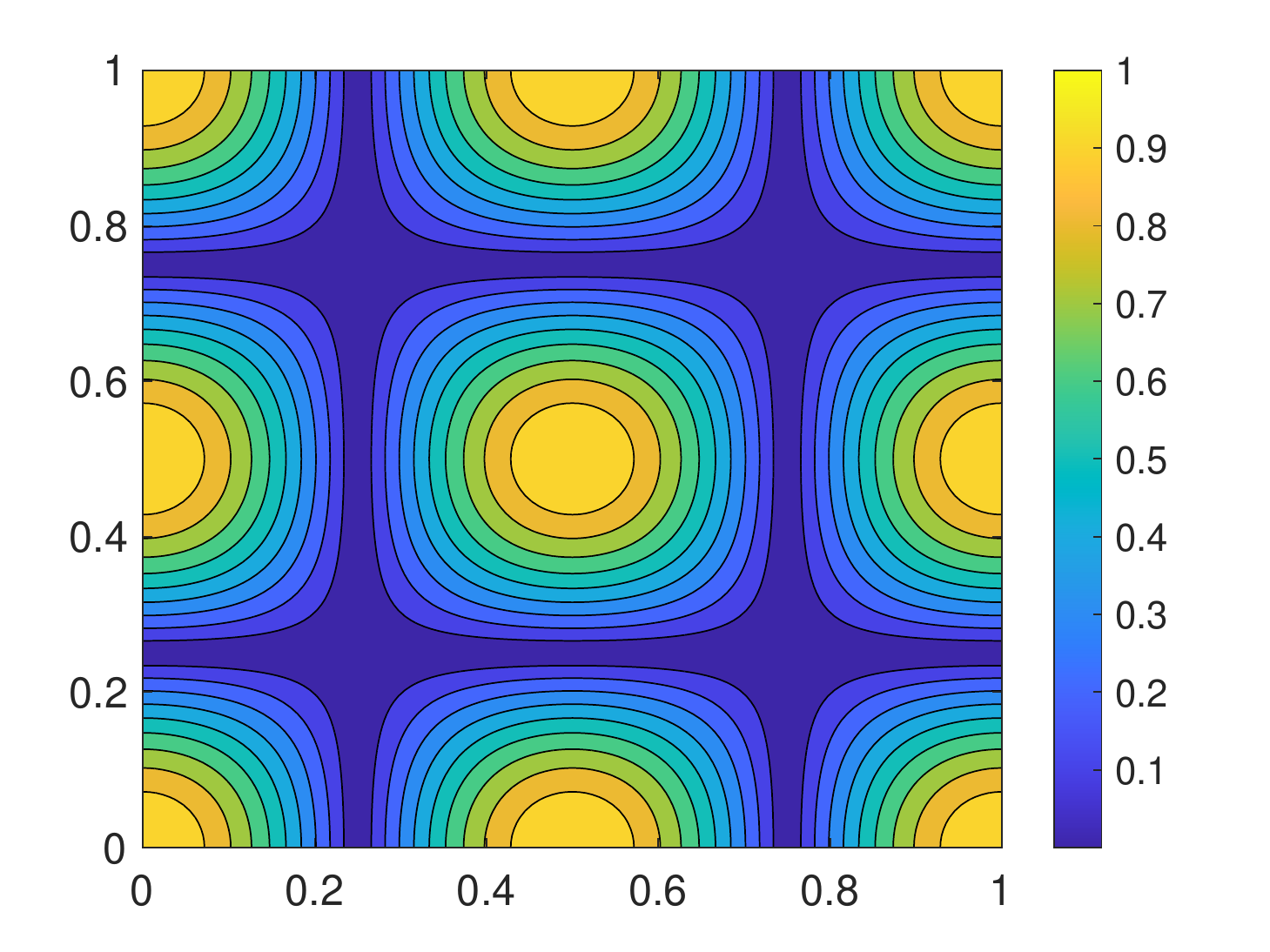}
\caption{Example 2. The first component (left) and the second component (right) of the initial velocity.}
\label{latticeinitial}
\end{figure}

The velocity errors are shown in \cref{latticeTH} and \cref{latticeBR2}. For $P_{2}/P_{1}$, modified CONV shows best performance; for $P_{2}^{bubble}/P_{1}^{disc}$, skew-linearized EMAC and nonlinear EMAC give best results. Compared to SKEW errors, the errors from modified CONV and EMAC do admit a slower growth speed. The scheme built on CONV blows up at nearly $t=2$ and $t=7$ for the two elements, respectively. Before it blows up, it seems to approximate the exact solution the most accurately.

The performance of modified CONV in lower Reynolds number cases is also of interest. One important problem is that whether the modified scheme will bring extra errors compared to CONV. We test the performance of CONV and modified CONV on several values of the viscosity (from $10^{3}$ to $10^{-3}$), with $P_{2}^{bubble}/P_{1}^{disc}$. The ratio of their errors, ``modified CONV/CONV", of each step are in
\cref{latticeBR2ratio}. One can find that there is almost no additional error for modified CONV. The difference is no more than $1\%$ in the worst case. For the case $\nu=10^{-3}$, modified CONV shows a visibly  better performance.
\begin{figure}[htbp]
\centering
\includegraphics[width=0.48\textwidth]{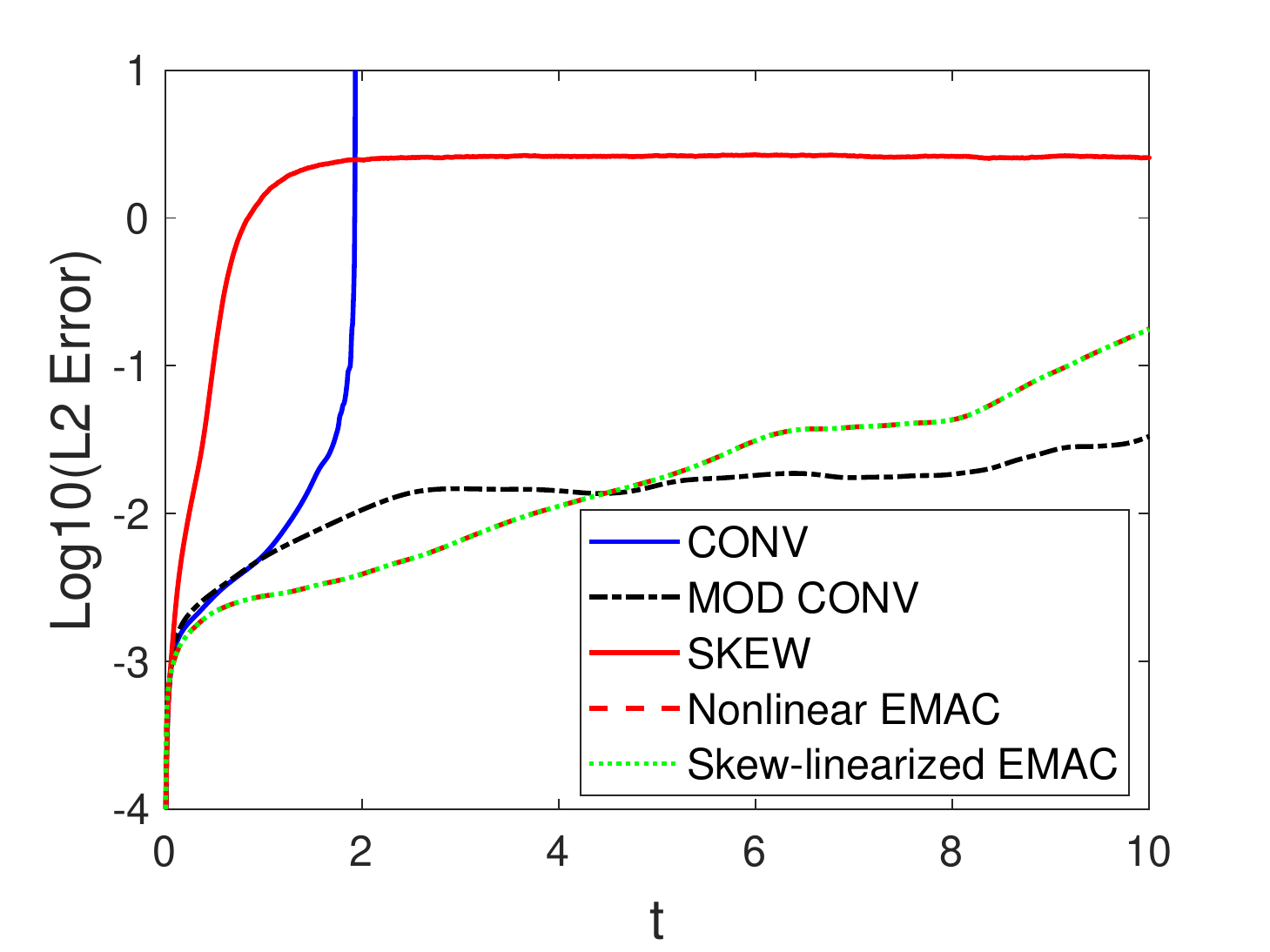}
\includegraphics[width=0.48\textwidth]{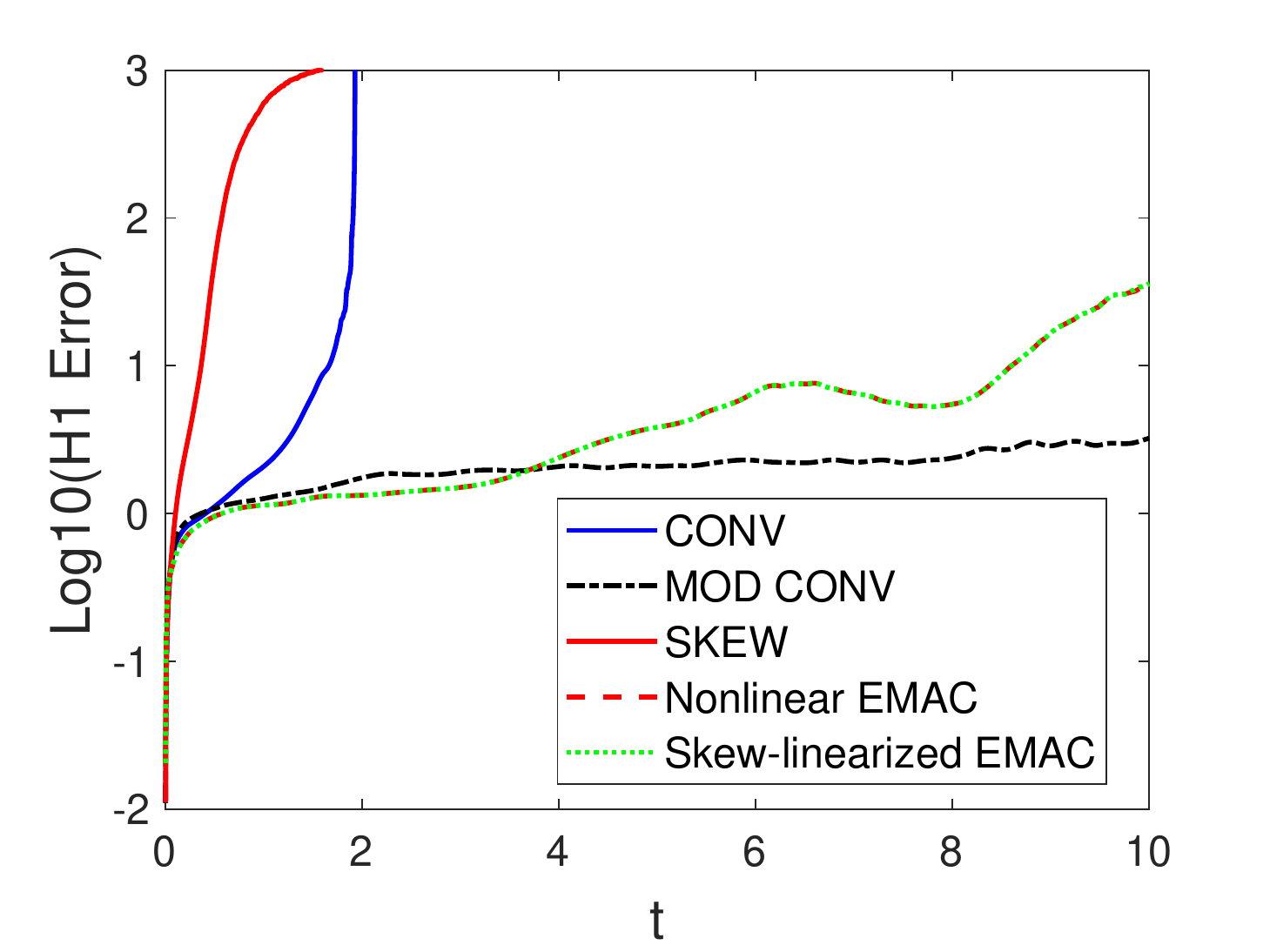}
\caption{Example 2. Velocity errors by $P_{2}/P_{1}$.}
\label{latticeTH}
\end{figure}
\begin{figure}[htbp]
\centering
\includegraphics[width=0.48\textwidth]{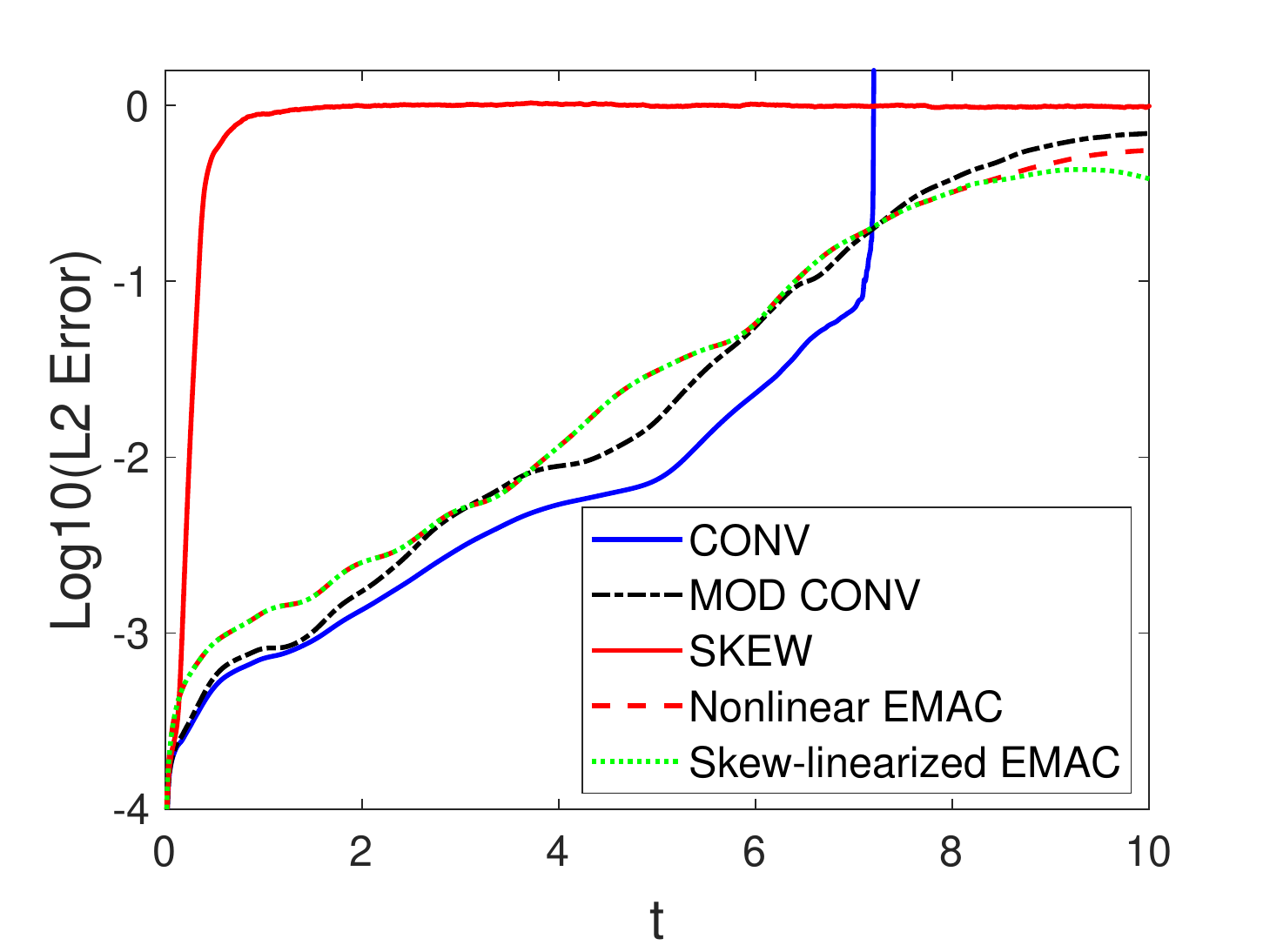}
\includegraphics[width=0.48\textwidth]{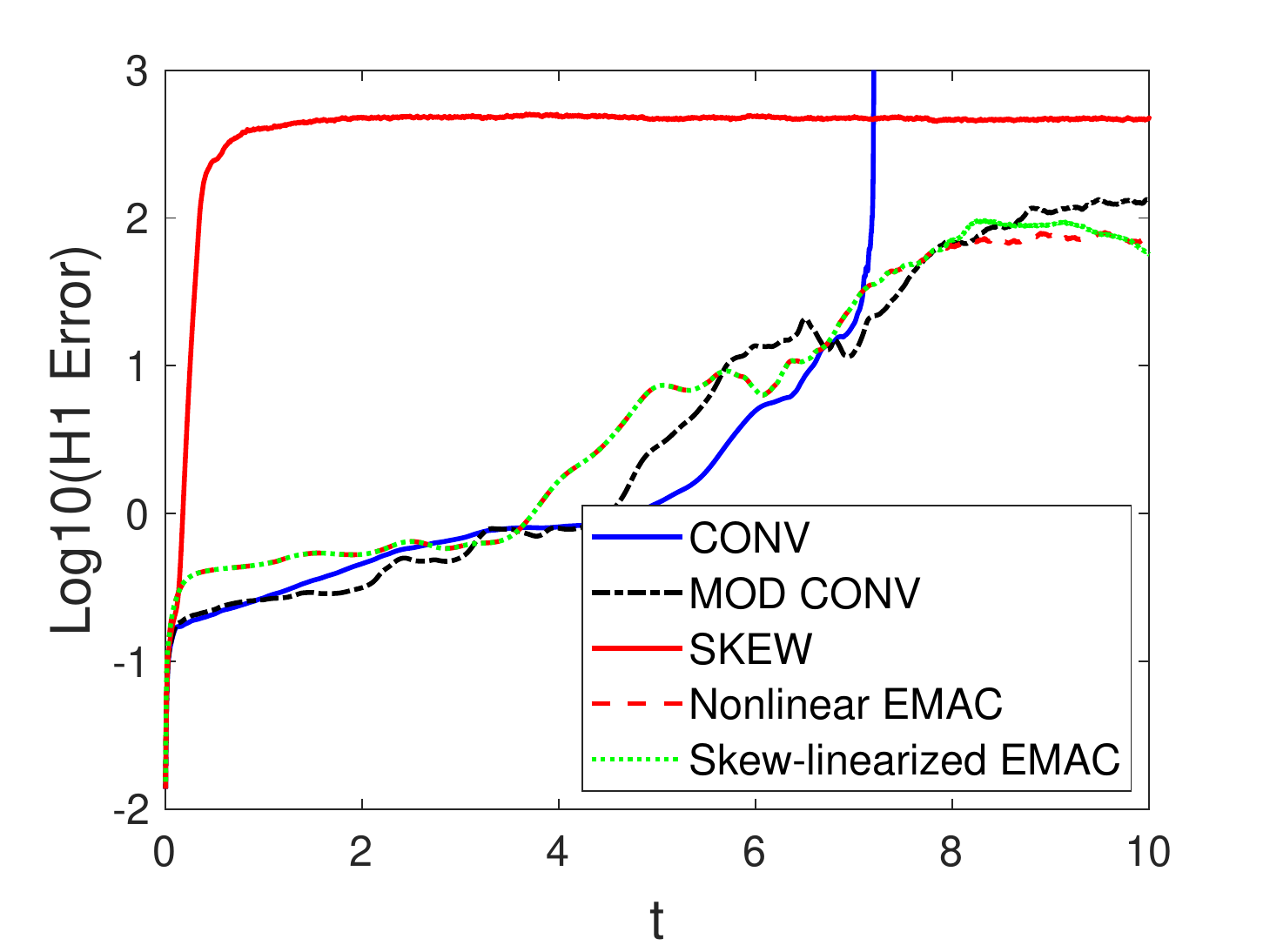}
\caption{Example 2. Velocity errors by $P_{2}^{bubble}/P_{1}^{disc}$.}
\label{latticeBR2}
\end{figure}
\begin{figure}[htbp]
\centering
\includegraphics[width=0.48\textwidth]{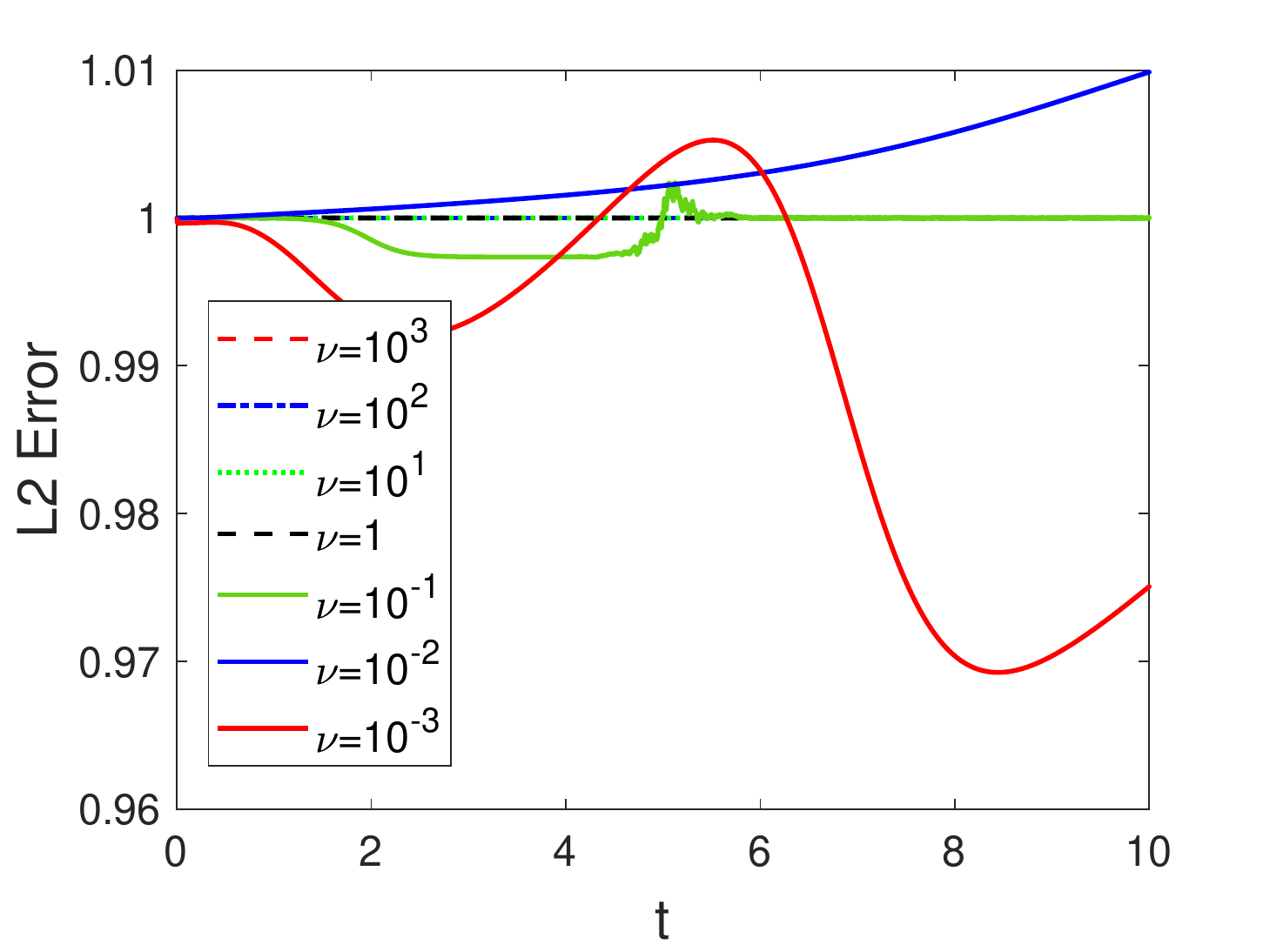}
\includegraphics[width=0.48\textwidth]{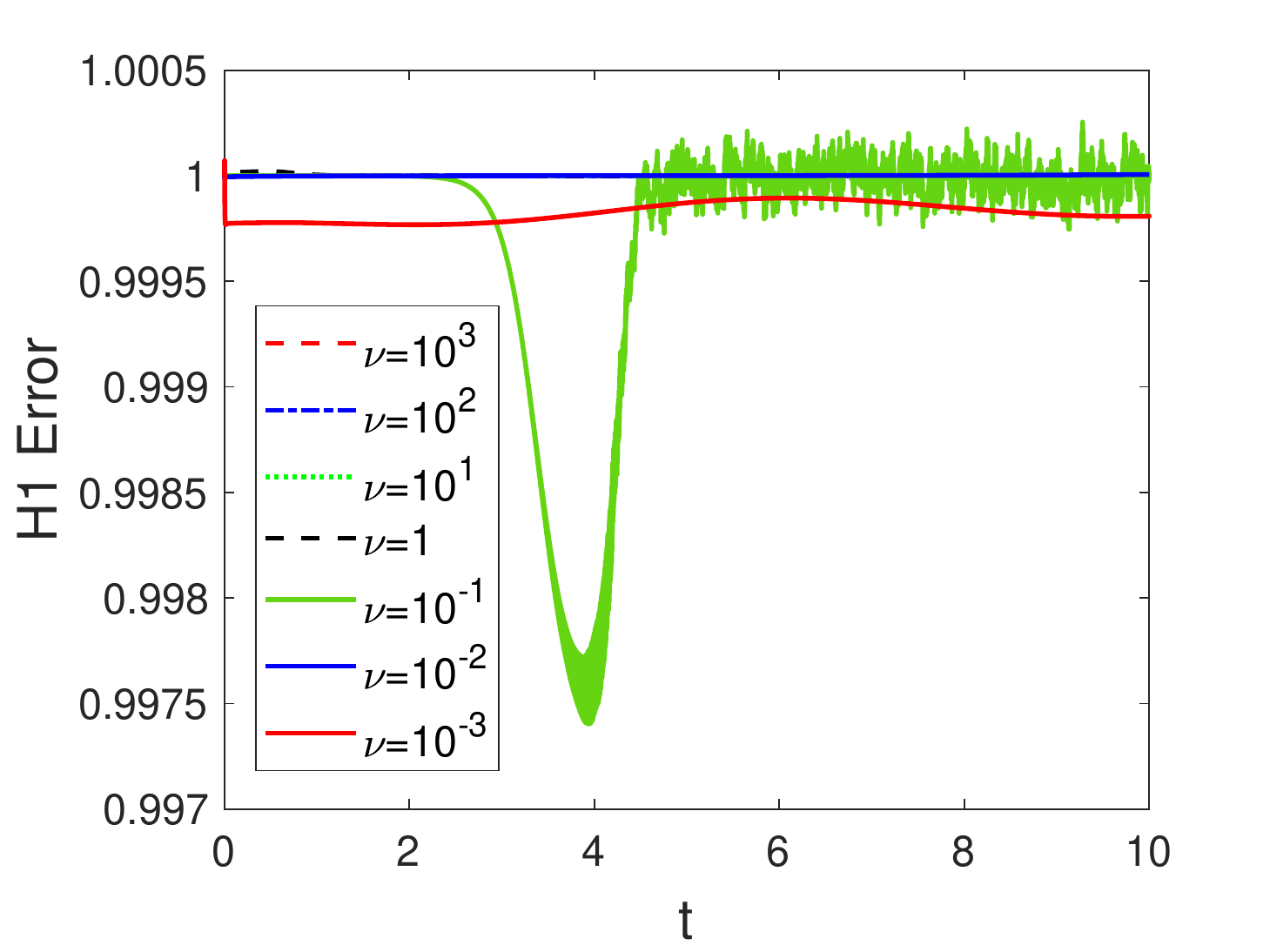}
\caption{Example 2. Ratios of the velocity errors between modified CONV and CONV by $P_{2}^{bubble}/P_{1}^{disc}$ element.}
\label{latticeBR2ratio}
\end{figure}
\subsection{Example 3: the forward-backward step problem}
Our final test is the forward-backward step problem \cite{John2006,layton2008,Rebholz2017}. The computational domain is a $40\times10$ rectangle with a $1\times 1$ step located at the bottom of the domain and $5$ units to the left bank. The inflow boundary condition $\boldsymbol{u}_{in}=(y(10-y)/25,0)$ is prescribed at $x=0$. At the outflow ($x=40$) we apply the ``do nothing" boundary condition. For other boundary, including the step, the no-slip condition is strongly enforced. The initial velocity is simply set as $\boldsymbol{u}^{0}=(y(10-y)/25,0)$ and the external force is taken to be zero all the time.

The boundary condition and the coarsest mesh (``Mesh 1") are shown in \cref{stepmesh1}. We use the $P_{2}^{bubble}/P_{1}^{disc}$ pair which provides 22738 degrees of freedom on Mesh 1. This partition can not resolve all the scale for this problem, which is common in practice. In the papers \cite{layton2008,Rebholz2017}, some cases of $\nu=1/600$ were studied. We tried this viscosity and found that there was almost no difference between CONV and modified CONV. Thus, we test a higher Reynolds number in this paper by setting $\nu=0.001$. The BDF2 time stepping method is used till an end time $T=80$ with $\Delta t=0.01$. We study the performance of CONV, modified CONV and SKEW. In addition, the second order Scott-Vogelius (SV2) element \cite{Arnold1992,john_divergence_2017} with CONV is also studied on the barycenteric refinement of Mesh 1, which is a type of exactly divergence-free element. Thus, for SV2 all formulations are equivalent.

In this problem, the eddies will form behind the step and then detach the step and flow downstream. To give a ``reference solution", we compute the schemes built on $P_{2}^{bubble}/P_{1}^{disc}$ with SKEW and grad-div stabilization (the stabilization parameter is taken to be $0.1$) on Mesh 2 (a regular refinement of Mesh 1), and SV2 with CONV on the barycentric refinement of Mesh 2. They give very similar results, which are sufficient to be as a reference for the methods on Mesh 1. The corresponding results are shown in \cref{stepref60} and \cref{stepref80}. Contours of speeds and plots of streamlines of the solutions on Mesh 1 could be found in \cref{step60} and \cref{step80} for $t=60$ and $t=80$, respectively.

At $t=60$, one could see that the solution by modified CONV is the closest to the reference solution. At $t=80$, all the methods fail to capture the furthest eddy near $x=28$ except modified CONV, although there is a slight deviation in the position.
\begin{figure}[htbp]
\centering
\includegraphics[width=1.0\textwidth]{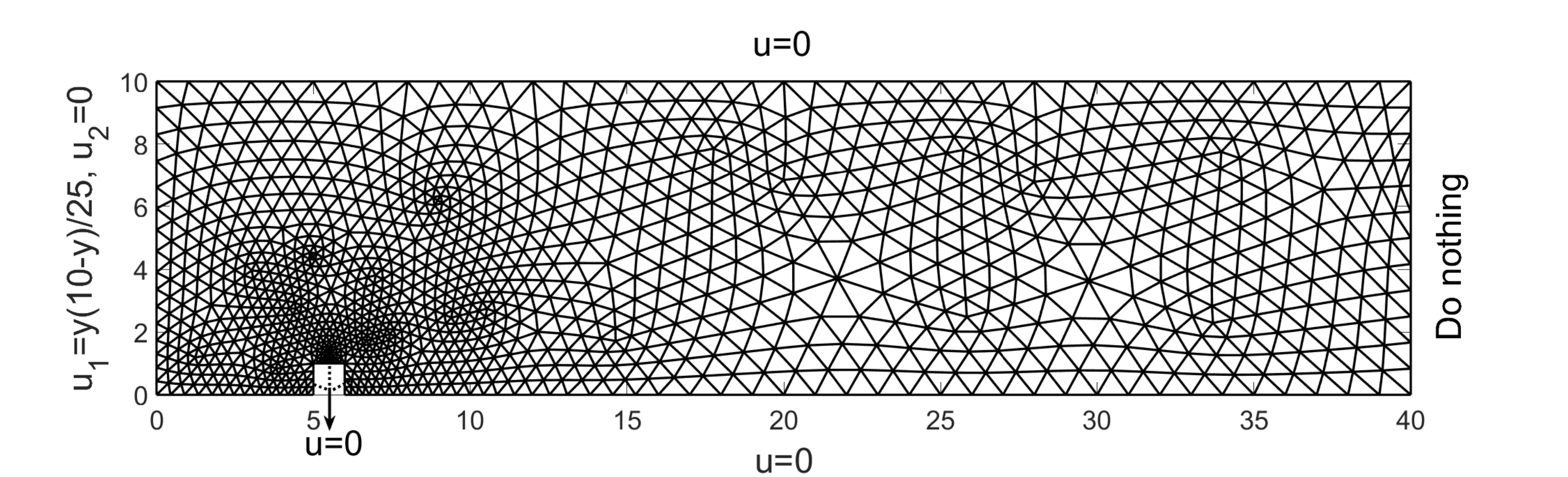}
\caption{Example 3. The boundary condition and the coarsest mesh ``Mesh 1".}
\label{stepmesh1}
\end{figure}
\begin{figure}[htbp]
\centering
\includegraphics[width=0.48\textwidth,height=4cm]{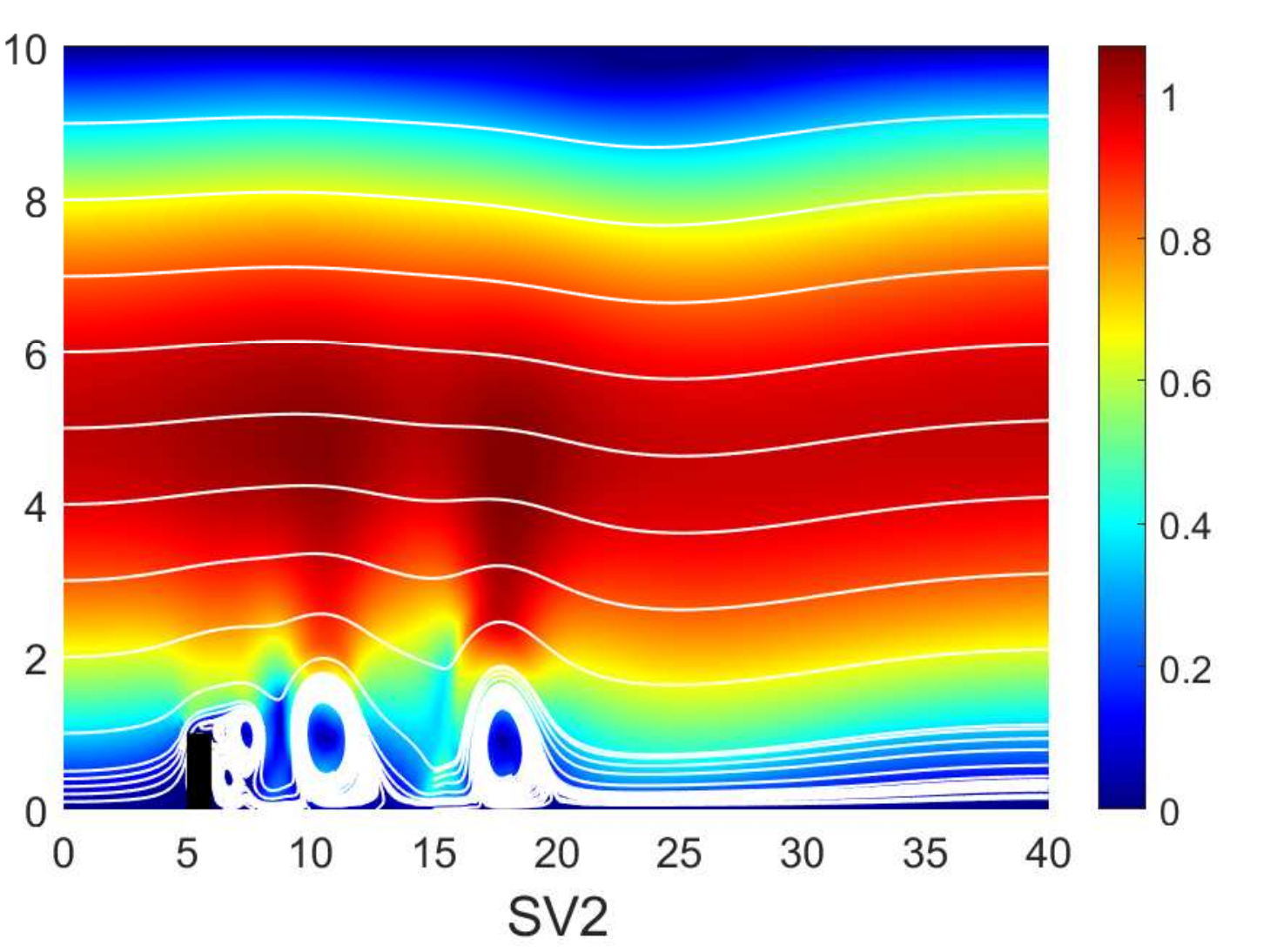}
\includegraphics[width=0.48\textwidth,height=4cm]{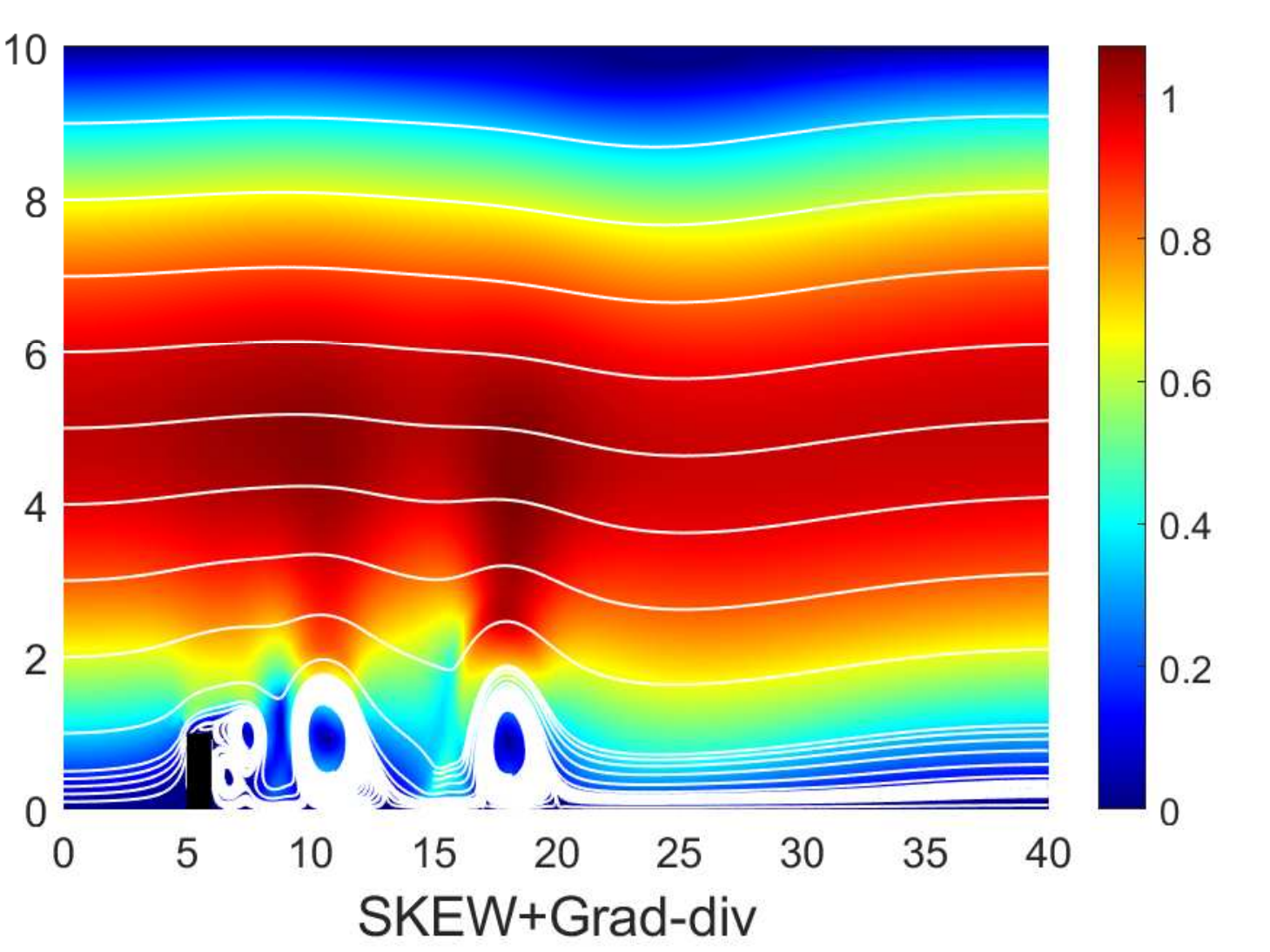}
\caption{Example 3. $\nu=0.001$. Contours of speeds and plots of streamlines of the ``truth" solutions at $t=60$. Left: SV2 on barycenteric refinement of Mesh 2; right: $P_{2}^{bubble}/P_{1}^{disc}$ on Mesh 2.}
\label{stepref60}
\end{figure}
\begin{figure}[htbp]
\centering
\includegraphics[width=0.48\textwidth,height=4cm]{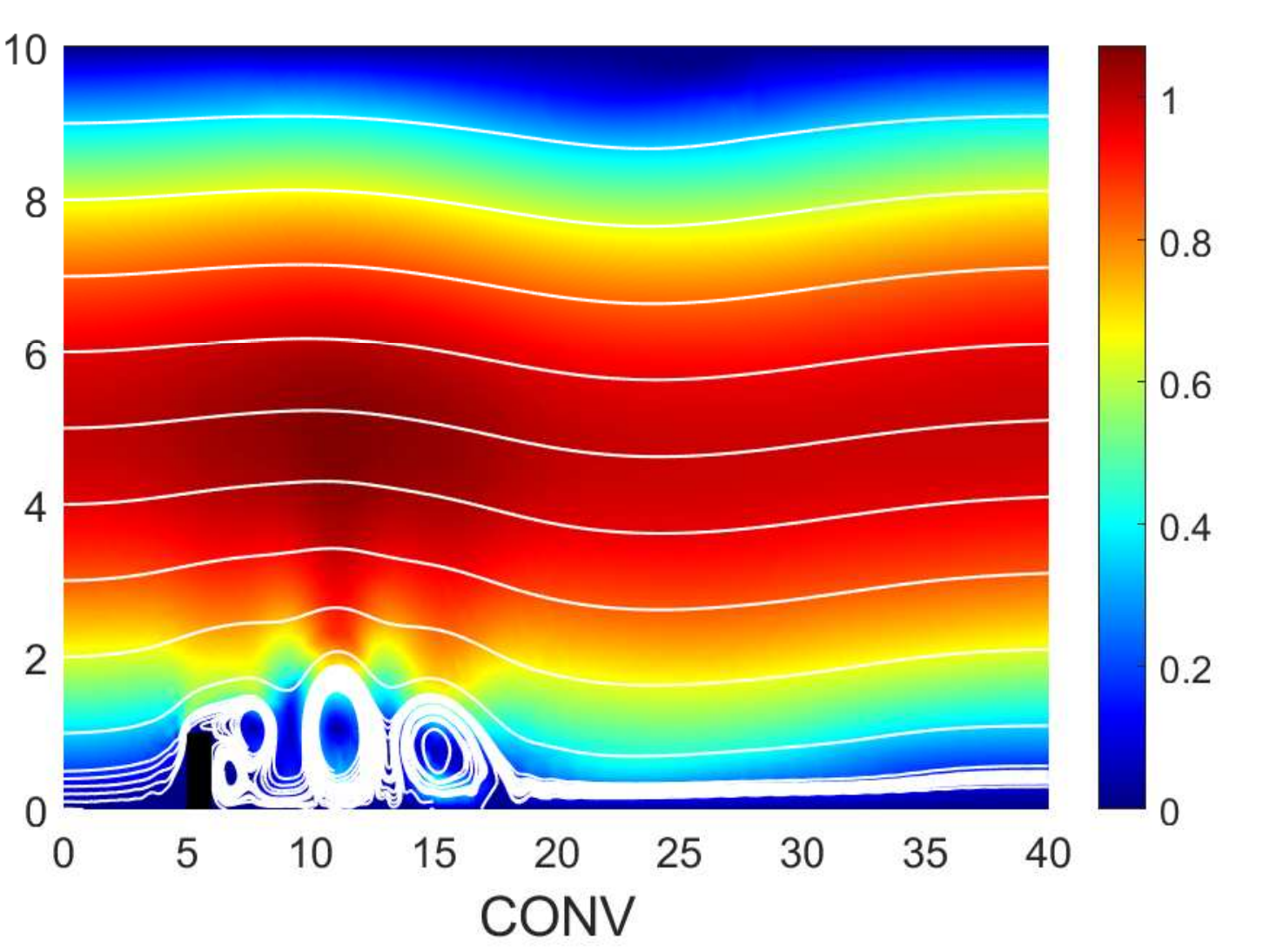}
\includegraphics[width=0.48\textwidth,height=4cm]{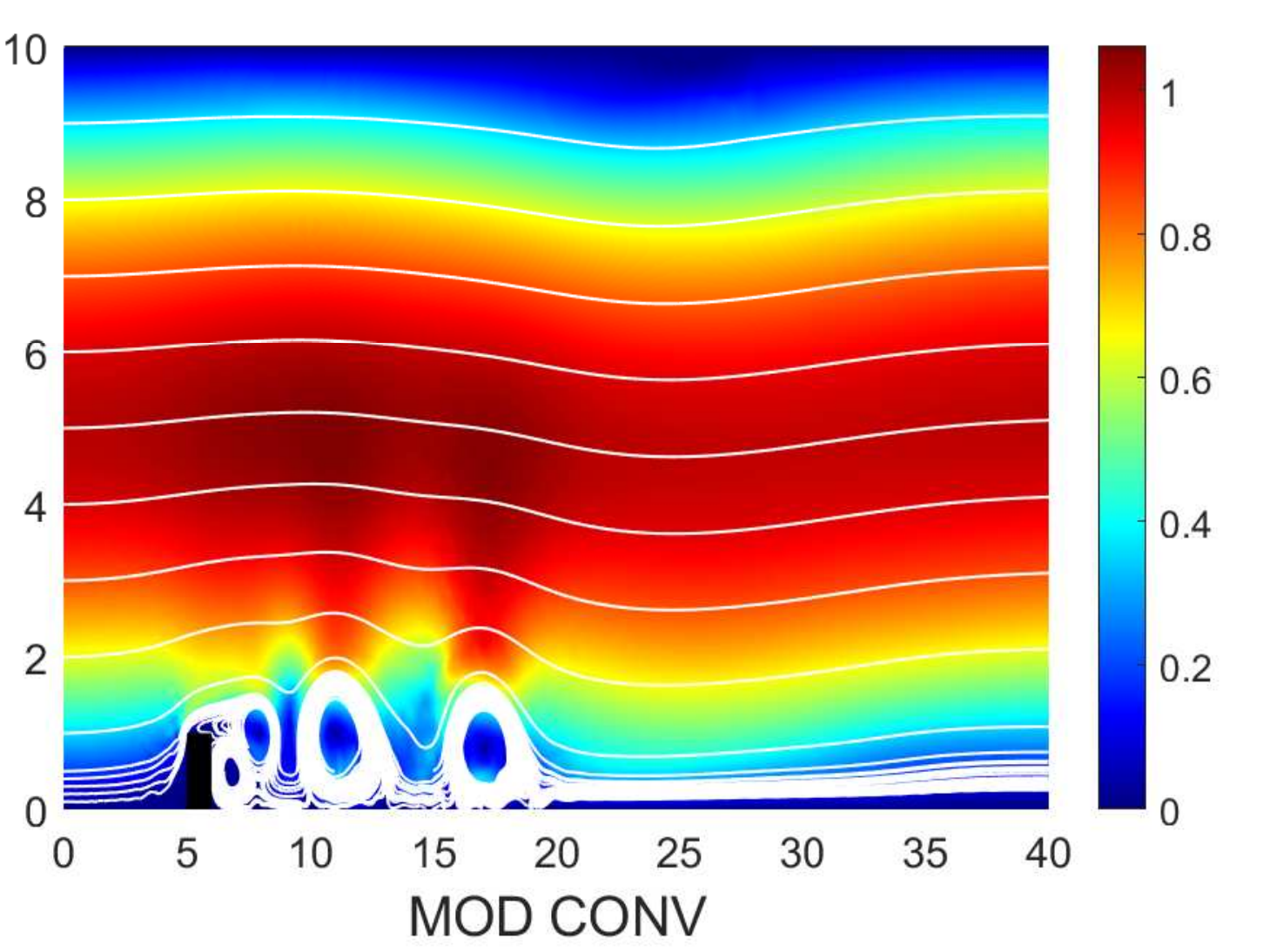}
\includegraphics[width=0.48\textwidth,height=4cm]{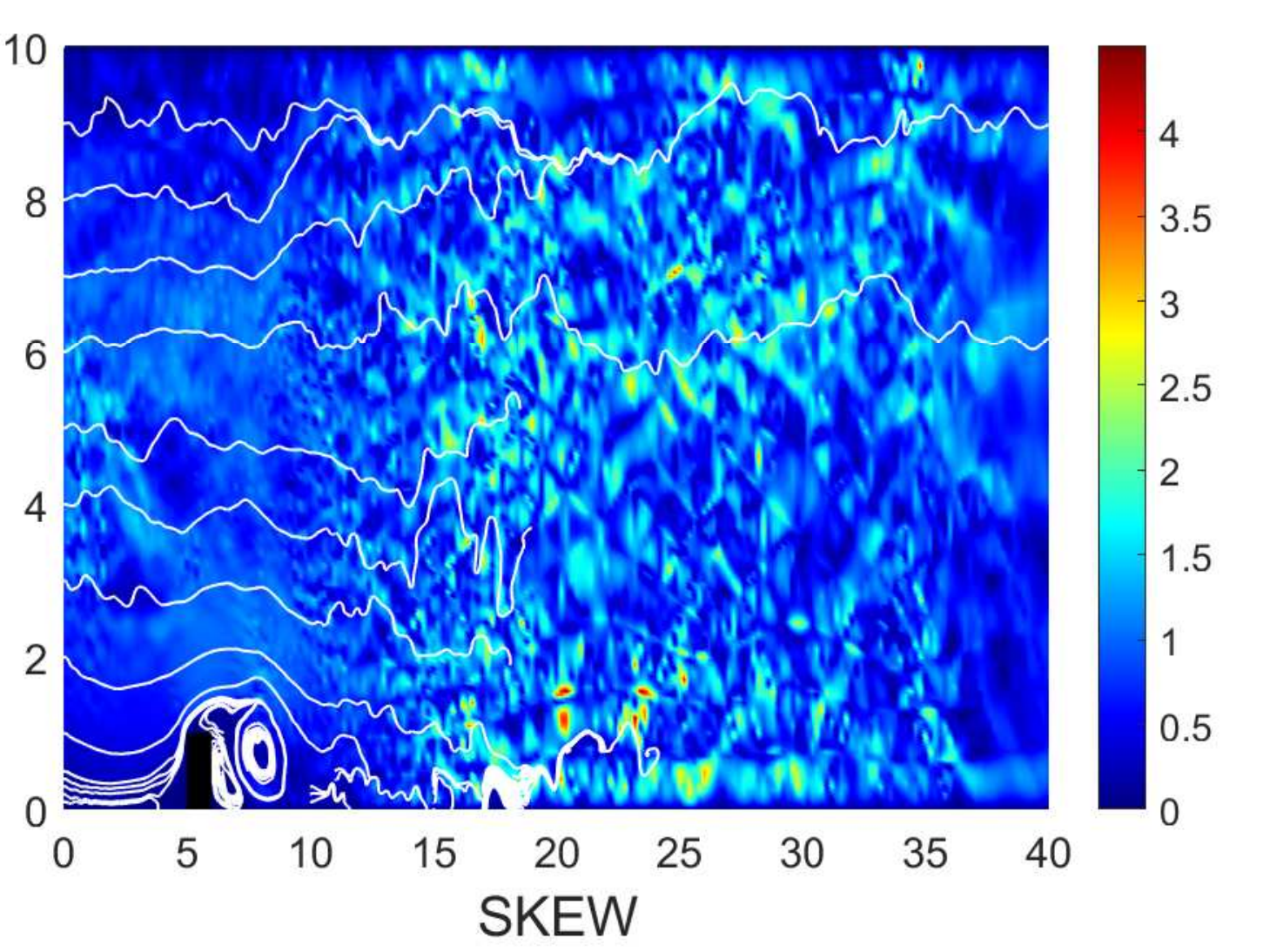}
\includegraphics[width=0.48\textwidth,height=4cm]{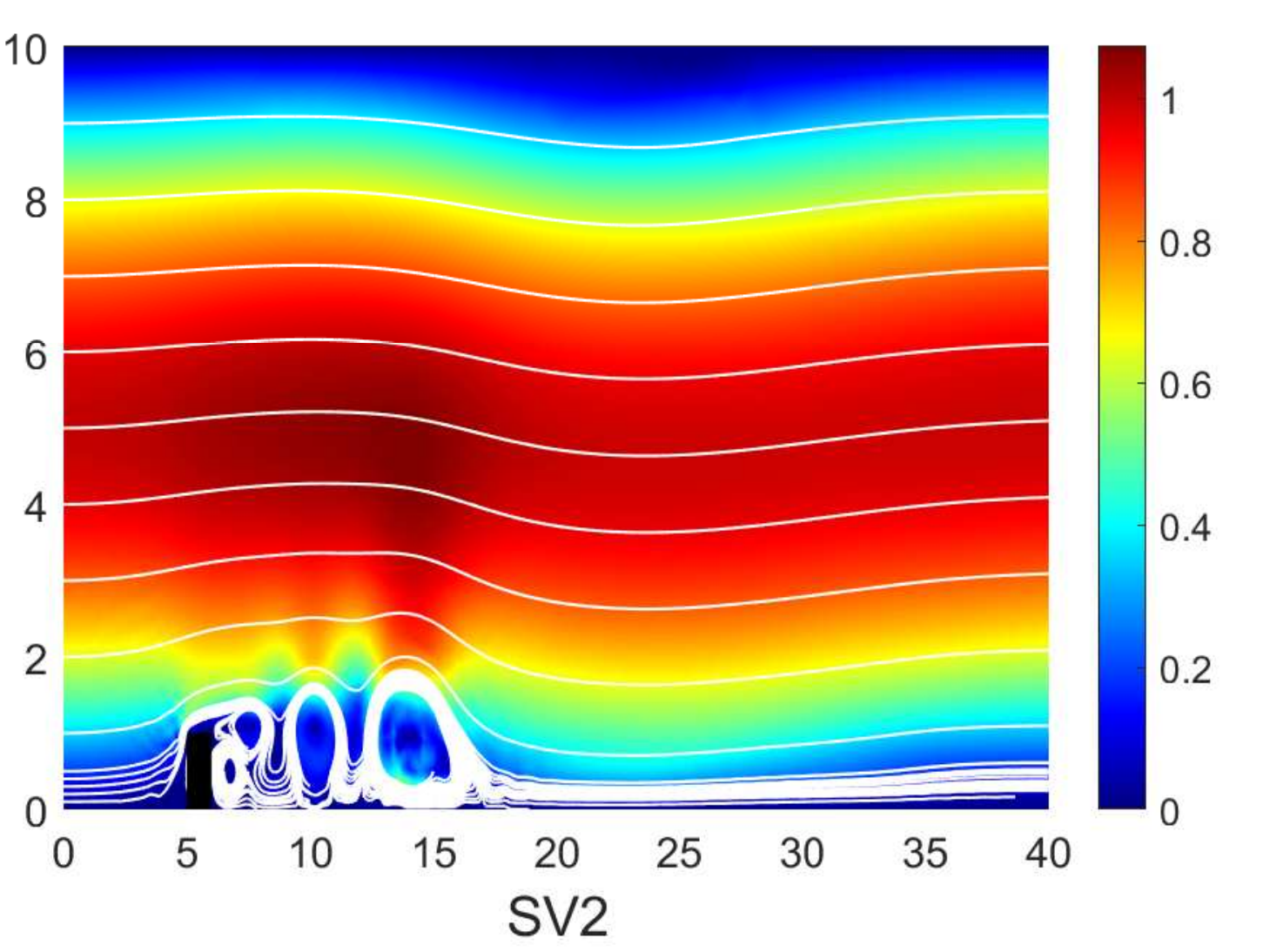}
\caption{Example 3. $\nu=0.001$. Contours of speeds and plots of streamlines of the solutions on Mesh 1 by $P_{2}^{bubble}/P_{1}^{disc}$ at $t=60$.}
\label{step60}
\end{figure}
\begin{figure}[htbp]
\centering
\includegraphics[width=0.48\textwidth,height=4cm]{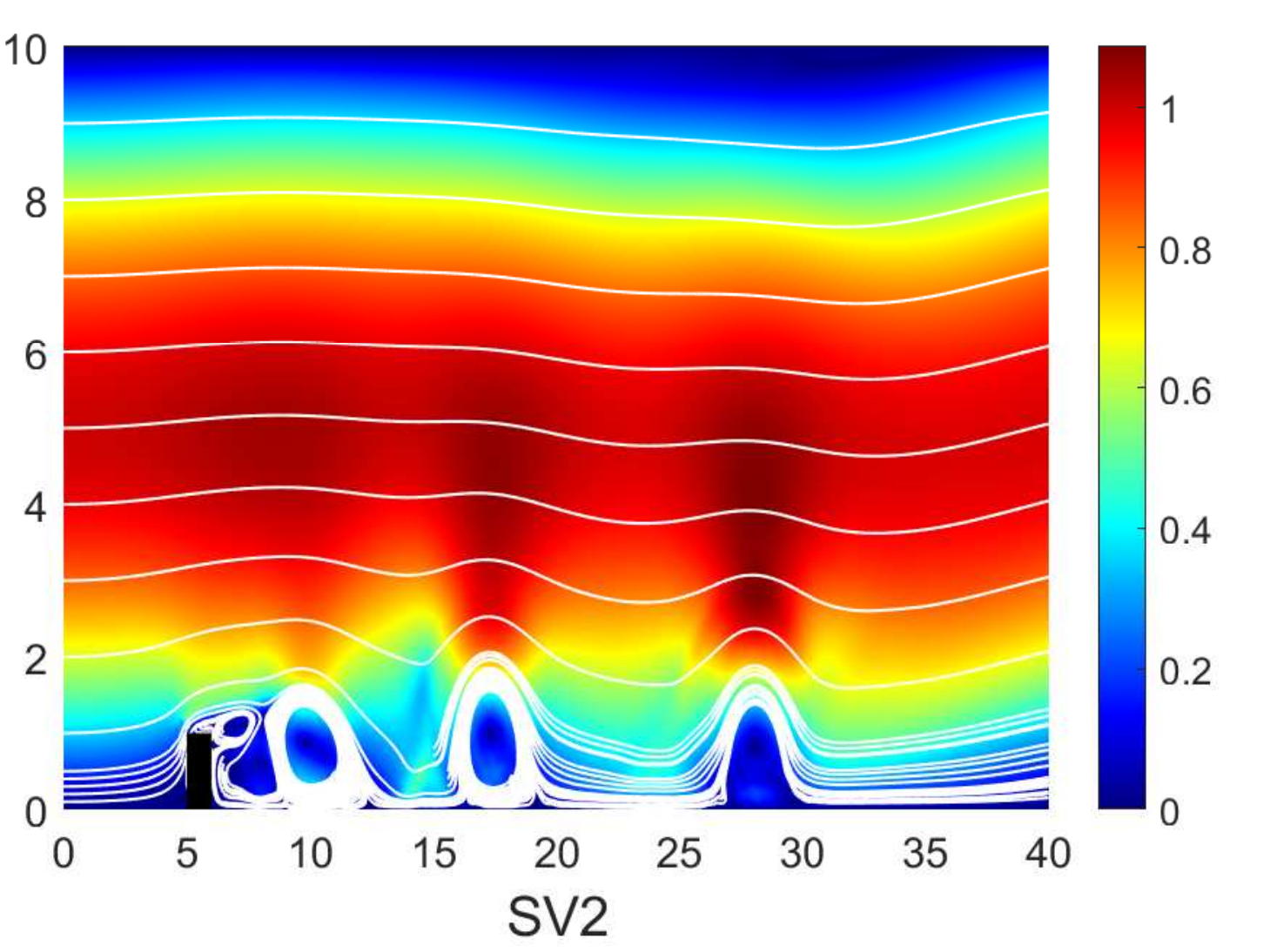}
\includegraphics[width=0.48\textwidth,height=4cm]{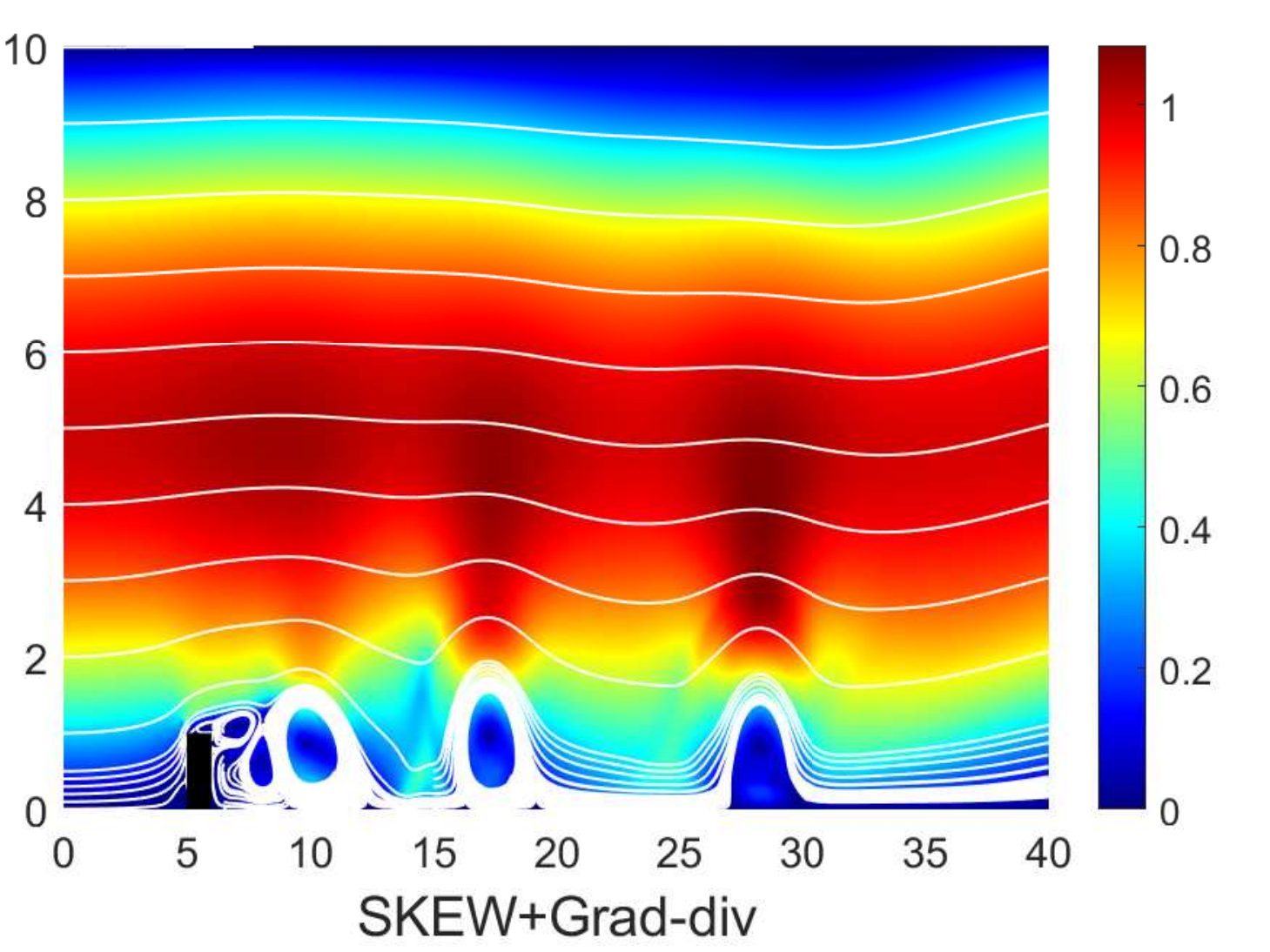}
\caption{Example 3. $\nu=0.001$. Contours of speeds and plots of streamlines of the ``truth" solutions at $t=80$. Left: SV2 on barycenteric refinement of Mesh 2; right: $P_{2}^{bubble}/P_{1}^{disc}$ on Mesh 2.}
\label{stepref80}
\end{figure}
\begin{figure}[htbp]
\centering
\includegraphics[width=0.48\textwidth,height=4cm]{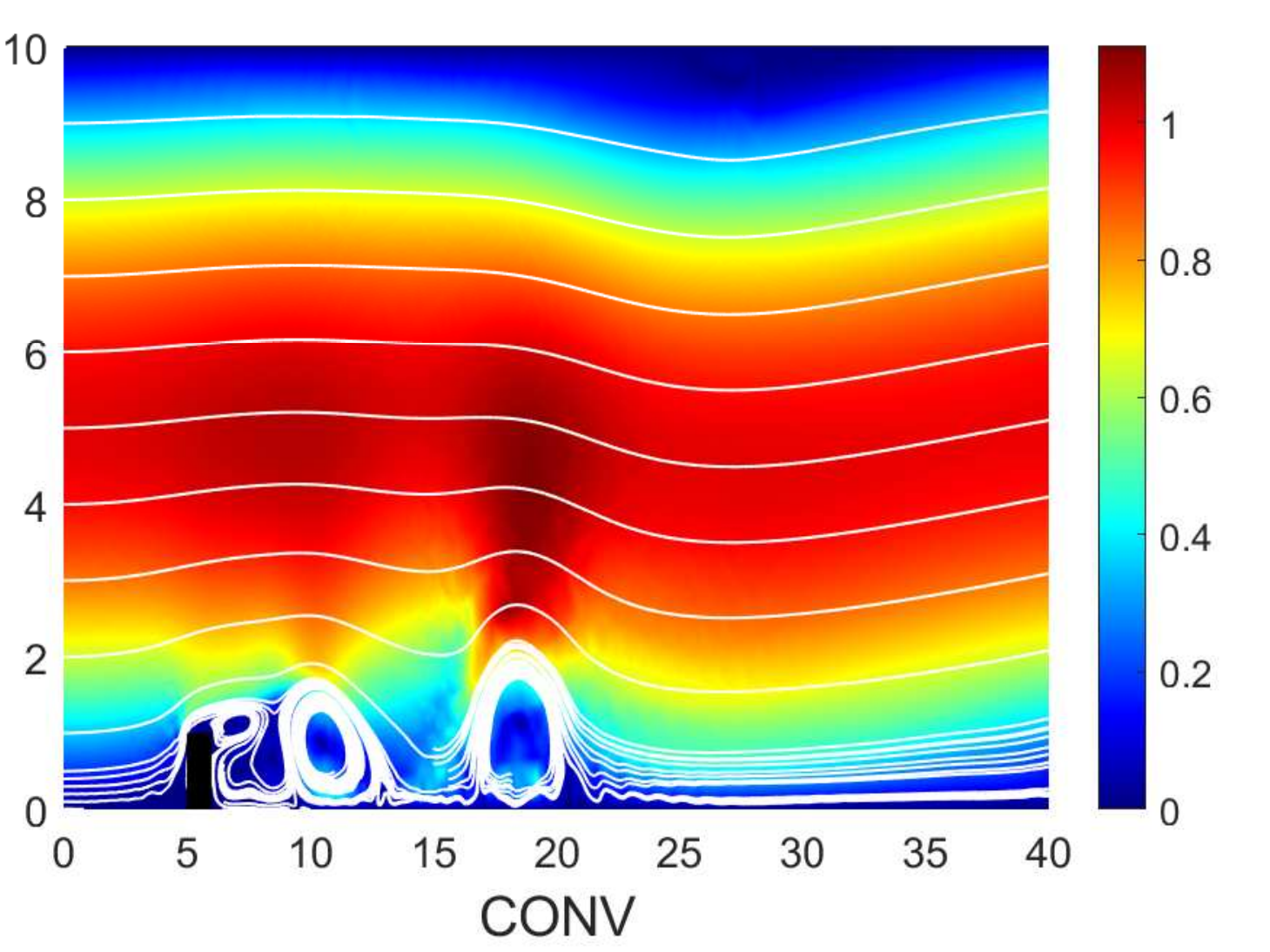}
\includegraphics[width=0.48\textwidth,height=4cm]{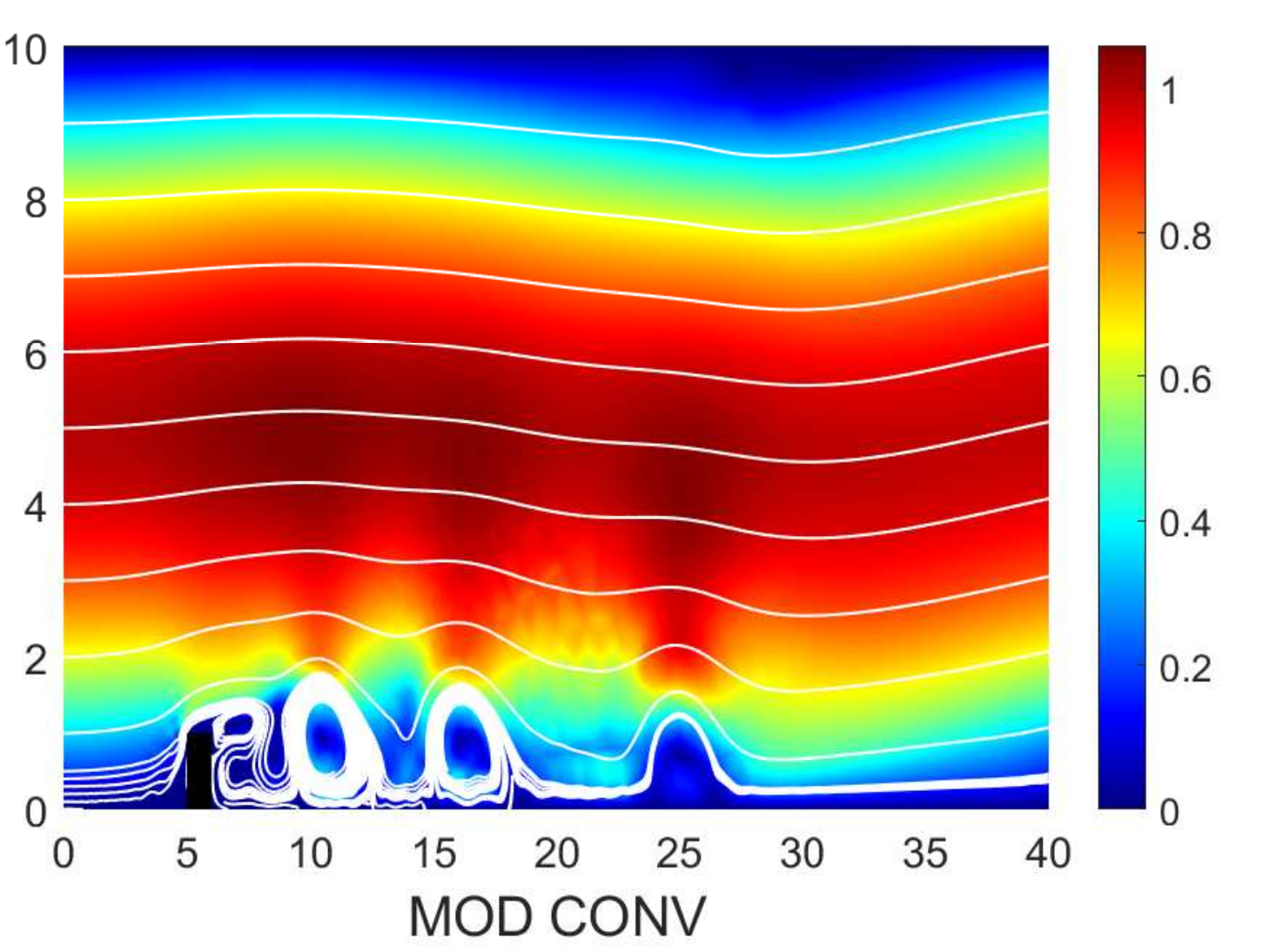}
\includegraphics[width=0.48\textwidth,height=4cm]{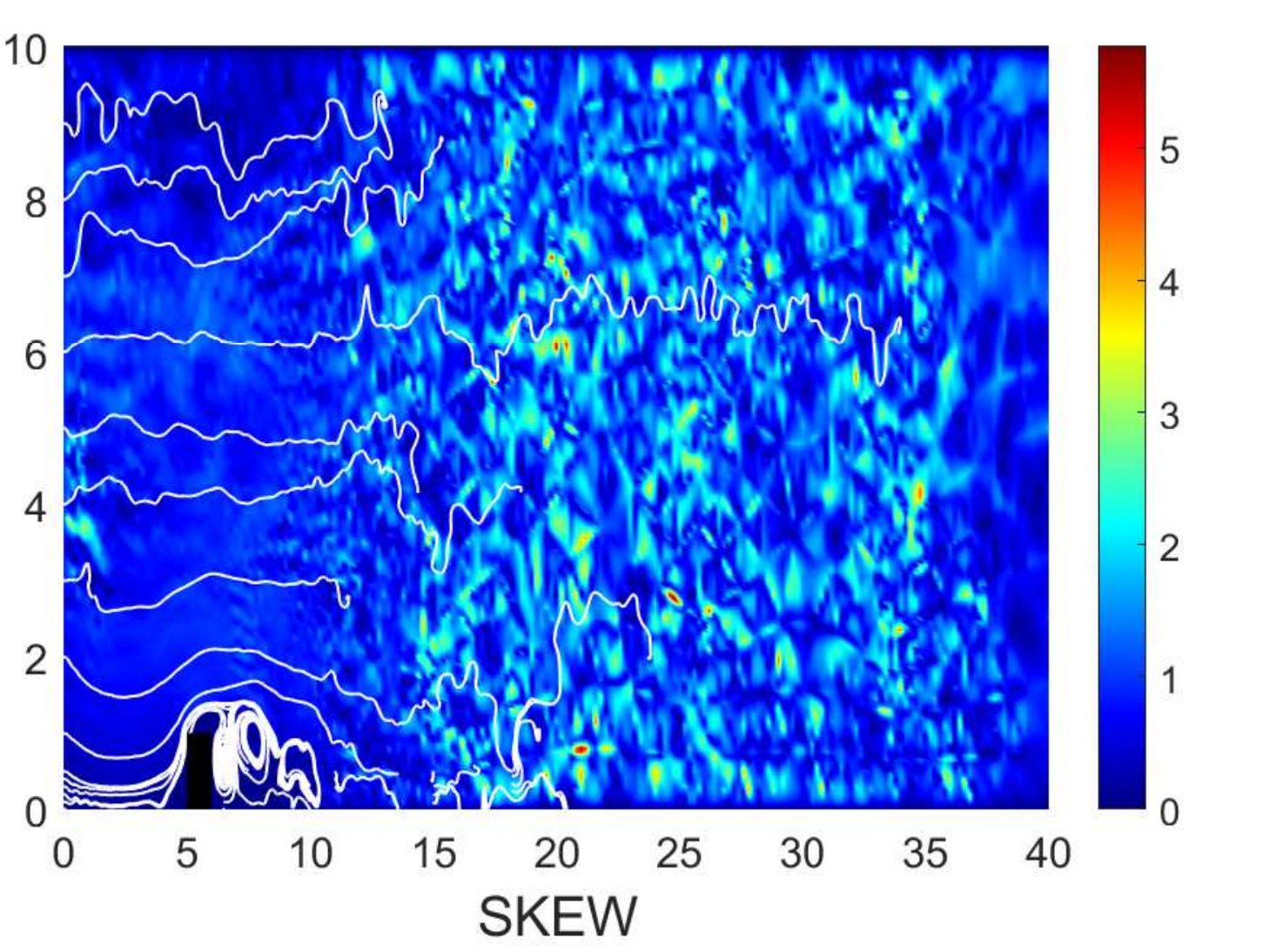}
\includegraphics[width=0.48\textwidth,height=4cm]{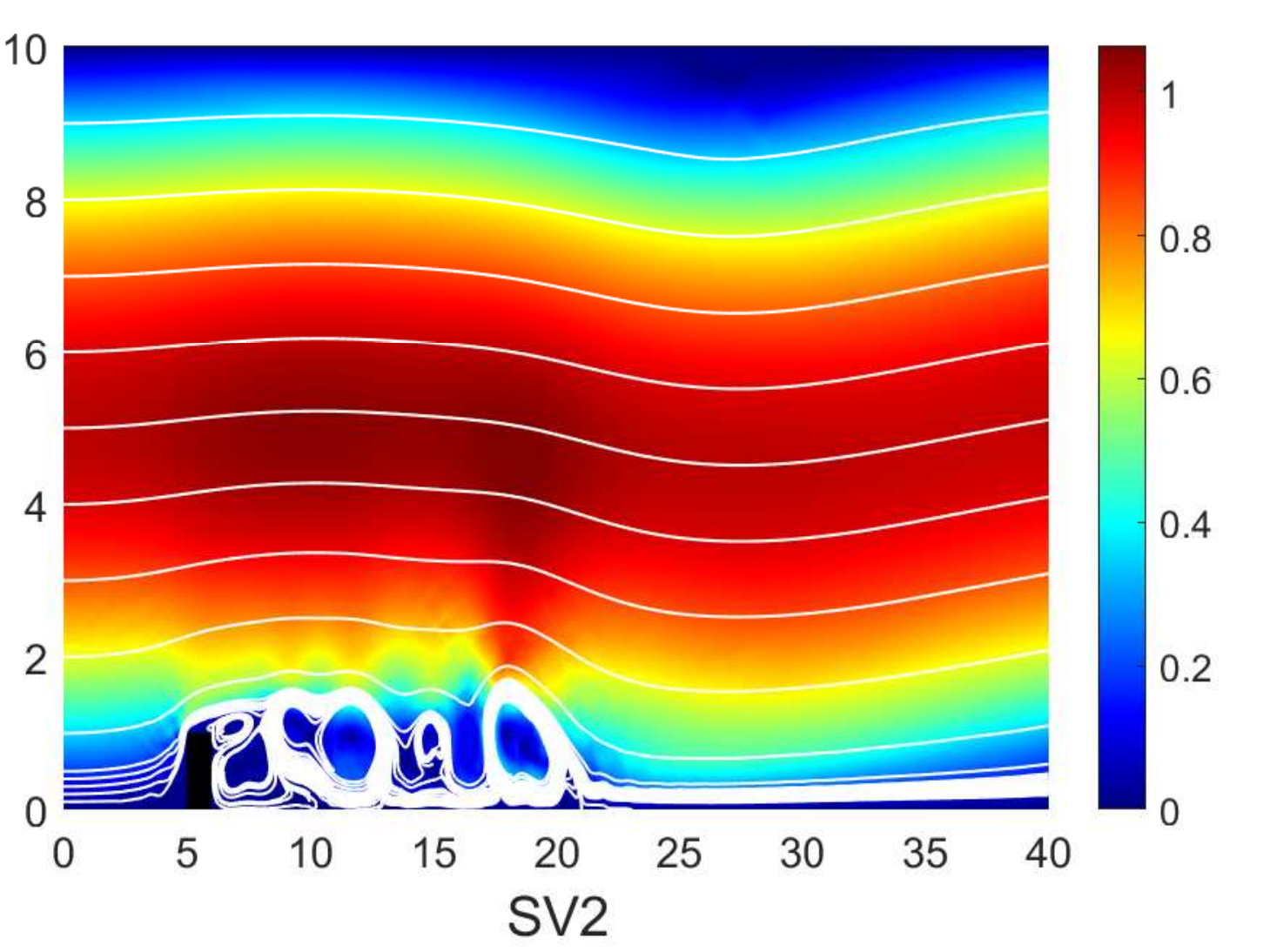}
\caption{Example 3. $\nu=0.001$. Contours of speeds and plots of streamlines of the solutions on Mesh 1 by $P_{2}^{bubble}/P_{1}^{disc}$ at $t=80$.}
\label{step80}
\end{figure}

\label{sec:5}
\bibliographystyle{amsplain}
\bibliography{references}
\end{document}